\documentclass{amsart}
\usepackage{amsmath,amssymb,amsthm,xspace,a4wide,amscd,latexsym,amscd}

\usepackage{epsfig}
\usepackage[]{graphicx}

\newcommand{\C}{\ensuremath{\mathbb{C}}}
\newcommand{\R}{\ensuremath{\mathbb{R}}}

\newcommand{\defined}{\ensuremath{\smash[t]{\overset{\scriptscriptstyle{def}}{=}}}}
\newcommand{\bd}{\ensuremath{\partial}}
\newcommand{\disc}{\ensuremath{\mathbb{D}}}
\newcommand{\D}{\ensuremath{\mathbb{D}}}
\newcommand{\eps}{\ensuremath{\varepsilon}}

\renewcommand{\H}{\ensuremath{\mathcal{H}}}

\newcommand{\GN}{\ensuremath{\mathcal{G}_0}}

\newcommand{\G}{\ensuremath{\mathcal{G}}}

\renewcommand{\c}{\ensuremath{\mathcal{C}}}
\newcommand{\F}{\ensuremath{\mathcal{F}}}

\theoremstyle{plain}

\newtheorem{theorem}{Theorem}
\newtheorem{lemma}{Lemma}
\newtheorem{definition}{Definition}
\newtheorem{proposition}{Proposition}
\newtheorem{corollary}{Corollary}

\newtheorem{question}{Question}

\newtheorem*{thmdgo}{Theorem \bf{DGO}}

\title[Envelopes of holomorphy]{Envelopes of holomorphy and holomorphic discs\\}

\author[B. J\"oricke]{Burglind J\"oricke}
\address{Max-Planck-Institut f\"ur Mathematik\\ P.O.Box: 7280\\ 53072 Bonn\\ Germany}
\email{joericke@mpim-bonn.mpg.de} \subjclass[2000]{32A40;
32E35;53D10} \keywords{envelopes of holomorphy, continuity
principle, holomorphic discs, Riemann surfaces, Stein filling,
planar trees}
\begin{document}

\begin{abstract} The envelope of holomorphy
of an arbitrary domain in a two-dimensional Stein manifold is
identified with a connected component of the set of equivalence
classes of analytic discs immersed into the Stein manifold with
boundary in the domain. This implies, in particular, that for each
of its points the envelope of holomorphy contains an embedded
(non-singular) Riemann surface (and also an immersed analytic disc)
passing through this point with boundary contained in the natural
embedding of the original domain into its envelope of holomorphy.
Moreover, it says, that analytic continuation to a neighbourhood of
an arbitrary point of the envelope of holomorphy can be performed by
applying the continuity principle once. Another corollary concerns
representation of certain elements of the fundamental group of the
domain by boundaries of analytic discs. A particular case is the
following. Given a contact three-manifold with Stein filling, any
element of the fundamental group of the contact manifold whose
representatives are contractible in the filling can be represented
by the boundary of an immersed analytic disc.

\end{abstract}

\maketitle

\vspace{1cm}

\centerline  {0. Introduction}

\vspace{1cm}

The notion of the envelope of holomorphy of domains in $\C ^n$ (or,
more generally, in Stein manifolds) is as classical as the notion of
pseudoconvex domains. Nevertheless, basic questions about envelopes
of holomorphy are open. For instance, not much is known in general
about the number of sheets of the envelope of holomorphy. It is not
clear in general when the envelope of holomorphy is single-sheeted
or at least (say smoothly) equivalent to a domain in the same Stein
manifold (see e.g. \cite{Sto}) .

One of the most interesting problems in this respect is to
understand invariants of the envelope of holomorphy in terms of
invariants of the original domain. It is known that the first Betti
number of the envelope of holomorphy does not exceed that of the
original domain \cite{Ke}. Moreover, it is proved in \cite{Ke} that
the natural homomorphism between fundamental groups is surjective.
In the same vein the natural map between first Cech cohomologies is
injective \cite{Roy}, but in the general situation not too much is
known beyond these results. Naive hopes are not justified (see e.g.
the paper \cite{FoZa}.)

The problem of understanding invariants of the envelope of
holomorphy in terms of invariants of the domain is even interesting
in the following particular case. The domain is a suitable one-sided
neighbourhood of the boundary of a strictly pseudoconvex domain a
Stein manifold (for instance, it equals the set $\{-\varepsilon <
\rho < 0 \}$ for a strictly plurisubharmonic defining function and a
small positive constant $\varepsilon$) and the envelope of
holomorphy is the domain itself. This case reduces to understanding
the topology of the Stein fillings of a contact manifold in terms of
the topology of the contact manifold and is well-known to symplectic
geometers. Despite recent progress and breakthroughs many problems
remain open. For instance, there are examples of contact
three-manifolds that have a Stein filling with second Betti number
strictly exceeding that of the three-manifold and an estimate of the
second Betti number of Stein fillings of a given contact
three-manifold is not known in general. For a contemporary account
see \cite{OSt}.

The general problem motivates the search for a geometric description
of the envelope of holomorphy. It is well-known that any domain in a
Stein manifold has an envelope of holomorphy. Several constructions
are known (see e.g. \cite{H}, \cite{Rossi}, \cite{Ma} ). It is not
obvious how to obtain from these constructions geometric information
about the envelope of holomorphy.

We give here a new description of the envelope of holomorphy of a
domain in a Stein manifold in terms of equivalence classes of
analytic discs. This description, in particular, implies that
analytic continuation to a neighbourhood of each point in the
envelope of holomorphy can be performed by applying the continuity
principle once along a family of immersed analytic discs (see below
for details).

The approach has further geometric consequences which were not known
before. To mention only one of them, for each of its points the
envelope of holomorphy contains an embedded (non-singular) Riemann
surface (and also an immersed analytic disc) passing through this
point with boundary contained in the natural embedding of the
original domain into its envelope of holomorphy. This is in contrast
to what is known for polynomial hulls.

In the paper we focus on the case of Stein manifolds of dimension
$2$, which is in several aspects the most interesting case. We
believe that the main results are true in higher dimensions. It
seems they are true even in more general situations and we intend to
work this out later.

\section{Statement of results}

\noindent Denote by $X^2$ a Stein surface, i.e. a two-dimensional
Stein manifold. Let $G\subset X^2$ be a domain. For the description
of the envelope of holomorphy we use analytic discs immersed into
$X^2$ with boundary in $G$. More precisely, we need the following
definition.

\begin{definition}\label{Definition 1} Consider a holomorphic
immersion from a neigbourhood of the closed unit disc $\overline{\D}
\subset \C$ into $X^2$. The restriction $d:\overline{\disc}
\looparrowright X^2$ is an analytic disc.

If the boundary $d(\bd {\disc} )$ of the disc is contained in $G$ we
will call the disc a $G$-disc. The set of $G$-discs is denoted by
$\mathcal{G}$.
\end{definition}

Fix a metric on $X^2$. For this we fix a proper holomorphic
embedding $\mathfrak{F}: X^2: \rightarrow \C^n$ into Euclidean space
$C^n$ of suitable dimension $n$ and pull back the metric induced on
$\mathfrak{F} X^2$ by $\C^n$. (By \cite{EGr} one can always take
$n=4$.) Having in mind this metric on $X^2$ we will usually endow
the set $\G$ of $G$-discs with the topology of $C^1$-convergence on
the closed disc $\overline \D$.

When dealing with an individual $G$-disc we usually consider the
generic case when its boundary is embedded. The following definition
selects those $G$-discs which participate in the continuity
principle.

\begin{definition}\label{Definition 2}
A $G$-disc $d$ is $G$-homotopic to a constant, or for short $d$ is a
$G_0$-disc, if there is a continuous family of $G$-discs joining $d$
to a constant disc. The set of $G_0$-discs is denoted by
$\mathcal{G}_0$.
\end{definition}

More detailed, the existence of the $G$-homotopy  means, that there
is a continuous mapping $F(t,z),\;\; t\in I=[0,1]$, $z$ in a
neighbourhood of $\overline {\disc}$, such that for each $t\in
(0,1]$ the mapping $z \rightarrow F(t,z)$ is a $G$-disc, moreover,
$F(1,z)=d(z)$ and the mapping $z \rightarrow F(0,z)$ maps the disc
to a point which is then automatically contained in $G$.

Notice that the existence of a $G$-homotopy to a constant is
equivalent to the existence of a $G$-homotopy to an analytic disc
which is embedded into $G$ and whose image has small diameter. In
other words, $\GN$ is the connected component of $\G$ that contains
small analytic discs embedded into $G$.

For convenience, in the sequel we will frequently use two ways of
notation for a continuous map $\mathcal{A}$ defined on a subset of
$\R \times \C$, namely $\mathcal{A}(t,z)=\mathcal{A}_t(z) $.

The reason to consider $G_0$-discs is the following lemma which can
be considered as continuity principle applied to $G$.

\begin{lemma}\label{Lemma 1}
Any $G_0$-disc  $d:\overline{\disc} \rightarrow X^2$ can be lifted
to a (uniquely defined) immersion $\tilde{d}$,
$\tilde{d}:\overline{\disc} \looparrowright \tilde{G}$ into the
envelope of holomorphy $\tilde{G}$ of $G$, such that $\mathcal{P}
\circ \tilde{d} =
d$ and $\tilde{d}(\bd{\disc}) \subset i(G)$.\\
\end{lemma}

Here $\mathcal{P}:\tilde{G}\rightarrow X^2$ is the natural
projection and $i:G \rightarrow \tilde{G}$ is the natural embedding
of $G$ into the envelope of holomorphy $\tilde{G}$ with $\mathcal{P}\circ i = id $ on $G$.

Note that the lifted disc $\tilde{d}$ may have less
self-intersections than the disc $d$. We do not know a description
of those $G_0$-discs which lift to embedded discs in the envelope of
holomorphy.

The proof of the lemma will be given below in section 3.

We are interested in the whole image $d({\D})$, but it will be
convenient to obtain each point in the image as center of another
analytic disc obtained by precomposing with an automorphism of the
unit disc. In detail, let $d$ be a $G$-disc and $p=d(z),\;\; z\in
{\D}$. Denote by $\varphi _z$ an automorphism of the unit disc
${\disc}$ which maps $0$ to $z$ and consider $d\circ \varphi _z:
\overline{\disc} \rightarrow X^2$. The disc $d\circ \varphi _z$ is a
$G$-disc with center $p=d\circ \varphi _z(0) $. Multiple points of
an immersed disc $p=d(z_1)=d(z_2)$ correspond to centers of
different discs $d\circ \varphi _{z_1}$ and $d\circ \varphi _{z_2}$.

Points in the envelope of holomorphy may occur as centers of many
different lifted $G$-discs. Introduce an equivalence relation in the
set $\mathcal{G}_0$ of $G_0$-discs. Notice that equivalent discs
have the same center.

\begin{definition}\label{Definiton E}
The equivalence relation on $\GN$ is the relation generated by the
following two conditions.

\begin{itemize}

\item[(1)] $\GN$-discs contained in $G$ and having common center are
equivalent.
\item[(2)] Equivalence is preserved under homotopies of equally
centered $G$-disc pairs.
\end{itemize}

\end{definition}

Equivalently, in condition (1) we may consider analytic discs with
images of small diameters embedded into $G$ instead of all
$\GN$-discs with image in $G$.

The second condition can be rephrased in more detail as follows. A
homotopy of pairs of equally centered $\GN$-discs is a continuous
family of ordered pairs of $G$-discs, i.e. a continuous family of
pairs of mappings $(F_1(t,z),F_2(t,z)), t\in I, $ $z$ in a
neighbourhood of $\overline{\disc}$, such that for each $t \in
[0,1]$ both mappings $F_j(t,z),\; z \in \overline{\disc}$, $j=1,2,$
define $G_0$-discs and their centers $p(t)=F_1(t,0)=F_2(t,0)$
coincide (but may depend on the parameter $t$).

Condition (2) says the following: Suppose the initial pair of discs
of the homotopy (i.e. the pair corresponding to the parameter $t=0$)
consists of equivalent discs, then so does the terminating pair
(i.e. the pair corresponding to the parameter $t=0$).

In section 2 below we describe a construction which leads to
building all possible pairs of equivalent $G_0$-discs according to
definition 3. The construction will be given in terms of trees. The
motivation for considering the introduced equivalence relation is
the following lemma which will be proved in section 3.

\begin{lemma}\label{Lemma 2}
Centers of equivalent $\GN$-discs lift to the same point in the
envelope of holomorphy: \\
If $d_1$ and $d_2$ are equivalent
$G_0$-discs then $\tilde d_1(0)= \tilde d_2(0)\in \tilde G$.
\end{lemma}

Our main theorem is the following.

\begin{theorem}\label{Theorem} Let $G$ be a domain in a Stein surface $X^2$.
Then the set of equivalence classes of $G_0$-discs can be equipped
with the structure of a connected Riemann domain $\hat G$ over
$X^2$. The natural projection $\hat {\mathcal{P}}: \hat G
\rightarrow X^2$ assigns to each equivalence class of discs their
common center. There is a natural embedding $\hat i:G \rightarrow
\hat G$, $\hat {\mathcal{P}} \circ \hat i = id$, which assigns to a
point in G the equivalence class represented by discs embedded into
$G$ (of small diameter) and centered at this point.

The Riemann domain $\hat G$ coincides with the envelope of
holomorphy $\tilde G$ of $G$.

The number of sheets of $\tilde G$ over a point $p \in X^2$ equals
the number of equivalence classes of $G_0$-discs with center $p$.
\end{theorem}

It has been a classical fact that the whole envelope of holomorphy $\tilde G$
of a domain $G$ in a Stein manifold $X^2$ can be covered by the following successive
procedure.

Put $\mathcal{D}_0 = i(G) \subset \tilde G$. Consider analytic discs
immersed into $\tilde G$ with boundary in $\mathcal D _0$ and call
them $\tilde{\mathcal{D}}_0$-discs. See definition 1, but now $G$ is
replaced by $\mathcal{D}_0 = i(G)$ and $X^2$ is replaced by $\tilde
G$. A continuous family of  $\tilde{\mathcal{D}}_0$-discs which
joins a given $\tilde{\mathcal{D}}_0$-disc $d$ with a constant disc
is called a continuity-principle-family. The points in the image of
$d$ are said to be reachable by applying the continuity principle
once. See definition 2 with $G$ replaced by $\mathcal{D}_0 =i(G)$
and $X^2$ replaced by $\tilde G$. By the continuity principle (see
e.g. \cite{FrGr}) any analytic function in $i(G)$ has analytic
continuation to a neighbourhood of the image of $d$. This
distinguishes the present situation from that of lemma 1. The discs
of the family in Lemma 1 are immersed into $X^2$ rather than into
$\tilde G$. In the situation of Lemma 1 near self-intersection
points of the disc multi-valued analytic continuation may occur.

Let $D_{j+1}, \; j=0,1,...., $ be the open subset of $\tilde G$
obtained from $D_j$ by adding all points of $ \tilde G$ reached from
$D_j$ by applying the continuity principle once. The classical fact
is that $\tilde G$ is equal to the union of all $D_j$.

The theorem states that, actually, all points of the envelope of
holomorphy $\tilde G$ can be reached from $i(G)$ by applying the
continuity principle only once. Moreover, another observation of
Theorem 1 is the following. Information about the topology of the
envelope of holomorphy is contained in the intersection behaviour of
homotopies of $\GN$-discs (which depends on the Stein manifold in
which the domain is included).
\medskip

Notice that there is no unique definition of Riemann domains in the
literature. Here we use the terminology of Grauert (see
{\cite{Grau}}). In this terminology a Riemann domain over an
$n$-dimensional Stein manifold $X^n$ is a complex manifold of
dimension $n$ with no more than countably many connected components
which admits a locally biholomorphic mapping (called projection) to
$X^n$. Such Riemann domains are separable (\cite{Ju}). We do not
require (as done e.g. in {\cite{H}}) that analytic functions on a
Riemann domain separate points.

Together with the projection $\hat {\mathcal{P}}: \hat G \rightarrow
X^2$ we will use the projection $\mathcal{P}_0: \GN \rightarrow X^2$
which assigns to each individual $G_0$-disc its center, and the
mapping $\hat {\mathcal{P}_0}: \GN \rightarrow \hat G$ which assigns
to each $G_0$-disc the equivalence class it represents. Notice that
$\mathcal{P}_0 = \hat {\mathcal{P}} \circ \hat {\mathcal{P}_0}$.
Later we will use liftings of mappings with respect to different
projections. For instance, let $E$ be a topological space and $\psi
: E \rightarrow X^2$ be a continuous mapping. A continuous mapping
$\overset{\circ}{\psi} :E\rightarrow \GN$ is a lift of $\psi$ to
$\GN$ if $\mathcal{P}_0 \circ \overset{\circ}{\psi} = \psi$.
Respectively, a continuous mapping $\hat \psi : E\rightarrow \hat G$
with $\hat{\mathcal{P}} \circ \hat {\psi} = \psi$ is a lift of
$\psi$ to $\hat G$. To specify which lift is meant we will either
indicate the projection itself or the source and the target space of
the projection.

As a corollary of the theorem we obtain the following result which was surprisingly
not known before.

\begin{corollary}\label{Corollary 1}
Let $G$ be a domain in a Stein manifold $X^2$ and $\tilde G$ its
envelope of holomorphy. Then for each of its point $p$ the envelope
of holomorphy $\tilde G$ contains a (non-singular) embedded Riemann
surface (and also an immersed analytic disc) passing through $p$ and
having its boundary in $i(G)$.
\end{corollary}

The proof of the corollary will be given below in section 11.

Corollary 1 should be contrasted to counterexamples known for
polynomial hulls. Namely, there are compact subsets $K$ of
$\mathbb{C}^n, n\ge 2,$ with the following property. There is a
point in the polynomial hull $\hat K $ such that for any small
enough neighbourhood $U$ of $K$ there is no Riemann surface with
boundary in $U$ passing through this point.

The following question seems natural.

\begin{question}\label{Question 1}
For a point $p\in \tilde G$, what is the minimal genus of a
(non-singular) Riemann surface in $\tilde G$ passing through $p$
with boundary in $i(G)$?
\end{question}

This genus may serve as a measure how "far" the point $p$ is from
$i(G)$.

The second corollary states that for each closed orientable surface
in $\tilde G$ there is a homotopy that moves a big part of it to
$i(G)$; what remains in $\tilde G \setminus i(G)$ is an immersed
analytic disc in $\tilde G$ with boundary in $i(G)$. We may assume
that the disc is either empty or belongs to $\G \setminus \GN$.

\begin{corollary}\label{Corollary 2}
Let $G$ and $\tilde G$ be as in the preceding corollary. Let $f:S
\hookrightarrow \tilde G$ be a connected closed orientable surface
embedded into $\tilde G$. Then there exists a homotopy to a
(singular) surface $F:S \rightarrow \tilde G$ ($F$ a continuous
mapping), such that either $F(S)$ is contained in $i(G)$ or there is
a disc $\Delta \subset S$ such that $F(S \setminus \Delta)$ is
contained in $G$ and (for a suitable complex structure on $\Delta$)
$F:\overline {\Delta} \rightarrow \tilde G$ is an immersed analytic
disc in the envelope of holomorphy $\tilde G$.

\noindent In particular, $F:S \rightarrow \tilde G$ represents the
same homology class in $H_2(\tilde G)$ as the original surface.

\end{corollary}

The condition that $f$ is an embedding
can be skipped. It is sufficient that $f$ is continuous.

The obstruction to move a surface $f:S \hookrightarrow \tilde G$ to the lift $i(G)$
of the original domain can be described in different terms.

Denote by $\mathfrak{L}^a$ the set of loops in $G$ that bound
analytic discs in $X^2$ (equipped with the topology of $C^1$
convergence). Let $\mathfrak{L}^a_0$ be the connected component of
$\mathfrak{L}^a$ which contains constant loops. In the situation of
Corollary 2 a non-trivial analytic disc $F:\overline {\Delta}
\rightarrow \tilde G$ emerges from the existence of a
non-contractible closed curve in the set $\mathfrak{L}^a_0$ (see
below section 11).

There is a variant of Corollary 2 for surfaces with boundary in $i(G)$.
We formulate only the following special case of it.

Denote by $\varphi$ the natural homomorphism from $\pi _1 (G)$ to
$\pi _1 (\tilde G)$ which is induced by inclusion $i: G \rightarrow
\tilde G$. It is known that $\varphi$ is surjective (\cite{Ke}).
(Notice that this result of (\cite{Ke}) can also be obtained as an
immediate consequence of Theorem 1, see below section 11.)

\begin{corollary}\label{Corollary 3}
Any element of the fundamental group of $G$ which is in the kernel
of $\varphi$ can be represented by a loop in $i(G)$ which bounds an
analytic disc that is immersed into $\tilde G$.

\end{corollary}

A reformulation of the corollary is the following. Any loop in
$i(G)$ which is contractible in $\tilde G$ is homotopic in $i(G)$ to
a loop that bounds an immersed analytic disc in $\tilde G$.

The corollary can be slightly strengthened. Namely, given any point
$p \in \tilde G$, the analytic disc of Corollary 3 may be taken to
pass through $p$. An analogous remark holds for Corollary 2.

\medskip

We do not know which elements of the kernel $\varphi$ can be
represented by boundaries of {\it embedded} holomorphic discs.

\medskip

We state separately the versions of Corollary 2 and 3 for Stein
fillings. A relatively compact strictly pseudoconvex domain $\Omega$
in a Stein surface is a Stein filling of the contact three-manifold
$M^3$ if $M^3$ is contactomorphic to $\partial \Omega$ with the
contact structure induced by the complex tangencies.

\begin{corollary}\label{Corollary 4a}
Let $\Omega$ be a relatively compact strictly pseudoconvex domain in
a Stein surface $X^2$ with boundary $\partial \Omega = M^3$. Let
$f:S \hookrightarrow \overline {\Omega}$ be a connected closed
orientable surface embedded into $\overline {\Omega}$. Then there
exists a homotopy to a (singular) surface $F:S \rightarrow \overline
{\Omega}$ ($F$ a continuous mapping), such that either $F(S)$ is
contained in $\partial {\Omega} = M^3$ or there is a disc $\Delta
\subset S$ such that $F(S \setminus \Delta)$ is contained in $M^3$
and (with a suitable complex structure on $\Delta) $
$F:\overline{\Delta} \rightarrow \overline{\Omega}$ is an immersed
analytic disc in $\overline {\Omega}$ with boundary in $M^3$.

In particular, $F:S \rightarrow \overline {\Omega}$ represents the
same homology class in $H_2(\overline {\Omega})$ as the original surface.
\end{corollary}

\begin{corollary}\label{Corollary 4}
Let as before $\Omega$ be a relatively compact strictly pseudoconvex
domain in a Stein surface $X^2$ with boundary $\partial \Omega =
M^3$. Denote by $\varphi$ the homomorphism from $\pi _1 (M^3)$ to
$\pi _1 ( \overline \Omega )$ induced by inclusion $M^3
\hookrightarrow \overline \Omega$.

Then any element in the kernel $ker \varphi$ can be represented by
the boundary of an analytic disc immersed into $\overline \Omega$.
\end{corollary}

Again, for any point $p \in \Omega$ the disc can be chosen passing
through $p$.

We do not know whether in the situation of Corollary 5 one can
always find an {\it embedded} analytic disc (in other words whether
a "holomorphic version" of the loop theorem holds) or whether the
minimal number of self-intersections of analytic discs whose
boundaries represent a given element of the fundamental group of
$M^3$ determines a non-trivial invariant depending on the contact
manifold $M^3$, the filling $\Omega$ and the element of the
fundamental group.

Stepan Orevkov proposed to consider the following example where
$\Omega$ is a tubular neighbourhood of a Lagrangian torus in
$\mathbb{C}^2$. In this case all elements in the kernel of the
homomorphism $\varphi$ can be represented by boundaries of embedded
analytic discs.

\medskip

\noindent {\bf Example}. Let $T$ be the tube domain $\Delta  \oplus
i{\mathbb{R}^2 } $ where $\Delta$ is the unit disc $\Delta  \defined
\{   x_1^2 + x_2^2 < 1 \} $ in $\mathbb{R}^2$. The map $\exp:
(z_1,z_2) \rightarrow (\exp (z_1), \exp (z_2))$ is a covering from
$T$ onto a neighbourhood $\Omega$ of the standard torus $\bd \D
\times \bd \D$. The image of $\bd {\Delta} \times \{0\}$ (with
counterclockwise orientation) under the aforementioned mapping
represents a generator of the kernel of the homomorphism $\varphi :
\pi _1 (\partial \Omega) \rightarrow \pi _1(\overline \Omega)$. The
analytic discs $f_{\pm}(z)=(z,\mp iz), z \in \overline \D,$ are embedded
into the closure of $T$. Their boundaries are homotopic to $\bd
{\Delta} \times \{0\}$ with counterclockwise, respectively,
clockwise orientation. The images of the discs under the map $\exp$
are embedded analytic discs in $\overline \Omega$ whose boundaries
represent a generator of the kernel $\varphi$, respectively, its
inverse. Multiples of the generator can be represented by the
boundary of the following embedded discs. Consider $N$ discs
$f_{+,j}(z)=(z,- iz + i \phi _j)$ in $\overline T , z \in \overline
\D,$ for $N$ different points $\phi_j \in [0,\pi)$. The analytic
discs $\exp \circ f_{+,j}$ are embedded and pairwise disjoint. Join
the boundaries of two consecutive discs by a Legendrian arc. Suppose
all Legendrian arcs are pairwise disjoint, without
self-intersections and meet the union of the boundaries of the discs
exactly at the endpoints . Approximate the union of all the analytic
discs and all the arcs by a single analytic disc. (See below section
11 for details.) In the same way we proceed with multiples of the
inverse of the generator. The boundaries of such discs represent all
elements of the kernel.

\medskip

\begin{question}\label{Question 2}
Let $p,q$ and $r$ be pairwise relatively prime integers and
$\varepsilon \ne 0$ a small complex number. Consider the
Milnor-Brieskorn spheres $M(p,q,r)
\defined \{z_1^p + z_2^q + z_3^r = \varepsilon\} \bigcap S^5 \subset
\C ^3$ and their natural filling. What is the minimal numbers of
self-intersections of an analytic disc whose boundary represents a
given element of the fundamental group of $M(p,q,r)$? What are these
numbers for a collection of elements that generate the fundamental
group in the sense of semigroups?
\end{question}

We conclude with the following observation for the case
$M^3=\partial \Omega$ is a homology sphere. Consider any embedded
loop $f:\bd \D \rightarrow M^3$ which bounds an analytic disc in the
filling $\overline \Omega$. We may always assume that the loop
passes through a given base point in $M^3$ (see below the sketch of
Lemma 23). The loop determines a unique element $s_f$ of the second
homology $H_2(\overline \Omega)$. Indeed, consider the analytic disc
$f:\overline \D \rightarrow \Omega$ bounded by this loop and attach
to it along the loop a compact surface with boundary, the surface
contained in $M^3$. We obtain a closed surface in $\overline
\Omega$. Since $H_2(M^3) = 0$ the homology class represented by the
closed surface in $H_2( \overline \Omega)$ does not depend on the
choice of the surface contained in $M^3$ that was attached to the
loop. Further, two loops $f_1$ and $f_2$, $f_j: \bd \D \rightarrow
M^3$ for $j=1,2$, both bounding analytic discs in $\overline \Omega$
determine the same element in $H_2( \overline \Omega)$ if they are
homotopic in $\bd \Omega$ through loops bounding analytic discs. We
do not have a satisfactory description of such homotopies. Notice
that the set of homotopy classes of boundaries of analytic discs
(passing through a given base point) has the structure of a
semigroup.

The present work was done at the Max-Planck-Institut f\"ur
Mathematik and at Toulouse University with a CNRS grant. The author
gratefully acknowledges the unbureaucratic support and hospitality
of these institutions. The author would like to thank N.Kruzhilin,
S.Nemirovski and S.Orevkov for enlightening discussions and a group
of visitors of a Mittag-Leffler semester, including N.Kruzhilin,
L.Lempert, S.Nemirovski, S.Orevkov and A.Tumanov for their interest.
The author is also grateful to F.Forstneric and L.Stout for useful
information concerning references.

\section{A constructive description of the equivalence condition}

Call a pair of equally centered $G_0$-discs an $ec$-pair for short.

\begin{lemma}\label{Lemma 4}
The set of all pairs of equivalent $G_0$-discs can be constructed by
successively choosing and applying a finite number of times one of
the following procedures.

\begin{itemize}
\item[(i)]  Take a pair of small equally centered embedded analytic discs
contained in $G$.
\item[(ii)] Take a pair of $G_0$-discs that is homotopic through $ec$-pairs
to a pair of equivalent discs.
\item[(iii)] Let $d_1,\; d_2,\;...,\; d_N$ be $G_0$-discs such that
consecutive discs $d_k,\; d_{k+1},\;k=1,2,...,N-1,\,$ are
equivalent. Take the pair $(d_1,d_N)$.
\end{itemize}
\end{lemma}

\begin{proof}
Procedures (i) and (ii) give pairs of equivalent discs by conditions
(1) and (2) of definition 3, respectively. Since an equivalence
relation is transitive (iii) gives pairs of equivalent discs.

It remains to see that all pairs of equivalent discs can be obtained
in this way. Consider the property of a pair of discs to belong to
the set constructed by the procedure described in Lemma 3. This is an equivalence relation
since it is symmetric and transitive. Moreover, it satisfies
conditions (1) and (2), and it is minimal with the latter property.
Therefore it coincides with the previous equivalence relation.

\end{proof}

Lemma \ref{Lemma 4} allows to characterize pairs of equivalent discs
as those for which there exists an associated planar rooted tree. (Such a
tree is not uniquely determined for a given pair of discs.) This goes as follows.

Recall that a rooted tree is a connected graph without simple closed
paths with a vertex chosen as root. If the root of the tree is not a
multiple vertex we call the rooted tree simple. Vertices that are
different from the root and have only one adjacent edge are called
leaves. For each pair of vertices there is a unique path joining
them. This allows to orient the edges of the graph "towards the
root". We call the two endpoints of an oriented edge its $minus$-end
and its $plus$-end respectively. (Orientation is towards the
$plus$-end.)

We will consider trees that are (embedded) subsets of the plane with
edges being straight line segments. The following additional
structure is given. Edges whose $plus$-end is a common vertex of the
graph (incoming edges for this vertex) will be given a label and
placed in the following way. When surrounding the common vertex
counterclockwise starting from a point on the first labeled edge, we
meet the edges in the order prescribed by labeling. There is at most
one edge whose minus end is a given vertex (outgoing edge for this
vertex). The outgoing edge is always placed between the last and the
first labeled incoming edge (with respect to counterclockwise
orientation).

Pairs of discs constructed by lemma \ref{Lemma 4} produce planar
rooted trees in the following way.

Pairs of small equally centered embedded analytic discs contained in
$G$ correspond to leaves. A single leaf (see procedure (i)) can be
considered as a tree without edges with its root coinciding with its
leaf.

Providing procedure (ii) with a pair of discs corresponds to
attaching an edge to the root of its tree. The attached edge
corresponds to the homotopy of $ec$-pairs, in particular, each point
on the edge corresponds to a single $ec$-pair. The $minus$-end of
the attached edge is the root of the previous tree, it corresponds
to the original pair of equivalent discs, the $plus$-end is the root
of the new tree, it corresponds to the pair of discs obtained from
the original one by applying procedure (ii).

Procedure (iii) obtains a pair of discs $d_1,d_N$ from the pairs
$(d_1,d_2)$,...,$(d_{N-1},d_N)$ of equivalent discs. This procedure
corresponds to gluing trees together along their common root. More,
detailed, consider the rooted trees $T_1,\;T_2, \;,...,\;T_{N-1}$
corresponding to the aforementioned pairs together with their label.
Identify their roots. The obtained tree may be represented as subset
of the plane, so that the previous trees are ordered
counterclockwise around the common root. We obtain a new rooted
tree, its root corresponds to the pair $(d_1,d_N)$.

We proved the following lemma.

\begin{lemma}\label{Lemma 5}
To each pair of equivalent $G_0$-discs corresponds a planar rooted
tree such that the root of the tree corresponds to this pair. Leaves
correspond to pairs of small equally centered analytic discs
embedded into $G$. Edges correspond to $ec$-homotopies. For each
multiple vertex those edges that have the vertex as $plus$-end are
ordered. In this order their ends correspond to pairs
$(d_1,d_2)$,$(d_2,d_3)$,..., $(d_{N-1},d_N)$. The respective
multiple vertex of the tree corresponds to the pair $(d_1,d_N)$.

There is a continuous mapping $\hat \Phi _T :T \rightarrow X^2$. It
assigns to each point of $T$ the class represented by the equivalent
discs corresponding to this point. The mapping $ \Phi _T = \hat
{\mathcal{P}} \circ \hat \Phi _T$ assigns to each point of the tree
the center of the equivalent discs corresponding to this point.
\end{lemma}

\begin{figure}[h]
\centering
\includegraphics[scale=1.0]{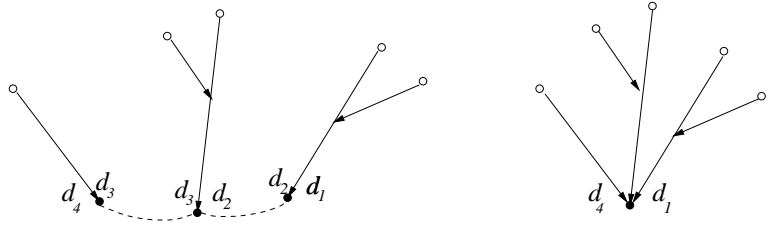}
\caption {Planar rooted trees associated to pairs of equivalent
discs (leaves indicated by white dots, roots by black dots)}
\end{figure}

Consider a planar tree $T$ that has a non-trivial edge. Its
complement $\hat {\C} \setminus T$ in the Riemann sphere is a simply
connected domain. Consider a conformal mapping $\phi : \D
\rightarrow \hat {\C} \setminus T$. The mapping $\phi$ extends
continuously to the closed disc $\overline \D$. Consider the
boundary curve $\phi :
\partial \D \rightarrow \C$ of the conformal mapping and reverse its
orientation. Note that this curve is the limit of the simple closed
curves $\phi (|z| = r), \; r<1,\; r \rightarrow 1$, oriented
suitably. The image of the limit curve is contained in the tree $T$.
We may think about the curve "surrounding the tree counterclockwise
along its sides." We have in mind that we associate to each edge of
the tree its left side and its right side (copies of the edge which
are the limit of its shifts to the left, respectively to the right,
when moving along the edge according to orientation; recall that
trees are oriented "towards the root").

\begin{definition}\label{Definition 3a}
For a planar tree $T$ the non-parametrized curve represented by the
curve $\phi (\partial \D)$ with reversed orientation is called the
pellicle of the tree $T$.

The punctured pellicle of the tree is obtained by removing from the
pellicle the point over the root and adding instead two endpoints
over the root.
\end{definition}

This means that the initial point
of the punctured pellicle is related to the tree in the following way. Consider
all edges of the tree adjacent to the root and have them labeled as
above, i.e. counterclockwise when traveling around the root. Take
the point over the root on the left side of the first labeled edge.
This is the initial point of the punctured pellicle of the tree.

Respectively, the terminating endpoint of the punctured pellicle is
the point over the root on the right side of the last labeled edge.

\begin{figure}[h]
\centering
\includegraphics[scale=0.5]{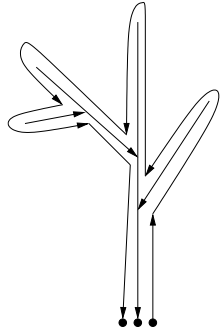}
\caption{A planar rooted tree $T$  and a curve approximating its
punctured pellicle}

\end{figure}

We will parametrize the punctured pellicle by an interval
(standardly it will be the unit interval $[0,1]$) with affine
parametrization on the sides of the edges. We denote the punctured
pellicle by $m_T :[0,1] \rightarrow \C $. The image of $m_T$ covers
the open edges of the tree $T$ twice and covers the vertices with,
maybe, higher multiplicity.

We need the following definitions.

\begin{definition}\label{Definition 3b1} Let $\alpha $ be a curve
in the plane and let $\Phi \circ \alpha$ be a curve in $X^2$. A
curve $\overset{\circ}{\alpha}$ in $\GN$ for which $\mathcal{P}_0
\circ \overset{\circ}{\alpha} = \Phi \circ \alpha$ is called a halo
assigned to $\alpha$ and $\Phi$.
\end{definition}

Notice that the halo is a continuously varying family of analytic
discs around points in the image of the curve $\mathcal{P}_0 \circ
\overset{\circ}{ \alpha}$ in $X^2$. The latter curve is the curve of
centers of the discs constituting the halo. The curve
$\overset{\circ}{\alpha}$ can be considered as a mapping with values
in $X^2$ of the trivial disc fibration over the curve $\alpha$. The
restriction of the mapping to the respective circle fibration has
values in $G$.

\begin{definition}\label{Definition 3b}
A planar rooted tree $T$ with punctured pellicle $m_T$ together with
a continuous mapping $\Phi _T : T \rightarrow X^2$ is called a
dendrite. The mapping $\Phi _T \circ m_T$ is called the punctured
pellicle of the dendrite (opposed to the punctured pellicle $m_T$ of
the underlying tree). If the mapping $\Phi _T \circ m_T$ lifts to a
mapping $\overset{\circ}{m}_T$ to $\GN$ (i.e. $\mathcal{P}_0 \circ
\overset{\circ}{m}_T  = \Phi _T \circ m_T$) we call
$\overset{\circ}{m}_T$  the punctured halo of the dendrite. The set
$(T,m_T,\Phi_T, \overset{\circ}{m}_T)$ is called a dendrite with
punctured halo and denoted by $\mathbf{T}$.
\end{definition}

Recall that for each point $\Phi _T \circ m_T(t)$ in the punctured pellicle of the
dendrite the value of the halo at this point is an analytic disc centered at this point.

Note that we do not require here that the tree is associated to a
pair of equivalent discs. In particular, we do not require that the values of
$\Phi$ at the leaves are contained in $G$ and the values of $\overset{\circ}{m}_T$
at the leaves are discs embedded into $G$.

The following lemma holds.

\begin{lemma}\label{Lemma 6} Let $(d_1,d_2)$ be a pair of equivalent $G_0$-discs.
Then there exists a dendrite $(T, m_T, \Phi _T, \overset
{\circ}{m}_T)$ with punctured halo $\overset {\circ}{m}_T$ such that
(for standard parametrization) $\overset {\circ}{m}_T (0)=d_1$ and
$\overset {\circ}{m}_T (1)=d_2$.

Moreover, at each of the leaves of the tree the value of
$\overset{\circ}{m}_T$ is an analytic disc of small diameter embedded into $G$
and its center, the value of $\Phi _T \circ m_T$, is a point in $G$.

Further, there is a lift $\hat \Phi_T : T \rightarrow \hat G$ of
$\Phi _T$, $\hat{\mathcal{P}} \circ \hat \Phi_T = \Phi_T$, such that
$\hat{\mathcal{P}}_0 \circ \overset {\circ}{m}_T  =  \hat \Phi_T
\circ m_T$.

\end{lemma}

A dendrite with the properties described in Lemma 5 is said to be
associated to the pair $(d_1,d_2)$ of equivalent discs.

\medskip

\noindent {\it Proof of Lemma 5}. Let $T$ be the planar rooted tree
associated to the pair $(d_1,d_2)$ by Lemma 4. Let $\Phi _T $ be the
mapping from the tree into $X^2$ defined in that Lemma. We want to
show that for the punctured pellicle $m_T$ of the tree $T$ the
mapping $\Phi _T \circ m_T$ lifts to a continuous mapping $\overset
{\circ}{m}_T$ with $\overset {\circ}{m}_T (0)=d_1$ and $\overset
{\circ}{m}_T (1)=d_2$.

Recall that edges of the tree $T$ correspond to homotopies of
(ordered) $ec$-pairs. A homotopy of pairs of $G_0$-discs consists of
two homotopies of $G_0$-discs, namely the homotopies defined by the
first labeled, respectively second labeled, discs. Assign the first
homotopy of $G_0$-discs to the left side (i.e. to the first side
when surrounding the edge counterclockwise starting from the root),
and the second homotopy to the right side of the edge.

The statement of the lemma can be proved by induction using the
successive procedure of construction described in lemma 3.

First we consider trees consisting of an edge adjacent to a leaf.
Change slightly those pairs of discs which correspond to points
close to the leaf so that the pair associated to the leaf itself
consists of two equal discs. Then the above described procedure
gives a continuous mapping from the punctured pellicle of the edge
into the set of $G_0$-discs with the desired values at the sides
over the root. The value of the punctured halo at the leaf is a
small disc embedded into $G$.

In the case corresponding to procedure (iii) there are several
rooted trees $T_j,\;J=1,...,N-1$, and we assume that for each tree
$T_j$ there is a continuous lift $\overset {\circ}{m}_{T_j}$ of
$\Phi_{T_j} \circ m_{T_j}$ to $\mathcal G_0$ which coincides at the
left, respectively right sides over the roots with $d_j$,
respectively $d_{j+1}$. The trees are glued together at their root
and placed in the plane counterclockwise around the common root. The
punctured pellicle of the new tree is obtained by gluing the right
side over the root of $T_j$ to the left side over the root of
$T_{j+1}$. It is clear now that the values of the
punctured halo of the trees $T_j$ match so that for the new tree $T$
we obtain a continuous lift of $\Phi_{T} \circ m_{T}$ into $\GN$. At
the leaves the halo takes values in the set of small analytic discs
embedded into $G$.

The general case corresponding to (ii) is easier and left to the
reader. \hfill  $\square$

\medskip

\begin{figure}[h]
\centering
\includegraphics[scale=1.0]{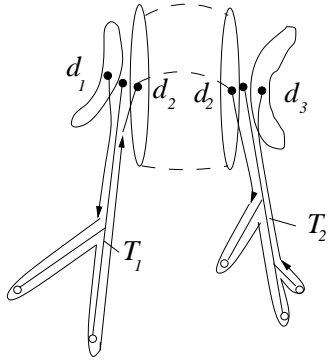}
\caption{Matching the halo at common endpoints of punctured
pellicles of two trees}
\end{figure}

\medskip

We will identify rooted trees realized as subsets of $\C$ if there
is a piecewise affine homeomorphism of the plane mapping one tree to
the other fixing the root and mappings edges (i.e. straight line
segments joining vertices) to edges. We will identify the
parametrized punctured pellicle and halo of such trees if they are
obtained by precomposing with the mentioned homeomorphism.

We will not distinguish between different parametrizations of the
pellicles and of the halo for a given embedding of a tree into $\C$
if the parametrization does not play a role.

\medskip
\section{Plan of Proof of the theorem}

The proof of the theorem is divided into three steps according to the
following propositions.

\begin{proposition}\label{Proposition 1}
The set of equivalence classes of $G_0$-discs can be equipped with
the structure of a connected Riemann domain $(\hat G, \hat
{\mathcal{P}})$ over $X^2$. The projection  $\hat
{\mathcal{P}}$ associates to each equivalence class its center.
There is a natural embedding $\hat i: G \rightarrow \hat G$ of
$G$ into $\hat G$, such that $\hat
{\mathcal{P}} \circ \hat i = id$ on $G$.
\end{proposition}

\begin{proposition}\label{Proposition 2}
For each analytic function on $G$ its push-forward to $\hat i (G)$
extends to an analytic function on $\hat G$.
\end{proposition}

The most subtle part of the proof of the theorem is the following
proposition.

\begin{proposition}\label{Proposition 3}
The Riemann domain $\hat G$ is pseudoconvex.
\end{proposition}

The concept of pseudoconvexity of Riemann domains over $\mathbb C
^n$ goes back to Oka (\cite{Oka}). Oka showed that pseudoconvex
Riemann domains over $\mathbb C^n$ are holomorphically convex (i.e.
hulls of compacts with respect to analytic functions on the Riemann
domain are compact.) In the paper \cite{DoGr} the notion of
pseudoconvexity of an arbitrary complex manifold is introduced.
Moreover, the authors present several equivalent characterizations
of pseudoconvexity and extend Oka's result to Riemann domains over
arbitrary Stein manifolds. Together with results of Grauert
(\cite{Grau}) this implies the following theorem.

\begin{thmdgo} \label{theoremDGO}
A pseudoconvex Riemann domain over a Stein manifold is a Stein manifold.
\end{thmdgo}

This theorem shows, in particular, that holomorphic functions on
pseudoconvex Riemann domains separate points (see \cite{Grau} and
\cite{DoGr}).

\medskip

\noindent {\it The three propositions imply the Theorem
\ref{Theorem}.} Indeed, propositions \ref{Proposition 1} and
\ref{Proposition 3} show that the set of equivalence classes of
$G_0$-discs can be equipped with the structure of a Riemann domain
$(\hat G,\; \hat {\mathcal{P}})$ over $X^2$, and moreover $\hat G$
is a  Stein manifold. Proposition 2 shows that $\hat G$ is a
holomorphic extension of $G$ (see \cite{H} , chapter 5..4).
Therefore $\hat G$ coincides (up to a holomorphic isomorphism) with
the envelope of holomorphy $\tilde G$ (see \cite{H} , theorem
5.4.3). \hfill $\square$

\medskip

We will provide now proofs of the propositions.

\medskip

\noindent {\it Proof of Proposition 1.} We start with the
construction of a complex atlas on the set of equivalence classes of
$G_0$-discs. Take an equivalence class $\hat d$ and choose a
representative $d \in \hat d$. Denote the point $d(0) \in X^2$ by
$p$. Associate to $d$ a Riemann domain $\mathcal R _d = (V_d, F_d)$
over $X^2$ such that $d$ lifts to it as an embedded disc and,
moreover, $\mathcal R _d$ is foliated by analytic discs close to the
lifted one. Such a Riemann domain may be constructed in a standard
way. Take a small tubular neighbourhood $V_d = (1+\eps)\disc \times
\delta \D$ of $\overline {\disc} \times \{ 0\}$ in $\mathbb C^2$.
Here $\eps > 0$, $\delta > 0$ are small numbers. Put
$F_d(z_1,0)=d(z_1)$, $|z_1|< 1 + \eps$, and choose a holomorphic
vector field $\mathcal V: V_d \rightarrow TX^2 $ such that $\mathcal
V |(1+\eps)\D \times \{0\}$ is transversal to $F_d(z,0),\, z\in
(1+\eps)\disc$. Denote by $\Phi$ its flow. Then, taking
$F_d(z_1,z_2)= \Phi_{z_2}(F_d(z_1,0))$ and shrinking the Riemann
domain $(V_d, F_d)$ if necessary, we arrive at a Riemann domain that
has the required properties. For each $z_2$, $|z_2|<\delta$, the
analytic disc $F_d \mid \overline \disc \times \{z_2\}$ is a
$G_0$-disc since the central disc $d$ is a $G_0$-disc.

Consider now the set of equivalence classes of $G_0$-discs. Take an
arbitrary element $\hat d$ of this set, choose a representative $d$
and associate to it a Riemann domain $\mathcal R _d $. We want to
define a Euclidean set in the set of equivalence classes that
contains $\hat d$. For this purpose we use the discs of the
foliation of $\mathcal R _d $ in the following way. Choose a
neighbourhood $N_d$ of zero in $V_d$ so that $F_d$ is biholomorphic
from $N_d$ onto a neighbourhood $Q_d$ of $p$ in $X^2$. Associate to
each point $q\in Q_d$ the unique disc $d^q$ of the foliation of
$\mathcal R _d$ which passes through $q$, normalized so that $q$
becomes its center. Take the equivalence class $ {\hat d} ^q$ which
is represented by $d^q$. Define the set ${\hat N}^d = \{ {\hat d}^q:
q \in Q_d \}$ and the mapping $ \hat {\mathcal {P}}_d: {\hat N}^d
\rightarrow Q_d,\; \hat {\mathcal{P}}_d ({\hat d} ^q)=q$. Call this
set a standard neighbourhood of $\hat d$ associated to the
representative $d \in \hat d$, the Riemann domain $\mathcal{R}_d$
and the set $Q_d$. Call $\hat {\mathcal{P}}_d$ the related standard
projection.

The following lemma implies that standard neighbourhoods form a basis of a Hausdorff
topology in the set of equivalence classes of $G_0$-discs.

\begin{lemma}\label{Lemma 7}
Let $\hat d_1$, and $\hat d_2$ respectively, be equivalence classes
of $G_0$-discs. Suppose $\hat N_1$ and $\hat N_2$ are standard
neighbourhoods of $\hat d_1$ and $\hat d_2$, respectively, and $\hat
{\mathcal{P}}_1: \hat N_1 \rightarrow Q_1$ and $\hat
{\mathcal{P}}_2: \hat N_2 \rightarrow Q_2$ are the related standard
projections onto the open subsets $Q_1$ and $Q_2$ of $X^2$. Suppose
$\hat N_1$ and $\hat N_2$ intersect. Let $\hat d$ be a point in
their intersection, hence $\hat {\mathcal{P}}_1 (\hat d)=\hat
{\mathcal{P}}_2 (\hat d)$. Denote the latter point by $p$. It is
contained in $Q_1 \cap Q_2$.

Then $\hat N_1$ and $\hat N_2$ intersect over the whole connected
component $Q^p$ of the intersection $Q_1 \cap Q_2$ which contains
$p$. In other words, for $q \in Q^p$ the inclusion $\hat
{\mathcal{P}}_1 ^{-1}(q)=\hat {\mathcal{P}}_2 ^{-1}(q) \subset \hat
N_1 \cap \hat N_2 $ holds.
\end{lemma}

\medskip

It is clear from the lemma, that standard neighbourhoods form the
basis of a topology. The lemma also implies that this topology is
Hausdorff. Indeed, equivalence classes of $G_0$-discs with different
center have obviously non-intersecting standard neighbourhoods. Let
now $\hat d_1$, and $\hat d_2$ be distinct equivalence classes with
equal center. Take standard neighbourhoods $\hat {\mathcal{P}}_j:
\hat N_j \rightarrow Q_j$ of $\hat d _j$, $j = 1,2$. Let $Q^{p_0}$
be the connected component of $Q_1 \bigcap Q_2$ that contains the
common center $p_0$ of $\hat d_1$ and $\hat d_2$ . Then by the lemma
${\hat {\mathcal{P}}_j}^{-1} (Q^{p_0})$ are disjoint standard
neighbourhoods of the $\hat d_j$.

\medskip

\noindent {\it Proof of lemma \ref{Lemma 7}.} Let $q$ be any point
in $Q^p$. Join $p$ with $q$ by a curve $\gamma $ in $Q^p$, $\gamma :
[0,1] \rightarrow Q^p, \; \gamma (0)=p,\; \gamma (1)=q$. Let $\hat
\gamma_j \defined \hat {\mathcal{P}}_j ^{-1} \circ \gamma,\; j=1,2$.
By construction the equivalence class $\hat \gamma_j (t),\; j=1,2,
\; t \in [0,1],$ is represented by the unique disc of the foliation
of $\mathcal R_{d_j}$ which passes through $\gamma (t)$ normalized
so that its center becomes $\gamma (t)$. Denote the respective
normalized disc by ${d_j}^{\gamma (t)}$. For $t=0$ the discs
${d_j}^{\gamma (t)},\; j=1,2,$ coincide with the central discs $d_j$
of the foliation.

By the conditions of the lemma the discs $d_1$ and $d_2$ are
equivalent, hence for $t=0$ the pair $({d_1}^{\gamma
(t)},{d_2}^{\gamma (t)})$ consists of equivalent discs. Therefore,
by Definition 3 (see (ii)) for each $t\in [0,1]$ the pair consists
of equivalent discs. For $t=1$ the pair coincides with
$({d_1}^q,{d_2}^q)$. By construction the respective equivalence
classes $\hat {d_1}^q = \hat {d_2}^q$ coincide with the respective
points of ${\hat N} _j$ over $q$. The lemma is proved. \hfill
$\square$

\medskip

The standard neighbourhoods equip the set of equivalence classes of $G_0$-discs with the
structure of a Riemann domain over $X^2$ which we denote by $(\hat G,
\hat {\mathcal{P})}$. The projection $\hat {\mathcal{P}}$ assigns to each equivalence
class of $G_0$-discs its center.

Prove that there is a natural holomorphic embedding of $G$ into
$\hat G$. Indeed, take any point $p \in G$. All analytic discs with
center $p$ and sufficiently small diameter are entirely contained in
$G$ and equivalent to each other (see Definition 3, (i)). Associate
to $p \in G$ this equivalence class of discs which is a point $\hat
p \in \hat G$. The mapping $\hat i$, which maps $p$ to $\hat p$ is
locally biholomorphic according to the way an atlas is introduced on
$\hat G$. The mapping is globally injective and $\hat {\mathcal{P}}
\circ \hat i$ is the identity mapping on $G$. Hence $\hat i$ is
biholomorphic onto its image.

It remains to show that $ \hat G$ is connected. This is an easy
consequence of the following two lemmas which will also be needed
further.

\begin{lemma}\label{Lemma 8}
Let $d: \overline \disc \rightarrow X^2$ be a $G_0$-disc. Let $U$ be
the connected component of $\{\zeta  \in \overline \disc : d(\zeta)
\in G\}$ which contains $\bd \disc$. Then for any $z \in U \cap \D$ the disc
$d \circ \varphi _z$ is equivalent to (small) discs centered at
$d(z) = d \circ \varphi _z(0)$ and contained entirely in $G$.
\end{lemma}

\begin{lemma}\label{Lemma 9}
Consider the set of analytic discs $d: \overline \disc \rightarrow X^2$ such
that $d$ extends to an analytic mapping in a neighbourhood of
$\overline \disc$. Endow the set with the topology of
$C^1$-convergence on the closed disc $\overline \D$. Then the set of
$G_0$-discs is open in this space and the mapping which assigns to
each $G_0$-disc its equivalence class in $\hat G$ is continuous.
\end{lemma}

Postpone the proof of the lemmas for a moment and finish the proof
of proposition \ref{Proposition 1}.

\medskip

{\it End of proof of proposition \ref{Proposition 1}.}

We show that any point in $\hat G$ can be connected with a point in
$\hat i(G)$ by a path. Let $\hat d \in \hat G$ and let $d$ be a
representative of $\hat d$. Take a segment $[0,r] \subset \disc$ in
the unit disc with $d(r) \in G$. Then $d\circ \varphi _t, \; t\in
[0,r]$, is a (continuous) curve of $G_0$-discs. By lemma \ref{Lemma
8} the disc $d\circ \varphi _r$ is equivalent to small discs through
$d(r) \in G$ that are entirely contained in $G$. Taking equivalence
classes ${\hat d}_t = \widehat {d\circ \varphi _t},\; t\in [0,r]$,
and applying lemma \ref{Lemma 9} we obtain a curve in $\hat G$ with
${\hat d}_0 = \hat d $ and ${\hat d}_r \in \hat i(G)$. The
proposition 1 is proved. \hfill  $\square$

\medskip

\noindent {\it Proof of lemma \ref{Lemma 8}.} Since $d$ is a
$G_0$-disc there is a homotopy of $G_0$-discs $d_s$, $s\in [0,1]$,
which joins $d_1 = d$ with a small disc $d_0$ embedded into $G$.
Consider a continuous path $z_s$ in $\D$, $s \in [0,1]$, such that
for each $s$ the point $z_s$ is in the connected component $U_s$ of
$\{\zeta \in \overline \D :d_s (\zeta) \in G\}$ which contains
$\partial \D$. The normalized discs $d_s \circ \varphi _{z_s}$ are
centered at $d_s(z_s) \in G$.

Consider a second continuous family of $G_0$-discs $ D_s$,
$s\in [0,1]$, consisting of small analytic discs embedded into $G$
and centered at $d_s(z_s)$. Then the two discs $d_0 \circ \varphi
_{z_0}$ and $D_0$ are equivalent, hence so are the discs $d_1
\circ \varphi _{z_1}$ and $D_1$ (see conditions (1) and (2)
defining the equivalence relation). \hfill $\square$

\medskip

\noindent {\it Proof of lemma \ref{Lemma 9}.} Let $d$ be a
$G_0$-disc and $\hat d \in \hat G$ its equivalence class. Choose a
Riemann domain $\mathcal R_d =(V_d,F_d)$ foliated by $G_0$-discs
with $d$ being the central leaf. Let $N_d \subset V_d$ be a
neighbourhood of zero and let $Q_d \subset X^2$ be a neighbourhood
of $d(0)$ in $X^2$ such that $F_d: N_d \rightarrow Q_d$ is
biholomorphic. Let $D: \D \rightarrow X^2$ be an analytic disc that
is close to $d$ in the topology of $C^1$-convergence on $\overline
\D$ such that $D$ extends analytically to a neighbourhood of $\overline \D$.
Then $D$ is an immersion of a neighbourhood of $\overline \D$ with
$D(\bd \disc) \subset G$ and
$D(0)$ is close to $d(0)$. After possibly decreasing the neighbourhood of
$\overline \D$ on which $D$ is given
there is a unique lift of $D$ to the
Riemann domain $\mathcal R _d$ that passes through the point
$F_d^{-1}(D(0))$. The lifted disc is equivalent to the disc of the
foliation of $\mathbb R _d$ that passes through this point.
Continuity of the mapping and openess of the set of $G_0$-discs are
now clear. \hfill $\square$

\medskip

The following two lemmas concern genericity of one-parameter
families of analytic discs and will be used in the sequel. Denote
the unit interval by $I=[0,1]$.

\begin{lemma}\label{Lemma 9a}
Let $\eps >0$ be a small number. Any continuous mapping $F: I \times
(1+\eps) \D \rightarrow X^2$ that is fiberwise holomorphic can be
approximated uniformly on $I\times (1+ \frac{\eps}{2}) \D$ by a
continuous mapping that is fiberwise a holomorphic immersion.

The approximation may be done keeping the centers of the discs fixed.

\end{lemma}

\begin{lemma}\label{Lemma 10} Let $\eps$ be a small positive
number. A continuous mapping $F: I\times (1+\eps) \D \rightarrow
X^2$ that is fiberwise a holomorphic immersion can be approximated
uniformly on $I \times (1+\frac{\eps}{2}) \D$ by a holomorphic
mapping $\F$ in a neighbourhood of $I\times (1+\frac{\eps}{2}) \D$
that is fiberwise a holomorphic immersion. Moreover, the
approximation can be made in such a way that $\F$ coincides with $F$
on $\{1\} \times (1+\frac{\eps}{2}) \D $ and is locally
biholomorphic in a neighbourhood of $\{1\} \times (1+\frac {1}{2} \eps)\D$.
\end{lemma}

\noindent {\it Proof of Lemma 10}. Assume first that $X^2$ equals
$\C ^2$. Decreasing $\epsilon >0$ we may replace $F$ by a
$C^1$-mapping which coincides with the previous one on $\{1\} \times
(1+\eps) \D$ and has injective differential on $[1-\delta,1]\times
(1+\eps) \D$ for some small positive number $\delta$. This can be
done so that the new mapping is uniformly close to the old one and
is fiberwise a holomorphic immersion. Denote the new mapping as
before by $F$.

The mapping $F$ can be expressed by Taylor series in the
$z$-variable that converge uniformly for $t \in I$ and $z \in
(1+\frac {3}{4} \eps)\D$:

\begin{equation*}
F(t,z)= \sum_{k=0}^{\infty} a_k(t) z^k.
\end{equation*}

We obtain a uniform estimate for the coefficients

\begin{equation*}
 |a_k(t)| \leq M (1+\frac {3}{4}\eps)^{-k},\;\;\; k=1,2,...,\;\; t\in I,
\end{equation*}

\noindent for a constant M not depending on $k$ and $t$. A similar
estimate holds for the $t$-derivatives $a_k'(t)$ of the
coefficients. The functions

\begin{equation*}
F_N(t,z)= \sum_{k=0}^{\infty} a_k(1) z^k + \sum_{k=0}^N (a_k(t) -
a_k(1)) z^k
\end{equation*}

\noindent converge to $F$ uniformly on $I\times (1+\frac{\eps}{2})
\D$ and $\frac{\partial}{\partial t} F_N(t,z)$ converge uniformly to
$\frac{\partial}{\partial t} F(t,z)$ on this set. It remains to approximate
finitely many of the $a_k$ in $C^1([0,1])$ by analytic
functions in a neighbourhood of $[0,1]$ so that their value at $1$
is fixed and the derivative at $1$ converges to $a_k'(1)$.

For general Stein surfaces $X^2$ we consider a holomorphic embedding
$\mathfrak{F} : X^2 \rightarrow \C ^4$ and proceed as above with the
coordinate functions of the mapping $\mathfrak{F} \circ F$. The
image of the approximating mappings is contained in a small tubular
neighbourhood of $\mathfrak{F} X^2$. It remains to compose with a
holomorphic projection of the tubular neighbourhood onto
$\mathfrak{F} X^2$. \hfill $\square$

\medskip

\noindent {\it Proof of Lemma 9}. The lemma follows from the
holomorphic transversality theorem (\cite{KaZa}, see also \cite{Fo})
by standard dimension counting. For convenience of the reader we
give the short argument.

After uniform approximation on $I\times (1+\eps)\D$ we may assume
that the mapping $F$ is holomorphic on $Y^2 \defined U\times
(1+\frac {3}{4} \eps)\D$  for a neighbourhood $U$ of $I$ in $\C$, in
other words $F$ is a holomorphic mapping from the Stein surface
$Y^2$ into the complex manifold $X^2$. We may assume that the
restriction $F|[0,1] \times \{0\}$ is the same as before and the
mapping is  a fiberwise immersion near the set $U \times \{0\}$.

Denote by $A$ the set of all
elements in the space of $1$-jets $J_{hol}^1(Y^2,X^2)$  of
holomorphic mappings from $Y^2$ to $X^2$ which have vanishing
derivatives in the $z$-direction.
$A$ is an analytic submanifold of $J_{hol}^1(Y^2,X^2)$. A mapping
$\F$ from a subset of $Y^2$ to $X^2$ is fiberwise (for fixed
$t$-variable) an immersion if its $1$-jet extension $j^1\F$ avoids
$A$.

Since the $1$-jet extension of $F$ restricted to $|U \times \{0\}$
avoids $A$, by the holomorphic transversality theorem (\cite{KaZa},
see also \cite{Fo}) the mapping $F$ can be uniformly approximated on
relatively compact open subsets $\overset{\circ}{Y}$ of $Y^2$ by
holomorphic mappings $\F$ with $1$-jet extension transversal to $A$,
fixing its $1$-jet on $U \times \{0\}$. Take for
$\overset{\circ}{Y}$ a set of the form $\overset{\circ}{U}\times
(1+\frac {\eps}{2})\D$ for a relatively compact open subset
$\overset{\circ}{U}$ of $U$ containing $I$.

Note that $A$ has real codimension $4$ in $J_{hol}^1(Y^2,X^2)$ and
$j^1\F$ maps the real $4$-dimensional manifold $\overset{\circ}{Y}$
into $J_{hol}^1(Y^2,X^2)$. Hence for a curve $J\subset
\overset{\circ}{U}$ which is a small perturbation of $I$ the
restriction of $\F$ to $J\times (1+\frac {\eps}{2})\D$ has the
desired property: its $1$-jet extension avoids $A$. \hfill $\square$

\bigskip

\noindent {\it Proof of lemma \ref{Lemma 1}.} Consider the subsets
$\mathfrak{c} \defined ( [0,1)\times \overline \D )\;\bigcup \; (
[0,1] \times \partial \D )$ and $\mathfrak{c}_0\;= \;( \{0\} \times
\overline{\D}) \; \cup \; ([0,1]  \times \bd{\D})\;$ of
$\mathbb{R}\times \C$ and their convex hull $\mathfrak{C} \defined
[0,1] \times \overline {\D}$.

Recall that the most elementary version of the continuity principle
states that any holomorphic function in a neighbourhood of the set
$\mathfrak{c}$ (more generally in a neighbourhood of
$\mathfrak{c}_0$) in $\C ^2$ extends to a holomorphic function in a
neighbourhood of $\mathfrak{C}$ in $\C ^2$.

The proof is completely elementary: The Cauchy type integral over
the circles $\{t\} \times (1+\varepsilon)\partial \D$ ($\varepsilon
>0$ small and $t \in [0,1]$) defines an analytic function in a neighbourhood of
$\mathfrak{C}$ which coincides with the original function in a
neighbourhood of the bottom disc $\{0\} \times \overline \D$.

Let $d$ be a
$G_0$-disc. Let $\F$ be the mapping of lemma 10.
For any analytic function $g$ in $G$ the function $g\circ \F$ is
analytic in a neighbourhood $U$ of $\mathfrak{c}_0\;= \; (\{0\}
\times \overline{\D})\; \cup \; ([0,1] \times \bd{\D})\;$. By the
continuity principle $g\circ \F$ extends analytically to a
neighbourhood of $\mathfrak{C}\; = \;[0,1] \times \overline{\D}$, in
particular it extends analytically to a neighbourhood $V$ of  $\{1\}
\times \overline{\D}$.

The neighbourhood $V$ together with the mapping $\F$ define a
Riemann domain over $X^2$. Use the mapping $\F$ to glue the Riemann
domain to the domain $G$ along a suitable connected neighbourhood of
$\{1\} \times \partial \D$. Any analytic function $g$ on $G$ extends
analytically to the  union of $G$ with the Riemann domain.

Identify points in the union which are not separated by extensions
of holomorphic functions on $G$. This factorization gives a
Hausdorff space (see \cite{H} for the case of $\C^2$ and
\cite{Rossi} for the general case), and hence a Riemann domain which
is an extension domain of $G$ the points of which are separated by
analytic functions. It is biholomorphically equivalent to a subset
of the envelope of holomorphy (the biholomorphic mapping being
compatible with projection), see e.g. \cite{H}. The described
procedure gives an immersion $\tilde d$ of the $G_0$-disc into
$\tilde G$ such that $d=\tilde {\mathcal{P}} \circ \tilde d$ and
$\tilde d (\bd {\D})$ is contained in $\tilde i(G)$. The lemma is
proved. \hfill $\square$

\medskip

\noindent {\it Proof of lemma \ref{Lemma 2}}. The lemma is true for
two discs of small diameter embedded into $G$. Indeed, the mapping
$\hat i$ maps the center of both of them to the same point in $\hat
i(G)$. The statement of the lemma is preserved under homotopies of
pairs of equally centered $G$-discs. Indeed, let $(F_1(t,\cdot),
F_2(t,\cdot)), t \in I,$ be such a homotopy. Suppose for
$d_j=F_j(0,\cdot)$ the desired equality $\tilde d_1(0)=\tilde
d_2(0)$ holds.

Apply lemma \ref{Lemma 1} to each disc $F_j(t,\cdot)$ with $ t \in
I, \, j=1,2$. We obtain a unique lift $\tilde F_j(t,\cdot)$  of each of
the discs to $\tilde G$. As in the proof of lemma \ref{Lemma 9} for
fixed $j$ the lifts of the discs depend continuously on the
parameter $t$. For $j=1,2$ the curve  $\tilde F_j(t,0)$ is a lift to
$\tilde G$ of the same curve in $X^2$, namely, of the curve of the common centers
$F_1(t,0)=F_2(t,0)$ of the pairs. Since by assumption the lifts of
the centers coincide for $t=0$, by uniqueness the lifts of the whole
curve coincide. The lemma is proved. \hfill $\square$

\medskip

\noindent {\it Proof of proposition \ref{Proposition 2}.} Take for
each equivalence class of $G$-discs a representative and consider
the lift of its center to the envelope of holomorphy $\tilde G,$ (see
Lemma \ref{Lemma 1}). By Lemma \ref{Lemma 2} this point does not depend on
the choice of the representative but only on the equivalence class.
This defines a continuous mapping $\rho : \hat G \rightarrow \tilde G$ which respects projections:
$\tilde {\mathcal{P}} \circ \rho = \hat {\mathcal{P}}$. Hence $\rho$ is locally biholomorphic.

This map
maps the set $\hat i(G) $ to $\tilde i(G)$ so that $ \tilde {\mathcal{P}}
\circ \rho = \hat {\mathcal{P}}$ on $\hat i (G)$.
The analytic continuation of functions from $\tilde i (G)$ to the envelope of holomorphy
$\tilde G$ determines analytic continuation of functions from $\hat i (G)$ to $\hat G$.
The statement of the proposition follows.
\hfill $\square$

\medskip

\section{Pseudoconvexity of the Riemann domain $\hat G$}

We come to the most subtle part of the proof of the theorem, namely the proof of Proposition 3.
In this section we reduce Proposition 3 to a lemma with which it is more convenient to work.

Our goal is to
prove that the Riemann domain $\hat G$ is $p^*_7$-convex in the
sense of Docquier and Grauert (see \cite{DoGr}, p. 105/ 106).
Docquier and Grauert proved that this convexity notion is the
weakest of the equivalent conditions for pseudoconvexity of a
Riemann domain over a Stein manifold.

Recall the notion of $p^*_7$-convexity for convenience of the
reader. Denote by $\c {\D}^2$ the set ${\disc}^2 \cup (\overline{\D}
\times \bd{\disc})$. This subset of the closed bidisc is obtained by
removing from $\overline {\D}^2$ its "open face" $\bd{\D} \times
\D$. Following Grauert we denote by $\tilde{\bd} \hat G$ the
"boundary of $\hat G$ in the sense of ends" defined by filters
(\cite{DoGr}, p. 104, \cite{FrGr}, p.100). The notion of
$p^*_7$-convexity uses the definition of an $R$-mapping. An
$R$-mapping into the Riemann domain $\hat G$ is a continuous mapping
$\phi$ from the closed unit bidisc $\overline {\disc}^2$ into the
closure $\hat G \cup \tilde{\bd} \hat G$ of the Riemann domain $\hat
G$ that has the following properties.

\begin{itemize}
\item[(I)] $\phi (\overline{\D}^2) \not \subset \hat G$,
\item[(II)]  $\phi (\c {\D}^2) \subset \hat G$
\item[(III)] The mapping $\hat {\mathcal{P}} \circ \phi$ extends to a
biholomorphic mapping of a neighbourhood of the closed bidisc
$\overline{\disc}^2$ into $X^2$.
\end{itemize}

According to the definition of Docquier and Grauert $\hat G$ is
$p^*_7$-convex, equivalently pseudoconvex, if each end $p \in
\tilde{\bd} \hat G$ of $\hat G$ has a neighbourhood $U(p)$ in $\hat
G  \, \cup \, \tilde{\bd} \hat G$ such that no $R$-mapping with
image in $U(p)$ exists. We will prove that any mapping satisfying
(II) and (III) will violate (I). More precisely, denoting the
extension of the mapping $\hat {\mathcal{P}} \circ \phi$ to a
neighbourhood of the closed bidisc (see (III)) by $\Psi$ and the
mapping $\phi$ extended to a neighbourhood of ${\c {\D}^2}$ in $\C^2
$ by $\hat \Psi$, proposition 3 reduces to the following statement.

\medskip
\noindent{\bf Proposition 3'}. {\it Let $\Psi$ be a biholomorphic
mapping from a neighbourhood $\mathcal N (\overline{\disc}^2)
\subset \C^2$ of the closed bidisc onto a subset of $X^2$. Suppose
the restriction of $\Psi$ to a neighbourhood $\mathcal N (\c{\D}^2)$
of $\c{\D}^2$ lifts to a biholomorphic mapping $\hat \Psi$ onto a
subset of $\hat G$ such that $\hat {\mathcal{P}} \circ \hat \Psi = \Psi$ on
$\mathcal N (\c {\D}^2)$. Then the mapping $\Psi$ lifts to a
biholomorphic mapping, again denoted by $\hat \Psi$, from a
neighbourhood of the closed bidisc onto a subset of $\hat G$, such
that $\hat {\mathcal{P}} \circ \hat \Psi = \Psi$ on
this neighbourhood.}

\medskip

To prove proposition 3' we have to show that for any point $p$ in
the face $\bd{\D}\times \D \; (= \overline \D^2 \setminus \c\D^2)$ of
the bidisc there is a neighbourhood $U$ of $p$ and a lift of the
mapping $\Psi \mid U$ to $\hat G$ which coincides with $\hat \Psi$
on $U \cap {\D}^2$. After rotation in the first variable we may
assume that $p \in \{1\} \times \D$.

Consider the intersections of the closed bidisc, respectively of the set
$\c{\D}^2$, with the set $[0,1]\times \overline \D$. The first intersection is equal to
$\mathfrak{C} \,= \,
[0,1] \times \overline {\D}  $, the second equals $\mathfrak{c}\,=\,
( [0,1)\times \overline \D )\;\bigcup \; ( [0,1] \times \partial \D
)$.

It will be enough to prove proposition 3'
for $\mathcal N (\overline{\disc}^2)$ replaced by a neighbourhood
of $\mathfrak{C}$ and $\mathcal N
(\c{\D}^2)$ replaced by a neighbourhood
of $\mathfrak{c}$ . Moreover, since lifting is an open property
it is enough to prove the following proposition.

\medskip
\noindent{\bf Proposition 3''}. {\it Suppose $\Psi : \mathfrak{C} \rightarrow X^2$
is a continuous mapping which is fiberwise a holomorphic immersion (of a neighbourhood of the closed
disc $\overline \D$ in $\C$ into $X^2$). Suppose $ \Psi |\mathfrak{c}$ lifts to a continuous mapping
$\hat {\Psi} : \mathfrak{c} \rightarrow \hat G$ with $\hat P\circ \hat {\Psi} = \Psi$. Then the mapping
$\Psi$ on the whole set $\mathfrak{C}$ admits a lift to $\hat G$.}

\medskip

Recall the following reformulation of the property to admit a lift to $\hat G$.

\medskip

A mapping $\Psi$ from a set $E \subset \mathfrak{C}$ into $X^2$
lifts to a mapping $\hat \Psi : E \rightarrow \hat G $ iff for each
point $(t,z) \in E$ there exists a  $G_0$-disc $d_{(t,z)}$ with
center at $\Psi (t,z)$ which represents the equivalence class $\hat
\Psi (t,z)=\hat d _{(t,z)}$ and, moreover, the equivalence classes
$\hat d _{(t,z)}$ depend continuously on $(t,z)$.

Let $\Psi : \mathfrak{C} \rightarrow X^2$ be a mapping for which the restriction
to $\mathfrak{c}$ lifts to a continuous mapping into $\hat G$. Write
$\Psi_t(\cdot) \defined \Psi(t,\cdot)$ and let $\hat {\Psi}_t(\cdot)$ be the lifted
mapping where it is defined.

The following simple lemma allows to modify the family
$\Psi _t$ to obtain a family with a stronger property of the initial
disc: Namely, one can
assume that the initial disc has small diameter and is embedded into
$G$ instead of assuming that through each of its points there is a
$G_0$-disc.

\begin{lemma}\label{Lemma 11} Under the conditions of Proposition 3'' there is a continuous family of
analytic discs $\Phi _t=\Phi (t,\cdot)$, $\Phi :\mathfrak{C}=[0,1]
\times \overline \D\ \rightarrow X^2$, which coincides for $t$ close
to $1$ with the family of the previous discs, i.e. $\Phi
(1,z)=\Psi(1,z)$ for $z \in \overline \D$ and $t$ close to $1$, and
has the following properties:
\begin{itemize}
\item [(1)] $\Phi \mid \mathfrak{c}$ lifts to a mapping $\hat \Phi :\mathfrak{c}\rightarrow
\hat G $.

\item [(2)] The lift $\hat {\Phi}_0 : \overline \D \rightarrow \hat G$ of the disc
$\Phi _0$ is embedded into $\hat i (G)$. Its projection $\Phi _0
(\overline \D)\,\, = \hat P  \circ \hat {\Phi}_0(\overline \D)$ is
an analytic disc of small diameter embedded into $G$.

\end{itemize}
\end{lemma}

\begin{proof}

We will extend the family $\Psi(t,z)$ for negative values of $t$ and
reparametrize in the parameter $t$ to obtain property (2).

The extension is constructed as follows. According to the conditions
the disc $\Psi _0 = \Psi(0,\cdot)$ lifts to a mapping  $\hat {\Psi}
_0 : \overline \D \rightarrow \hat G$.

For $t\in [-1,0]$ we define a mapping $ \hat {\Psi}_t $ as a
contraction of $ \hat {\Psi} _0$ along the radius. More precisely,
choose a small enough positive number $\sigma$ and define  $\hat
{\Psi}_t (z) \defined \hat{\Psi}_0 (\rho (t)z), \;z\in \overline
\D$, for an orientation preserving diffeomorphism $\rho : [-1,0]
\rightarrow [\sigma ,1]$.

Connect the center $\hat\Psi _0 (0)$ of the lifted disc $\hat
\Psi_0$ with a point on $\hat i(G)$ by a curve $\hat h
:[-2,-1]\rightarrow \hat G$. Associate to the curve a continuous
family of analytic discs $\hat {\Psi} _t : \overline \D \rightarrow
\hat G$, $t \in [-2,-1]$, such that the curve of centers $\hat
{\Psi}_t(0)$ coincides with $\hat h (t),\; t \in [-2,-1]$ and the
analytic disc $\hat{\Psi}_{-1}$ coincides with the previous analytic
disc $z \rightarrow \hat {\Psi}_0(\sigma z)$. If $\sigma >0$ is
small enough such a family can be found. Indeed, one can take small
analytic discs embedded into $\hat G$ with center $\hat h$.
Moreover, this family can be chosen so that $\hat {\Psi}_{-2}$ is an
embedding into $\hat i (G)$. Projecting to $X^2$ gives a family
$\Psi _t = \hat{\mathcal{P}} \circ \hat{\Psi }_t$, $t \in [-2,-1]$,
which is a continuous extension of the family ${\Psi}_t, \; t \in
[0,1]$.

The mapping $\Phi$ is obtained by changing the parameter
$t$ by an orientation preserving diffeomorphism of the interval
$[-2,1,]$ onto $[0,1]$ which is the identity near $1$.
\end{proof}

Lemma 13 below will be the key for proving proposition 3''. We will
state the Lemma after formulating the weaker lemma 12 which
considers a single analytic disc instead of a family of discs. Lemma
12 is easier to state than Lemma 13. Later we will formulate a more
elaborate version of lemma 12 which will be used in the proof of the
corollaries (see Lemmas 17 and 18 below).

\begin{lemma}\label{Lemma 12} Let $\Phi : \overline \D \rightarrow X^2$ be an
analytic disc such that its boundary lifts to $\hat G$.
Then through each point $\Phi (z), z \in \D,$ passes a $G$-disc (but maybe, not a
$G_0$-disc).

\end{lemma}

\begin{lemma}\label{Lemma 13} Let $\Phi :\mathfrak{C} \rightarrow X^2$
be a continuous family of analytic discs that satisfy conditions (1)
and (2) of Lemma 11. Then the mapping $\Phi$ lifts to a mapping
$\hat \Phi :\mathfrak{C} \rightarrow \hat G$.

\end{lemma}

Lemmas 11 and 13 imply proposition 3''. In the following sections we
will prove Lemmas 12 and 13.

\section{Neurons}

This section is based on the key observation stated in Lemma 14
below. Start with the following definition.

\begin{definition}\label{Definition 3d}
Let $\alpha$ be a piecewise smooth curve in the plane. (It may be a
mapping of a closed interval or of the circle). We call a piecewise
smooth curve $\alpha^*$ in the plane an excrescence of $\alpha$ if
$\alpha^*$ is obtained by cutting $\alpha$ at finitely many points
and pasting each time on the "right" of $\alpha$ (according to its
orientation) the punctured pellicle of a planar rooted tree. We
require that the trees are pairwise disjoint and meet $\alpha$
exactly at their roots.

Let $\sigma$ be a continuous mapping of the image of $\alpha$ into
$X^2$ which has a continuous lift $\hat \sigma$ to $\hat G$, $\hat
{\mathcal{P}} \circ \hat \sigma = \sigma$.

Suppose there is an excrescence $\alpha^*$ and extensions $\sigma
^*$ and $\hat \sigma ^*$ of $\sigma$ and $\hat \sigma $ defined on
the image of $\alpha^*$, $\hat {\mathcal{P}} \circ \hat \sigma^*=
\sigma ^*$, with the following property. There is a halo
$\overset{\;\; \; \circ \;*}{\alpha}$ for which $\hat
{\mathcal{P}}_0 \circ \overset{\;\; \; \circ \;*}{\alpha} = \hat
\sigma^* \circ {\alpha}^*$.

Then we say that $\alpha $ has an excrescence  $\alpha ^*$ with halo
$\overset{\;\; \; \circ \;*}{\alpha}$ associated to $\hat \sigma$.

\end{definition}

\begin{lemma}\label{Lemma 13a}
Let $\alpha$ be a piecewise smooth curve in the plane such that
small shifts to the right of the smooth parts do not meet the curve.
Let $\sigma$ be a continuous mapping from its image into $X^2$
which admits a lift $\hat \sigma$ to $\hat G$. Then there exists an
excrescence $\alpha^*$ with halo $\overset{\;\; \; \circ
\;*}{\alpha}$ associated to $\hat \sigma$.
\end{lemma}

\begin{proof} Let $\alpha$ be a mapping of the unit circle into
$X^2$. (For mappings of an interval the proof is the same.) Cover
the circle by a finite number of closed arcs with pairwise disjoint
interior so that on each arc one can choose a continuous family of
$\GN$-discs representing $\hat \sigma \circ \alpha$. At each common
endpoint of two of the closed arcs we obtain two equivalent
$G_0$-discs $d^-_j$ and $d^+_j$ (limits from the left, respectively
from the right of the point). Consider for each of the discontinuity
points $t_j$ a tree $T_j$ rooted at $\alpha ({t_j})$ and
corresponding to the respective pairs of equivalent $G_0$-discs by
Lemma 4. Realize the trees as pairwise disjoint subsets of the
plane, each attached to the curve on its "right" side and meeting
the curve exactly at the root. Associate to each tree $T_j$ the
structure of a dendrite with halo $\overset{\circ}m_{T_j}$ such that
$\overset{\circ}m_{T_j}$ takes the value $d^-_j$ at the initial
point and the value $d^+_j$ at the terminating point of the
punctured pellicle of the tree $T_j$. Cut the curve at each
discontinuity point and paste the punctured pellicle of the
respective tree. Denote the obtained curve by $\alpha^*$. Extend
$\sigma$ and $\hat \sigma $ by the mappings $\Phi _{T_j}$ and $\hat
\Phi _{T_j}$ (see Lemma 5) to each of the trees and hence to each
punctured pellicle and denote the extended mappings by $\sigma ^*$
and $\hat \sigma ^*$. By the choice of the dendrites the mapping
$\sigma^* \circ \alpha^*$ lifts to $\GN$. The lift is the required
halo $\overset{\;\; \; \circ \;*}{\alpha}$.

\end{proof}

Lemma 14 will be applied, in particular, to boundaries of analytic
discs. We need the following terminology. It will be convenient to
consider analytic discs up to reparametrization by conformal
mappings of simply connected planar domains to the unit disc.

\begin{definition}\label{Definition 3c}
1)(Generalized disc) Let $D$ be a relatively compact simply
connected domain in the complex plane with smooth boundary. Let
$T_j$ be a finite collection of pairwise disjoint planar trees.
Suppose the trees have pairwise different root on $\partial D$ and
meet the closure $\overline D$ of the domain exactly at the root.
Denote by $T$ the union $ \bigcup T_j$ of the trees. The set $\nu =
\overline D \bigcup T$ is called a generalized disc, the set $\nu
\setminus D$ is called the boundary of the generalized disc $\nu$
and the excrescence of $\partial D$ (traveled counterclockwise)
determined by the union of the trees is called the pellicle of the
generalized disc $\nu$ and is denoted by $m$.

\medskip

\noindent 2) (Preneurons) Suppose, moreover, that there is a
continuous mapping $\Phi :\nu \rightarrow X^2$ that is analytic on $
D$. Then the triple $(\nu,m,\Phi)$ is called a preneuron.
We will call $\Phi \circ m$ the pellicle of the preneuron.\\
Points on the circle which are not roots of attached trees are called regular points.

\medskip

\noindent 3) (Halo of a preneuron) If the pellicle $\Phi \circ m$ of
the preneuron admits a continuous lift $\overset{\circ}{m}$ to $\GN$
then the preneuron together with the mapping $\overset{\circ}{m}$ is
called a preneuron with a halo.

\medskip
\noindent 4) (Main body) The restriction of the mapping $\Phi$ to
the closure of the domain, $\Phi :\overline D \rightarrow X^2$, is
called the main body of the preneuron.

\medskip

\noindent 5) (Axon and neuron) A non-empty dendrite whose tree
consists of a single edge with leaf mapped into $G$ (or consists of
a single leaf mapped into $G$) is called an axon. A preneuron with
an axon attached is called a neuron. A halo of a neuron is a lift
$\overset{\circ}{m}$ of the mapping $\Phi \circ m$ to $\GN$ with the
additional property that the value of $\overset{\circ}{m}$ at the
leaf of the axon is a small disc embedded into $G$.

\medskip

\noindent 6) (Continuity) We will say that a family $\nu _t$ of
generalized discs depends continuously on the real parameter $t$ if
suitable parametrizations $m_t$ of their pellicles are continuous
functions in all parameters. A family of (pre)neurons $(\nu_t,m_t,
\Phi_t)$ is continuous if in addition the mapping $\Phi _t \circ
m_t$ is continuous in all parameters. For continuity of a family of
neurons with halo we have to add the condition that the mappings
$\overset{\circ}m _t$ are continuous in all parameters.

\end{definition}

\vspace{1cm}

\begin{figure}[h]
\centering
\includegraphics[scale=0.5]{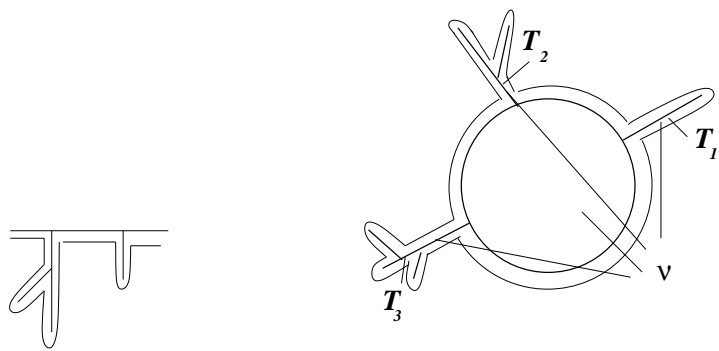}
\caption{a) An excrescence of an interval and b) a generalized disc
and a surrounding curve that approximates the pellicle}
\end{figure}

With this terminology, any analytic disc in $X^2$ is a preneuron,
but it admits the structure of a neuron only if some part of its
boundary is contained in $G$. In the latter case any boundary point
contained in $G$ can be chosen to serve a one-vertex (or degenerate)
axon. There are many ways to extend the unit disc to a generalized
disc and to give it the structure of a preneuron whose main body is
the original disc. If the generalized disc has non-empty trees
attached and $\Phi$ maps at least one leaf of certain tree into $G$
the preneuron can be given the structure of a neuron. This is always
the case if a non-empty tree of the generalized disc together with
the mapping $\Phi$ form a dendrite related to a pair of equivalent
discs according to lemma 5. Any edge of its tree that is adjacent to
a leaf may serve as the tree of an axon. Notice that the notion of
the halo of a neuron is stronger than that of the halo of a
preneuron.

The main reason for constructing neurons out of analytic discs is
the following fact: If an analytic disc is performed into the main
body of a neuron with halo then the neuron structure may be used for
obtaining $G$-discs which approximate the original disc uniformly
along compacts (see below the proof of Lemma 12; for a refinement of
this assertion see the proof of Lemma 13).

The following lemma extends Lemma 14 to preneurons.

\begin{lemma}\label{Lemma 14}
Suppose the pellicle of a preneuron $n=(\nu,m,\Phi)$, $ \Phi \circ m
\rightarrow X^2$, has a lift $\hat m$ to $\hat G$. Then there is a
neuron with halo $n^* =(\nu^*,m^*,\Phi,\overset{\circ}m^*)$ whose
generalized disc $\nu ^*$ contains $\nu$ with the following
properties. The pellicle $m^*$ of $\nu ^*$ is an excrescence of the
pellicle $m$ of $\nu$ such that the halo $\overset{\circ}m^*$ of
$m^*$ is associated to $\hat m$. The values of $\overset{\circ}m^*$
over each leaf of a tree contained in $\nu ^* \setminus \nu$ (not
only over the leaf of the axon) is a small disc embedded into $G$.
\end{lemma}

The lemma can be rephrased as follows. If the boundary of a
preneuron lifts to $\hat G$ then after further attachment of
dendrites a neuron is obtained with the following property. There is
a closed curve $\gamma :\partial{D} \rightarrow \mathcal{L}^a_0$
meeting the set of small discs contained in $G$ and such that the
curve described by the centers of the discs $\gamma (\zeta),\; \zeta
\in
\partial{D}$, coincides with the pellicle of the neuron.

\begin{proof}

Apply lemma 14 to the pellicle $m $ of the generalized disc $\nu$.
We obtain an excrescence $m^*$ which is the pellicle of a
generalized disc $\nu^*$, which is obtained from $\nu$ by attaching
further trees (either with root at the circle or with root at a tree
of $\nu$). Moreover, $m^*$ is chosen so that the mappings $\Phi$ and
$\hat \Phi$ extend to the image of $m^*$ in such a way that $\Phi
\circ m^*$ lifts to a halo $\overset{\circ}m^*$ with $\hat
{\mathcal{P}}_0 \circ \overset{\circ}m^* = \hat \Phi \circ m^* $. We
may assume that $\nu ^*$ differs from $\nu$ by at least one
non-trivial tree corresponding to a pair of equivalent discs. We
obtained a neuron $n^* =(\nu^*,m^*,\Phi,\overset{\circ}m^*)$ with
halo. The second assertion of the lemma is clear.
\end{proof}

Let $n = (\nu,m,\Phi,\overset{\circ}{m})$ be a neuron. Parametrize
the pellicle $m$ of $\nu$ by the unit circle $\partial \D$. Consider
the evaluation mapping of the halo $\overset{\circ}{m}$:
$\overset{\circ}{m}(\zeta,z) \defined
\overset{\circ}{m}(\zeta)(z),\; \zeta  \in \partial \D, \; z \in
\overline \D$. This evaluation mapping is a continuous mapping from
the the set $\partial \D \times \overline \D$ into $X^2$ which is
holomorphic on the disc fibers. (Recall that the mapping
$\overset{\circ}{m}_D$ is a continuous mapping of $\partial D $ into
the space $A^1(\D)$ of holomorphic mappings from the unit disc into
$X^2$ that have $C^1$ extension to the closed unit disc.) Let
$m(\zeta _0)$ be the tip of the axon tree of the neuron. Consider
the (image of the) disc fiber $\overset{\circ}{m}(\zeta_0)(\overline
\D)$ and the union of all (images of) circle fibers $\bigcup_{\zeta
\in
\partial \D} \overset{\circ}{m}(\zeta)(\partial \D)$. The union of
the two sets, $\kappa _n \defined \bigcup_{\zeta \in \partial \D}
\overset{\circ}{m}(\zeta)(\partial \D) \;  \bigcup
\overset{\circ}{m}(\zeta_0)(\overline \D)$ is a compact subset of
$G$ associated to the neuron $n$.

\medskip

The {\bf idea of the proof of Lemma 12} in case $X^2 = \C ^2$ is the
following (see below section 7 for details).

Let $\Phi : \overline { \D} \rightarrow \C^2$ be an analytic disc
with boundary lifting to $\hat G$. Lemma 15 produces a neuron $n$
with halo whose main body coincides with the analytic disc $\Phi$. A
neuron can be considered as a degenerate analytic disc. Mergelyan's
theorem allows uniform approximation of the neuron by a true
analytic disc ("fattening of dendrites", see below section 6).

The domain of definition of the disc is a simply connected smoothly
bounded domain $D$, whose closure contains the generalized disc of
the neuron and approximates it.

If the original neuron had a halo the approximating disc-neuron may
be given a halo. Denote the new disc-neuron with halo by $(D, m_D,
\Phi _D, \overset{\circ}{m}_D)$. Here $m_D$ just denotes the
boundary curve of the domain $D$. In other words, the disc-neuron is
an analytic disc $\Phi _D: \overline D \rightarrow X^2$ with a halo
$\overset{\circ}{m} _D:
\partial D \rightarrow \GN$. The halo defines the following (image of a) torus
$\bigcup _{\zeta \in \partial D} \overset{\circ}{m}_D
(\zeta)(\partial \D)$ consisting of the union of the boundaries of
$G_0$-discs. Call them meridians of the torus. The torus is a
compact subset of $G$ contained in a small neighbourhood of
$\kappa_n$. Approximate solutions of the Riemann-Hilbert boundary
value problem for this torus produce holomorphic discs $f_D:
\overline D \rightarrow X^2$ with boundary in a small neighbourhood
of the torus. Such discs are $G$-discs. Approximate solutions of
Riemann-Hilbert boundary value problems are constructed in
\cite{FoGl}. There is a closed arc $\Gamma \subset
\partial D$ (an arc that is close to the tip of the axon tree of $n$)
such that for $\zeta \in \Gamma$ the meridian $\overset{\circ}{m}_D
(\zeta)(\partial \D)$ bounds an analytic disc {\it of small diameter
contained in} $G$. This implies the following additional property of
approximate solutions of the Riemann-Hilbert boundary value problem.
Given any compact subset $K \subset D \bigcup \Gamma$,  after
possibly squeezing some meridians along the analytic discs bounded
by them, the value $\max_K |f_D - \Phi _D|$ is small compared to the
distance of $\kappa _n$ to the boundary of $G$.  Hence, for each
point in $\Phi _D(K)$ a small translation of the disc $f_D:\overline
D \rightarrow X^2$ produces a $G$-disc through this point. For more
detail see below section 7.

We give an argument different from that in \cite{FoGl} to construct
approximate solutions of the Riemann-Hilbert boundary value problem.
For a curve $\zeta \rightarrow (\zeta, g(\zeta)),\; \zeta \in
\partial D, $ with $g(\zeta) \in \overset{\circ}{m}_D(\zeta)(\partial \D)$ for
each $\zeta \in \partial D$, we consider the winding number around
the meridians. For the approximate solutions of the Riemann-Hilbert
boundary value problem given in \cite{FoGl} the winding number of
the boundary curve grows uncontrolled with the rate of
approximation. This fact and the hope to handle more general
situations are the reasons to choose here an argument that differs
from that in \cite{FoGl}. Namely, instead of squeezing some
meridians of the original torus, we take an open arc
$\overset{\circ}{\Gamma}$ whose closure is contained in the interior
$Int \, \Gamma$ such that $K \cap
\partial \D \subset \overset{\circ}{\Gamma}$, and approximate the mapping
$\overset{\circ}{m}_D :
\partial D \rightarrow \GN$ on $\partial D \setminus \overset{\circ}{\Gamma}$
by a continuous mapping $\overset{\circ}{M}$ from $\overline D$ into
$\G$ that is holomorphic on $D$. Moreover, $\overset{\circ}{M}$ is
chosen so that the evaluation mapping of $\overset{\circ}{M}(\zeta)$
at the point $0 \in \D$ equals $\Phi_D (\zeta)$ for each point
$\zeta \in \overline D$. Here we call a mapping from $D$ into $\G$
holomorphic if it is locally the sum of a power series with
coefficients being $\G$-discs. (The metric in $\G$ is the $C^1$-norm
of the mappings on $\overline \D$. We use the notion of holomorphic
mappings $\overset{\circ}{M}$ into $\G$ only here for the purpose of
explaining the concept. Later we will only use the evaluation
mapping $\overset{\circ}{M}(\zeta)(z),\; \zeta \in \overline D,\; z
\in \overline \D,$ of such a holomorphic mapping which is a
continuous mapping from $\overline D \times \overline \D$ that is
holomorphic on the interior of this set.) The approximating mapping
$\overset{\circ}{M}$ defines a new torus over the boundary of the
domain $D$. The part of the new torus over $\partial D \setminus
\overset{\circ}{\Gamma}$ is close to the respective part of the old
torus. Squeeze the meridians corresponding to points in $\Gamma$ as
much as needed along the analytic discs bounded by them. The thus
obtained tori are still contained in a small neighbourhood of
$\kappa_n$ . There are exact solutions of the corresponding
Riemann-Hilbert boundary value problem with winding number of the
boundary curve not depending on the rate of approximation and of
squeezing of meridians. For details see below section 7.

\medskip

The {\bf proof of lemma 13} is more subtle. Under the conditions of
Lemma 13 there is a homotopy of the disc $\Phi _1$ to an analytic
disc $\Phi _0$ where $\Phi _0$ is embedded into $G$ and lifts to
$\hat i(G)$. The homotopy consists of analytic discs $\Phi _t$ whose
boundaries lift to $\hat G$. We have to take a $G$-disc related to
$\Phi_1$ as constructed by Lemma 12 and find a $G$-disc homotopy to
an analytic disc embedded into $G$.

The key point is to obtain a continuous family $\phi_t$ of neurons
with continuously changing halo and continuously changing axons such
that for $t$ in neighbourhoods of $0$ and of $1$ the main bodies of
the $\phi_t$ coincide with the analytic discs $\Phi _t$.

Indeed, the scheme of proof of lemma 12 applies not only for an
individual neuron with halo but also for continuous families of such
neurons. This observation allows to obtain from the aforementioned
continuous family of neurons a homotopy of $G$-discs. The homotopy
of $G$-discs joins the given $G$-disc obtained in lemma 12 to a disc
embedded into $G$. The conclusion is that each point in $\Phi_1
(\D)$ is contained in the projection of $\hat G$. The existence of a
continuous lift of $\Phi_1$ to $\hat G$ follows from lemma 7 (see
below section 8 for details).

The first step towards the construction of the continuous family of
neurons $\phi_t$ (see below Lemma 19) is to convert the continuously
family of analytic discs $\Phi _t :\overline \D \rightarrow X^2$
into a {\it piecewise} continuous family of preneurons with the
following property. To each of the preneurons an axon can be
attached and the axons can be chosen continuously depending on the
parameter $t$.

The tips of the axons form a curve that is mapped into $G$. Fatten the axons continuously
depending on $t$ (see section 6 below). We obtain
a piecewise continuous family $\Psi _t$ of neurons and a fixed arc $\Gamma $
of the circle mapped into $G$ by all $\Psi _t$. More precisely, the mapping $(t,z) \rightarrow \Psi _t$
is a continuous mapping from $[0,1] \times \Gamma$ into $G$. We may assume that $1 \in \Gamma$.

The mapping $\Psi,\; \Psi(t,\zeta)\defined \Psi_t(\zeta),$
restricted the set $[0,1] \times \Gamma$ lifts to $\GN$. Indeed, any
continuous mapping $\overset{\circ}{\Psi}$ into the set of small
discs embedded into $G$ such that the center of
$\overset{\circ}{\Psi}(t,\zeta)$ equals $\Psi(t,\zeta)$ may serve.

Attaching further dendrites we associate with each of the thus
obtained neurons a new neuron $n_t$ which has already a halo. We do
it in such a way that the halo on $[0,1] \times \Gamma$ equals the
above chosen one and the family $n_t$ is piecewise continuous.

From the piecewise continuous family we get a continuous family of neurons in the following way.
Let $t_0$ be a discontinuity point of the family $n_t$. Let
 $n_{t_0}^-$, and  $n_{t_0}^+$ respectively, be the limit neurons at $t_0$ from
 the left and, from the right respectively. We show that we can attach a
 dendrite ${\bf {\mathfrak{T}}} _{t_0}$ to $n_{t_0}^+$ at a point of $\Gamma $ in such a
 way that $n_{t_0}^+ \cup  {\bf {\mathfrak{T}} _{t_0}}$ has a halo and
there is a homotopy of neurons with halo joining $n_{t_0}^-$ with
$n_{t_0}^+ \cup  {\bf {\mathfrak{T}} _{t_0}}$. A continuously
changing copy of the dendrite ${\bf {\mathfrak{T}}} _{t_0}$ will be
attached to all neurons $n_t$ with $t>t_0$. We proceed in this way
with each discontinuity point of the family $n_t$.

The most subtle part of the aforementioned proof is the construction
of the homotopy joining $n_{t_0}^-$ with $n_{t_0}^+ \cup  {\bf
{\mathfrak{T}} _{t_0}}$ (see below Lemma 20). This construction will
be a procedure which preserves the main body (which is common for
$n_{t_0}^-$ and $n_{t_0}^+$ ) and can be considered as continuously
"peeling off the halo of the left neuron $n^-_{t_0}$" starting at a
point in $ \Gamma $ and letting "grow the halo of the right neuron
$n^+_{t_0}$ on the peeled places and symmetrically on the inside of
the removed peel".

\section{ Partial fattening of dendrites.}

Here we describe in detail the procedure of "fattening dendrites"
which is used in the proof of Lemmas 12 and 13. In the proof of
lemma 12 the procedure is applied to a single neuron. In the
proof of lemma 13 it is applied to a family of neurons. We will
describe the version for families.

Consider a single generalized disc $\nu = \overline \D \cup \bigcup T_j $.
For each tree $T_j$ we consider a connected open (in the topology
induced on $T_j$ by $\C$) subset $S_j \subset T_j$ which contains
the root of $T_j$. The closure $\overline S_j$ of $S_j$ is again a
tree with root coinciding with that of $T_j$. Each set $\overline
S_j$ contains together with each point the path on $T_j$ connecting
it with the root of $T_j$. A rooted tree $\overline S_j$ obtained in
this way is called a subtree of $T_j$.

Any connected component of $T_j \setminus S_j$ is also a tree (if
the set is not empty). A vertex of such a component may belong to
$\overline S_j$. Since $T_j$ is a tree there is exactly one such
point in each connected component. (This point may be a multiple
vertex.) With this point chosen as root the connected component
becomes a rooted tree. Note that a connected component of $T_j
\setminus S_j$ may consist of several trees adjacent to this root.

Provide a "cutting of trees" : replace each tree $T_j$ by $\overline
S_j$. Denote by $S$ the union of trees $\bigcup \overline S_j$ and
consider the generalized disc  $\nu _S = \overline \D \cup S$. For a
positive number $\tau _0$ we associate to $\nu _S$ a family $E
^{\tau} _S,\; \tau \in (0,\tau _0,]$, of bounded smoothly bounded
simply connected domains with the following properties.

\begin{itemize}
\item[(1)] The sets $E ^{\tau} _S \setminus \D,\; \tau \in (0,\tau _0]$, are contained in a small
neighbourhood of S (i.e. $ E ^{\tau} _S \setminus  \D$ are
fattenings of $S$).

\item[(2)] For each $\tau \in (0,\tau _0]$ the set $E^{\tau}_S$ contains $\D \bigcup \bigcup S_j$.
Moreover for each $\tau$ and each $j$ all leaves of $\overline
S_j$ are on the boundary of $E^{\tau}_S$ and $E^{\tau} _S$ does not
intersect $\bigcup (T_j \setminus S_j)$.
\item[(3)] The family decreases, i.e. $E^{\tau _1} _S \subset E^{\tau _2}
_S$ for $0 < \tau _1 < \tau _2 \leq \tau _0$. Moreover, the family
is continuous and converges to $\nu _S$ for $\tau \rightarrow 0$. We
put $E^0 _S \defined  \nu _S ( = \lim_{\tau \rightarrow 0} E^{\tau} _S )$.
\end{itemize}

Consider the set $\nu ^{\tau} \defined \overline { E^{\tau} _S} \cup
\bigcup T_j$ for $\tau \in [0,\tau _0]$. Note that $\nu ^0 = \nu$.
The $\nu ^{\tau}$ are generalized discs. The trees of $\nu ^{\tau}$
correspond to the connected components of $T_j \setminus S_j$.

\begin{figure}[h]
\centering
\includegraphics[scale=0.75]{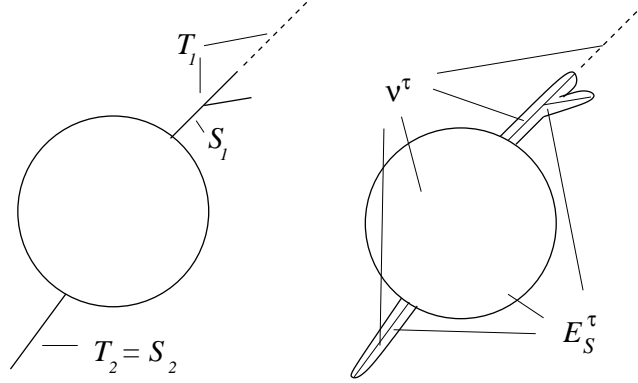}
\caption {Partial fattenings of trees of a generalized disc}
\end{figure}

The described procedure is a "partial fattening of trees". The sets
$E^{\tau}_S \setminus \overline \D$ are the fattenings of $S$. We
always assume that the connected components $E_{S,j}^{\tau}$ of
$E^{\tau}_S \setminus \overline \D$ are in a one-to-one
correspondence with the trees $S_j$.

Note that for a continuous family  $\nu _t,\; t \in [0,\tau _0],$  of generalized
discs and continuous families of unions of subtrees $\bigcup
({\overline S}_j)_t$ of $\bigcup (T_j)_t$
the "partial fattening of trees" can be arranged continuously depending on the
parameter $t$. In other words, it can be made so that it leads to a
family $\nu _t ^{\tau}$ which is continuous in both parameters
$t$ and $\tau$.

In the following lemma we consider neurons. The lemma extends the
procedure of partial fattening of trees to a "partial fattening of
dendrites". For each $t$ the generalized disc is the union of the
closed unit disc with attached trees.

\begin{lemma}\label{Lemma 15} Suppose $n_t = (\nu _t, m_t, \Phi _t), t \in [0,1],$
is a continuous family of neurons. Let $S _t = \bigcup
\overline{(S_j)_t}$ be a continuous family of unions of subtrees of
the trees of their generalized discs $T_t = \bigcup (T_j)_t$. Let
$\nu _t ^{\tau} = \overline { E^{\tau} _t} \cup \bigcup (T_j)_t ,\;
t \in [0,1], \; \tau \in [0,\tau _0],$ be a continuous family of
generalized discs obtained from the $\nu _t$ by fattening the trees
constituting $S_t$. Then there is a continuous family of mappings
$\Phi _t ^{\tau} : \nu _t ^{\tau} \rightarrow X^2, \; t \in [0,1],\;
\tau \in [0,\tau _0]$, that are holomorphic on the interior $E_t
^{\tau}$ of $\nu _t ^{\tau}$ such that $\Phi _t ^0 = \Phi_t$. If the
restriction of $\Phi$ to $\bigcup_{t \in [0,1]} \{t\} \times (\nu _t
\setminus \D)$ has a lift $\hat \Phi$ to $\hat G$ then the
restrictions of $\Phi ^{\tau}$ to $\bigcup_{t \in [0,1]} \{t\}
\times \partial E_t^{\tau}$,  $\tau \in [0,\tau _0],$ have lifts
$\hat \Phi ^{\tau}$ depending continuously on $\tau$.

Let $m_{j,t}$ be the punctured pellicle of $(T_j)_t$ and
$m^{\tau}_{j,t}$ the arc of the pellicle of $\nu_t^{\tau}$ whose
image is contained in  $\partial E_{S,j}^{\tau} \bigcup (T_j
\setminus S_j)$. If for some $j$ all dendrites $(\mathbf{T}_j)_t =
((T_j)_t, m_{t,j}, \Phi _t |(T_j)_t), \; t \in [0,1],$ have
punctured halo $\overset{\circ}{m}_{j,t}$ associated to $\hat \Phi$
that depends continuously on $t$ then (possibly after decreasing
$\tau _0$) also the curves $m^{\tau}_{j,t}$ have a halo
$\overset{\circ}{m}_{j,t}^{\tau}$ associated to $\hat \Phi ^{\tau}$
that depends continuously on $t$ and $ \tau$ and converges to
$\overset{\circ}{m}_{j,t}$ for $\tau \rightarrow 0$.

\end{lemma}

\begin{proof} {\bf In case $X^2 = \C ^2$} the first assertion of the lemma is a
standard approximation lemma for the coordinate functions of the
mappings $\Phi _t$. Let $E_t ^{\tau}$ be the generalized discs
obtained by fattening the trees constituting $S$. The idea of proof
of this approximation lemma is to extend for each $t$ the function
$\Phi _t$ to a continuous function in the whole plane $\C$ and to
smoothen the extension (in dependence on $\tau$) in such a way that
the $\overline{\partial}$-derivative is small near points of $(\nu
_S)_t$ and vanishes on a big compact subset of $\D$. For details we
refer to the book  \cite{Ru} (see the proof of theorem 20.5). The
construction can be made continuously depending on $t$ and $\tau$.
The approximating function $\Phi _t^{\tau}$ is obtained by
correcting the extended and smoothened function by the solution of a
$\overline{\partial}$-equation related to the interior of $ \nu _t
^{\tau}$.

Prove the second assertion for the case $X^2 = \C^2$. For suitable
parametrizations of $m_{t,j}$ and $m_{t,j}^{\tau}$ by  $s \in [0,1]$
we have uniform convergence $m_{t,j}^{\tau} \rightarrow m_{t,j}$ for
$\tau \rightarrow 0$, hence the arc $ \Phi _{t,j} ^{\tau} \circ
m_{t,j}^{\tau}$ in $X^2$ converges to the arc $ \Phi_{t,j} \circ
m_{t,j}$ for $\tau \rightarrow 0$. It remains to make for $s \in
[0,1]$ and small $\tau$ the following choice for
$\overset{\circ}{m}_{j,t}^{\tau}$. Take
the parallel translation in $\C ^2$ of the $G_0$-disc
$\overset{\circ}{m}_{t,j}(s) $ for which the center equals $\Phi
^{\tau} _{t,j} \circ m_{t,j}^{\tau}(s)$.

\medskip

{\bf For general Stein surfaces $X^2$} we consider a holomorphic
embedding $\mathfrak{F}:X^2 \rightarrow \C ^4$. The approximation of
$\mathfrak{F} \circ \Phi _t$ works as in the proof of the first
assertion for $\C ^2$. Given the halo $\overset{\circ}{\mathfrak{F}
\circ {m}_{t,j}}$ on $\mathfrak{F} \circ {m}_{t,j}$, the halo on the
approximating arcs in $\C ^4$ can be chosen by using small
translations. It remains to compose all constructed  mappings (they
all have image in a small tubular neighbourhood of
$\mathfrak{F}(X^2)$)  with a holomorphic projection from the tubular
neighbourhood onto $\mathfrak{F}(X^2)$. The assertions of the Lemma
are proved in the case of general Stein surfaces.

\end{proof}

\section{Proof of lemma 12}

The proof of Lemma 12 is based on the following approximation lemmas
which will be needed also in section 11 below. Let $D$ be a bounded,
smoothly bounded simply connected domain in the complex plane and
let $\Gamma \subset \partial D$ be an arc. Put $\mathcal{S}
_{\partial D } \defined (\overline D \times \{0\})\; \bigcup\;
((\partial D)  \times \overline \D)$. Notice that suitable
neighbourhoods of $\mathcal{S} _{\partial D }$ are usually called
Hartogs figures. In other words,  $\mathcal{S} _{\partial D}$ is the
core of Hartogs figures. Denote the compact subset $(\overline
{\partial D} \times
\partial \D )\; \bigcup \; ( {\Gamma} \times \overline \D)$ of
$\mathcal{S} _{\partial D}$ by $\mathcal{Q}_{\Gamma}$.

Recall that for defining a metric on $X^2$ we fixed a holomorphic
embedding of $X^2$ into $\mathbb{C}^4$ and pulled back the Euclidean
metric. $\varepsilon$-approximation of mappings into $X^2$ refers to
this metric. Note that the second part of Lemma 17 below concerns
continuous families of mappings and is needed in the proof of Lemma
13.

Denote by $A_{X^2}(D \times \D)$ the space of continuous mappings
from $\overline D \times \overline \D$ into $X^2$ that are
holomorphic on the interior $D \times \D$.

\begin{lemma}\label{Lemma 16}
Let $J_D: \mathcal{S} _{\partial D} \rightarrow X^2$ be a continuous
mapping that is analytic on $D\times \{0\}$ and fiberwise analytic
on $\partial D \times \overline \D$. Let $\Gamma \subset
\partial D$ be a closed arc.

Then for each positive number $\varepsilon$ and each neighbourhood
$V$ of $J_D(\mathcal{S} _{\partial D})$ in $X^2$ there exists a
mapping $\H \in A_{X^2}(D \times \D)$, such that

\begin{itemize}
\item [(1)] $\H |\overline D \times \{0\} = J_D | \overline D \times
\{0\}$,

\item [(2)] $\H (\partial D \times \partial \D)$ is contained in an
$\varepsilon$-neighbourhood of $J_D(\mathcal{Q}_{\Gamma})$.

\item [(3)] the image of $\H$ is contained in $V$, moreover,
for each compact subset $K$ of $D \bigcup \Gamma$ the
mapping $\H$ can be chosen so that for each $\zeta \in K$ the whole
fiber $\H (\{\zeta\} \times {\overline \D} )$ is contained in an
$\varepsilon$-neighbourhood of $\Phi _D (\zeta)$.

\end{itemize}

Suppose $D_t$, $t \in [0,1]$, is a continuous family of simply
connected bounded and smoothly bounded planar domains. Let
$\mathfrak{A}_t$ be continuously changing closed arcs,
$\mathfrak{A}_t \subset \partial D_t$. Let further $K_t, \; t \in
[0,1],$ be a family of compact subsets of $D_t \cup \mathfrak{A}_t$
depending continuously on the parameter $t$ (hence $\bigcup _{t \in
[0,1]} \{t\} \times K_t$ is a compact subset of $\mathbb{R} \times
\C$). Consider the continuously changing family of sets $\mathcal{S}
_{\partial D _t}$ and $\mathcal{Q}^t \defined (\overline {\partial D
_t} \times
\partial \D )\; \bigcup \; ( {\mathfrak{A}_t} \times \overline \D)$.

Suppose $J^t_{D_t}: \mathcal{S} _{\partial D _t} \rightarrow X^2
,\;t \in [0,1]$, is a continuous family of mappings, each of it
being analytic on all analytic discs contained in $\mathcal{S}
_{\partial D _t} $.

Then for any number $\varepsilon >0$ there exists a continuous
family of mappings $\H _t \in A_{X^2}(D _t \times \D)$, such that
each $\H _t,\; t \in [0,1],$ satisfies conditions (1),(2) and (3)
above with respect to the objects specified for the number $t$.

\end{lemma}

Fix $K$. Let  $\overset{\circ}{\Gamma}$ be as in section 5 an open
arc, $\overset{\circ}{\Gamma} \Subset Int\, \Gamma$, $K \subset \D
\bigcup \overset{\circ}{\Gamma}$. Denote the set $(\overline D
\times \{0\}) \; \bigcup \; (\partial D \setminus
\overset{\circ}{\Gamma} \times \overline \D)$ by $S_{(\partial D
\setminus \overset{\circ}{\Gamma})}$.
 The proof of Lemma 17
is based on the following variant of the Weierstra{\ss} approximation
theorem for the arc $\partial D \setminus \overset{\circ}{\Gamma}$.

\begin{lemma}\label{Lemma 16a}
For any positive number $\varepsilon$ and any neighbourhood $V$ of
$J_D(\mathcal{S} _{\partial D})$ there exists a neighbourhood $U$ of
$S_{(\partial D \setminus \overset{\circ}{\Gamma})}$ in $\overline D
\times \overline \D$ and a continuous mapping $\mathfrak{H}:U
\rightarrow V \subset X^2$ that is holomorphic on the interior
$Int\, U$ of $U$ such that $\mathfrak{H}|\overline D \times \{0\} =
J_D | \overline D \times \{0\}$ for $\zeta \in \overline D$ and
$\mathfrak{H}$ is uniformly $\varepsilon$-close to $J_D$ on
$(\partial D \setminus \overset{\circ}{\Gamma}) \times \overline
\D$.

\end{lemma}

\begin{proof}
In case $X^2$ is different from $\C^2$ we compose the mapping
${J}_D$ with the holomorphic embedding $\mathfrak{F}$ of $X^2$ into
$\C^4$. Denote the composition by $\mathbf{J}_D$. The target space
for this mappings is $\C^4$. In case $X^2 = \C^2$ the target space
was $\C^2$ from the beginning. For unifying notation we use the fat
letter $\mathbf{J}_D$ for the mapping $J_D$ in this case as well. So
in any case $\mathbf{J}_D$ is a mapping into some $\C^n$ (either
$n=2$ or $n=4$).

Notice that for $r \in (0,1),\; r \rightarrow 1,$ the mappings
$\mathbf{J} _{D,r},\; \mathbf{J} _{D,r}(\zeta, z)\defined \mathbf{J}
_D(\zeta, rz), \; \zeta  \in \bd D \setminus
\overset{\circ}{\Gamma}, z \in \overline \D$, converge uniformly to
$\mathbf{J}_D (\zeta,z),\; \zeta  \in \bd D \setminus
\overset{\circ}{\Gamma}, z \in \overline \D$.

Write the mapping  $\mathbf{J}_D | (\partial D \setminus \overset{\circ}{\Gamma}) \times \overline \D$
in form of power series:

\begin{equation*}
 \sum_{k=0}^{\infty} a_k(\zeta) z^k,\;\; \zeta  \in \partial D \setminus \overset{\circ}{\Gamma},\;\; z \in \overline {\D}.
\end{equation*}

Choose a number $r<1$ sufficiently close to $1$
and a big enough number $N$ so that the mapping

\begin{equation*}
\mathbf{J}_{D,r,N}(\zeta,z) \defined  \sum_{k=0}^{N} a_k(\zeta)r^k
z^k, \;\;\zeta  \in \partial D \setminus
\overset{\circ}{\Gamma},\;\; z \in {\overline \D},
\end{equation*}

\noindent approximates the mapping $\mathbf{J}_D$ sufficiently well
on $(\partial D \setminus \overset{\circ}{\Gamma}) \times \overline
\D$. Note that both mappings, $\mathbf{J}_D$ and
$\mathbf{J}_{D,r,N}$ coincide on $\partial D \setminus
\overset{\circ}{\Gamma } \times \{0\} $ with  $\mathbf{J}_D$.
Approximate each of the coefficients $a_k(\zeta),\; k=1,...,N$,
uniformly for $\zeta \in
\partial D \setminus \overset{\circ}{\Gamma}$ by holomorphic
mappings from a neighbourhood of $\overline D$ to $\C^4$. We obtain
a continuous mapping $\mathcal{I}$ from $\overline D \times
\overline \D$ into $\C ^n$ which is holomorphic on $D \times \D$,
approximates the mapping $\mathbf{J}_D$ uniformly on $(\partial D
\setminus \overset{\circ}{\Gamma}) \,\times\ \overline {\D}$ and
coincides with $\mathbf{J}_D$ on $\overline D \times \{0\}$.

Being uniformly close to $\mathbf{J}_D$ on $S_{(\partial D \setminus
\overset{\circ}{\Gamma})}$ the mapping $\mathcal{I}$ maps a
neighbourhood $U$ of this set (in $\overline D \times \overline \D$)
into a small tubular neighbourhood of $\mathfrak{F}(X^2)$. (Recall
that $\mathbf{J}_D (S_{(\partial D \setminus
\overset{\circ}{\Gamma})}) \subset \mathfrak{F}(X^2)$.) Consider the
composition $\mathfrak{Pr} \circ \mathcal{I} $ of the mapping
$\mathcal{I}$ with a holomorphic projection $\mathfrak{Pr}$ of a
tubular neighbourhood of $\mathfrak{F}(X^2)$ onto
$\mathfrak{F}(X^2)$ and apply to it the inverse of $\mathfrak{F}$ we
obtain a holomorphic mapping $\mathfrak{H}$ from $U$ into $X^2$ that
approximates ${J}_D$ on $(\partial D \setminus
\overset{\circ}{\Gamma}) \,\times\ \overline {\D}$. If $U$ is chosen
small enough depending on $V$ the image of $\mathfrak{H}$ is
contained in $V$.

\end{proof}

\noindent {\it Proof of Lemma 17.}
Notice that for each $\zeta \in
\partial D \setminus {\overset{\circ}{\Gamma}} $ the set $U$ of
Lemma 18 contains the fiber $\{\zeta\} \times \overline \D$. For
$\zeta \in {\overset{\circ}{\Gamma}} $ the set $U$ may not contain
the respective fibers but it contains a small neighbourhood of
${\overset{\circ}{\Gamma}} \times \{0\}$.
We want to shrink the fibers over points in $\Gamma$ suitably. Take
a smooth positive function $\rho$ on $\bd D$ that equals $1$ outside
$\Gamma$, does not exceed $1$ everywhere on $\bd D$ and is as small
as needed in a neighbourhood of the closure of
$\overset{\circ}{\Gamma}$.

Consider an analytic function $g$ on $D$ with boundary values having
absolute value $\rho$. The function $g$ is smooth up to the boundary
if $\rho$ is smooth. (Recall that $D$ has smooth boundary.)
Moreover, on the compact subset $K$ of $D \bigcup
\overset{\circ}{\Gamma}$ the absolute value $|g|$ of the function
does not exceed a small constant depending on the compact set $K$
and the function $\rho$ and tending to $0$ if the maximum of the
function $\rho$ on $\overset{\circ}{\Gamma}$ tends to $0$ . This is
a consequence of an estimate of the harmonic measure of
$\overset{\circ}{\Gamma}$ on $K$.

Define the mapping $\Upsilon ^g,\; \Upsilon ^g (\zeta,z)
\defined (\zeta, g(\zeta)z)$ of the closed bidisc $\overline D
\times \overline \D$ onto $U^g$, $U_g
\defined \{(\zeta, z) \in \overline D \times \overline \D: |z| \le
|g(\zeta)|\}$. With a suitable choice of $\rho$ for each fixed $z
\in \overline \D$ the distance $|\Upsilon ^g (\zeta,z) - (\zeta,0)|$
is as close as needed uniformly for $\zeta \in K$.

Increasing the compact subset $K$ of $D \bigcup
\overset{\circ}{\Gamma}$ we may assume that each point $\zeta$
outside the compact $K$ is as close as needed to $\partial D
\setminus \overset{\circ}{\Gamma}$. Therefore the choice of the
function $\rho$ can be made in such a way that the set $U_g $ is
contained in the small neighbourhood $U$ of $\mathcal{S}_{\bd D
\setminus \overset{\circ}{\Gamma}}$ in $\overline D \times \overline
\D$.

Let $\H$ be the composition of the mapping $\mathfrak{H}$ with the
mapping $\Upsilon ^g,\; \Upsilon ^g (\zeta,z) \defined (\zeta,
g(\zeta)z)$ of the closed bidisc $\overline D \times \overline \D$
onto $U^g$, $\H = \mathfrak{H} \circ \Upsilon ^g $. The mapping $\H$
has the required properties.

Indeed, since $\rho$ has absolute value $1$ on $\partial D \setminus
{\Gamma}$ and absolute value not exceeding $1$ on $\Gamma \setminus
\overset{\circ}{\Gamma}$ the set $\H (\partial D \setminus
\overset{\circ}{\Gamma} \times \partial \D)$ is contained in a small
neighbourhood of $J_D (\mathcal{Q}_{\Gamma})$. (See Lemma 18 for the
properties of $\mathfrak{H}$ and use the fact that
$\mathcal{Q}_{\Gamma} \supset(\partial D \setminus
\overset{\circ}{\Gamma} \times
\partial \D)\; \bigcup \; (\Gamma \setminus \overset{\circ}{\Gamma} \times
\overline \D$).) If $\rho$ is small enough on
$\overset{\circ}{\Gamma}$ then also $\H (\overset{\circ}{\Gamma}
\times \partial \D)$ is contained in a small neighbourhood of $J_D
(\mathcal{Q}_{\Gamma})$.

Property (3) is a consequence of the properties of $\Upsilon ^g$.

The proof of the respective assertion for continuous families of
mappings $J^t_{D_t}$ is straightforward.

Lemma 17 is proved. \hfill $\square$

\medskip

\noindent {\it Proof of Lemma 12}. Let $\Phi :\overline \D
\rightarrow X^2$ be an analytic disc whose boundary lifts to a
mapping  $\hat{\Phi} : \partial \D \rightarrow  \hat G$. Lemma 15
produces a neuron $n=(\nu,m,\Phi)$ which has halo
$\overset{\circ}{m}$ associated to $\hat \Phi $ and has the disc as
main body. Apply Lemma 16 ("fattening of dendrites") for the single
neuron $n$, its halo and the set of all trees of its generalized
disc $\nu$, so that we obtain a true analytic disc with halo $(D,
m_D, \Phi _D , \overset{\circ}{m}_D)$. We assume that $\Phi_D$ is an
$\varepsilon$-approximation of $\Phi$ and $\overset{\circ}{m}_D$ is
an $\varepsilon$-approximation of $\overset{\circ}{m}$. The
evaluation mapping of the halo $\overset{\circ}{m}_D $ defines a
continuous mapping from the set $\partial D \times \overline \D $
into $X^2$ which is fiberwise holomorphic. Moreover,
$\overset{\circ}{m}_D(\zeta) (0) = \Phi_D (\zeta)$ for all $\zeta
\in \partial D$. Thus, $\Phi_D: \overline D \rightarrow X^2$ and
$\overset{\circ}{m}_D:\partial D \times \overline \D \rightarrow X^2
$ define a continuous mapping $J_D$ from the set $\mathcal{S}
_{\partial D } = (\overline D \times \{0\})\; \bigcup\; ((\partial
D)  \times \overline \D)$ into $X^2$.

Let $  \Phi _D: \Gamma \rightarrow X^2$, $\Gamma \subset \partial D
$,  be a closed arc of the pellicle of $D$ that is close enough to
the tip of the axon tree of the original neuron. Then for the subset
$\mathcal{Q}_{\Gamma}=(\overline {\partial D} \times
\partial \D )\; \bigcup \; ( {\Gamma} \times \overline \D)$ of
$\mathcal{S} _{\partial D}$  the set $J_D (\mathcal{Q}_{\Gamma})$ is
contained in a $2 \varepsilon$-neighbourhood of $\kappa _n \subset
G$ (see the definition of $\kappa _n$ after the proof of Lemma 15).
An application of Lemma 17 with the same number $\varepsilon $ and
with a compact subset $K$ of $\D \bigcup \overset {\circ}{\Gamma}$
provides a mapping $\H \in A_{X^2}(D \times \D)$, such that $\H
(\partial D \times
\partial \D)$ is contained in an $\varepsilon$-neighbourhood of
$J_D(Q_{\Gamma})$ and for each fixed $\zeta \in K$ the fiber $\H
(\{\zeta\} \times \overline \D)$ is $\varepsilon$-close to $\Phi _D
(\zeta)$ on $K$.


For each $z \in \partial \D$ the disc $f^z (\zeta) = \H(\zeta,z),\;
\zeta \in \overline D,$ has its boundary in a $3
\varepsilon$-neighbourhood of $\kappa _n \subset G$. The family
$f^{rz}, r \in [0,1],$ provides a homotopy joining $\Phi_D(\cdot) =
J_D(\cdot,0)$ and $f^z$. If $\H$ is chosen to satisfy (3) for given
$K \subset D \cup \Gamma$ then $\max _K |\Phi _D -f ^{rz}| <
\varepsilon $ for each $r \in [0,1]$. Choose the point $z \in
\partial \D$. An $\varepsilon$-approximation of $f^z$ provides an immersed
analytic disc, hence a $G$-disc provided $\varepsilon$ is small.

In case $X^2 = \mathbb{C}^2$ a suitable translation of the disc
passes through $\Phi (p)$ and has boundary contained in a $5
\varepsilon$-neighbourhood of $\kappa _n$.

In the case of general Stein manifolds $X^2$ translations can be
replaced by diffeomorphisms close to the identity from a suitable
relatively compact subset of $X^2$ onto another subset of $X^2$.
Such diffeomophisms are defined as compositions of the holomorphic
embedding $\mathfrak{F}$ of $X^2$ into $\C ^4$, a small translation
in $\C ^4$, a holomorphic projection of a tubular neighbourhood of
$\mathfrak{F}(X^2)$ to $\mathfrak{F}(X^2)$ and the inverse of the
mapping $\mathfrak{F}$.

We proved that through each point of $\Phi_D (K)$ passes a $G$-disc.
Given $\zeta \in \D$ the compact set $K$ can be chosen to contain
$\zeta$. Lemma 12 is proved. \hfill $\square$

\section{A piecewise continuous family of neurons with continuously changing axon}

This paragraph is a preparation for the proof of Lemma 13.

Let $\Phi _t : \overline \D \rightarrow X^2, \, t \in [0,1]$,  be a
continuous family of analytic discs enjoying properties (1) and (2)
of Lemma 11. The following lemma allows a further improvement of the
properties of the family of analytic discs without changing the
discs $\Phi _0$ and $\Phi _1$.

\begin{lemma}\label{Lemma 17}

There is a continuous family of analytic discs $\Psi _t : \overline
\D \rightarrow X^2,\; t \in [0,1],$ coinciding with the previous
family $\Phi_t$ for $t$ close to $0$ and close to $1$ such that
condition (1) and (2) of Lemma 11 hold and the following additional
condition is satisfied.

The curve  $\alpha (t) = (t,1)$, $t \in [0,1],$ in $[0,1] \times \bd
\D \subset \mathfrak{c}$ has the following property: the mapping $
{\Psi}_t (\alpha (t)), \;t \in [0,1]$, admits a lift
$\overset{\circ}{\alpha}$ to $\GN$ such that $\hat{\mathcal{P}}_0
\circ \overset{\circ}{\alpha} = \hat \psi _t (\alpha (t))$.

\end{lemma}

\begin{proof}
Consider the mapping $\Phi (t,z) \defined \Phi_t(z), \; t \in
[0,1],\; z \in \overline \D$ with values in $X^2$. By the condition
(1) of Lemma 11 the restriction of this mapping to $[0,1] \times
\partial \D$ lifts to $\hat G$, hence the mapping $\Phi \circ \alpha(t), \; t \in [0,1], $
lifts to $\hat G$. The curve $\alpha$ is contained in the cylinder
$[0,1] \times \partial \D$. It can therefore be considered as a
planar curve and Lemma 14 applies. It will be convenient to realize
the excrescence of $\alpha$ in a slightly different way. Namely,
consider a tree and its punctured pellicle which participate in the
construction of the excrescence of $\alpha$ in the cylinder. Let the
root of the considered tree be the point $(t_i,1)$ of the cylinder.
We may assume that all points $t_i$ are contained in the open
interval $(0,1)$. We take another realization of the tree and its
pellicle, namely, we consider a tree $T_i$ in the complex plane with
root at the point $1$ that meets the closed disc $\overline \D$
exactly at the root and which is a homeomorphic copy of the tree in
the cylinder. Call the product of the one-point set $\{t_i\}$ with
the punctured pellicle of the tree $T_i \subset \C$ the punctured
pellicle of $\{t_i\} \times T_i$. Cut $\alpha$ at the point
$(t_i,1)$ and paste the punctured pellicle of the tree $\{t_i\}
\times T_i$. Doing this with all trees we obtain the realization of
the excrescence $\alpha^*$ we will work with.

The trees $T_i$ define a  piecewise continuous family of generalized
discs $\nu _t, t \in [0,1],$ given by the relation $\nu _t
\defined \overline \D$, if $t$ is not equal to one of the $t_j$, and
$\nu _t
\defined \overline \D \bigcup T_j$, if $t = t_j$. The new
curve $\alpha ^* $ has values in $\bigcup _{t \in [0,1]} \{t\}
\times \nu _t$. By Lemma 14 there are continuous extensions of the
mappings $\Phi$ and $\hat \Phi$ to the image of $\alpha^*$ such that
the curve $\Phi \circ \alpha ^*$ has a lift to $\GN$ that is
associated to $\hat \Phi$. Take a $C^0$-small deformation of the
curve $\alpha ^*$ which fixes the punctured pellicles of the trees
and provides small changes of the original part $\alpha$ of the
curve $\alpha^*$ so that the image of the deformation of the part
$\alpha$ of $\alpha^*$ is the union of finitely many vertical
segments of the form $I_z \times \{z\}$ for an interval $I_z \subset
[0,1]$ and a point $z \in
\partial \D $, and finitely many horizontal arcs of the form $\{t_j\}
\times \beta _j$ for one of the aforementioned points $t_j \in
[0,1]$ and an arc $\beta_j$  in the unit circle. We may assume that
the perturbed curve coincides with the previous one near the points
$(0,1)$ and $(1,1)$. Denote the approximating curve again by $\alpha
^*$. Still, $\Phi \circ \alpha ^*$ has a lift to $\GN$ that is
associated to $\hat \Phi$.

Consider the piecewise continuous family of generalized discs $\nu
_t, \; t \in [0,1],$ that was defined above. Notice that the image
of $\alpha ^*$ is contained in $\bigcup _{t \in [0,1]} \{t\} \times
(\nu _t \setminus \D)$ and the mappings $\Phi$
and $\hat{\Phi}$ extend continuously to the union $\{t\} \times (\nu
_t \setminus \D)$ of the boundaries of the generalized discs.

Replace the family of generalized discs $\nu _t$ by a {\it
continuous} family of generalized discs $\nu ^*_t$ in the following
way. Choose small disjoint intervals $I_j \subset (0,1)$ around
$t_j$ and define a continuous family of trees $T(t), \; t \in
[0,1],$ with root $1$ such that $T(t_j) = T_j$ and $T(t)$ is equal
to a one point (degenerate) tree for $t$ close to the endpoints of
the $I_j$ and outside the $I_j$. This is possible since each rooted
tree is contractible to its root. Put $\nu ^*_t \defined \nu _t \cup
T(t)$.

The intervals and the contractions of the trees can be chosen in
such a way that the mappings $\Phi $ and $\hat \Phi$ extend
continuously to $\bigcup _{t \in [0,1]} \{t\} \times (\nu ^*_t
\setminus \D) $. Denote the extended mappings again by $\Phi$ and
$\hat \Phi$.

Lemma 16 provides fattenings of the dendrites $\mathbf{T}_t$
depending continuously on the parameter $t$. This yields a
continuous family of simply connected domains $D_t, \; t \in [0,1],$
and a continuous mapping $\psi : \bigcup _{t \in [0,1]} \{t\} \times
\overline {D_t} \rightarrow X^2$ which is holomorphic on each $\{t\}
\times D_t$, approximates $\Phi$ uniformly on $\bigcup \{t\} \times
\nu _t$ and coincides with $\Phi$ for values of $t$ close to $0$ and
close to $1$. Moreover, the restriction of the mapping $\psi$ to the
set $\bigcup _{t \in [0,1)} \{t\} \times \overline D_t \; \bigcup \,
\bigcup _{t \in [0,1]} \{t\} \times \partial D_t) $ lifts to a
mapping $\hat \psi$ into $\hat G$ which coincides with $\hat \Phi$
for $t$ close to $0$ and close to $1$.

Deform the arcs of $\alpha ^*$ contained in the set $t = t_j$ into
arcs that are $C^0$-close to the previous ones and run along the
boundary $\{t_j\} \times \partial D _{t_j}$. Denote the deformed
curve by $\alpha ^0$.

\begin{figure}[h]
\centering
\includegraphics[scale=0.4]{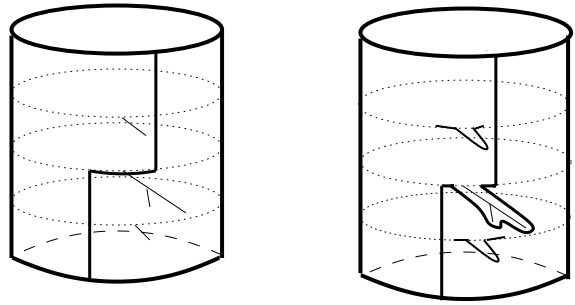}
\caption{Fattening of trees of a family and deformation of the
excrescence}
\end{figure}

Provide a further deformation of the curve so that its
$t$-coordinate is strictly increasing. Parametrize the thus obtained
curve by the $t$-coordinate of its image and denote it again by
$\alpha ^0$. The mapping $\psi \circ \alpha^0$ admits a lift to
$\GN$ which is associated to $\hat \psi$.

Choose a continuous family of conformal mappings $\varphi _t: \D
\rightarrow D_t$ (which extend to a continuous family of
homeomorphisms between the closed unit disc and the closures of the
domains) that map the point $1 \in \partial \D$ to the point $\alpha
^0 (t) \in
\partial D_t$. The mappings $\Psi_t \defined \psi_t \circ \varphi_t$
(with $\psi _t(z) = \psi(t,z) $ for $t \in [0,1]$ and $z \in
\overline \D$) have the desired property.

\end{proof}

Choose an arc $\Gamma$ of the unit circle containing the point $1$
so that the mapping $(t,\zeta) \rightarrow \Psi (t,\zeta) = \Psi _t
(\zeta), \; (t,\zeta) \in [0,1] \times \Gamma,$ lifts to a
continuous mapping $\overset{\circ}{m}: [0,1] \times \Gamma
\rightarrow \GN$ for which $\hat {\mathcal{P}} _0 \circ
\overset{\circ}{m} = \hat \Psi$. In other words, the analytic discs
$\Psi _t $ have continuously changing halo on $\Gamma$ that is
associated to $\hat \Psi _t$.

According to Lemma 15 by attaching dendrites each disc $\Psi _t :
\overline {\D} \rightarrow X^2 $ can be performed into a neuron with
halo associated to the lift $\hat \Psi _t |\partial \D$. This can be
done so that the halo of the neuron on $\Gamma$ coincides with
$\overset{\circ}{m}(t,\cdot), \; t \in [0,1]$. In particular, for
each $t$ the arc $\Gamma$ consists of regular points for the neuron.
Further, the attaching of dendrites may be done in such a way that
the neurons depend piecewise continuously on the parameter $t$.

Consider the constructed neurons as preneurons and attach for each
$t \in [0,1]$ an axon $\mathbf{T}^{ax}_t$ to the respective
(pre)neuron such that the root of its tree is the regular point $1$.
The trees $T^{ax}_t$ of the axons $\mathbf{T}^{ax}_t$ are chosen to
depend continuously on $t$, for $t$ close to $1$ being equal to the
edge $T^{ax}_t = [1,2]$ which is orthogonal to the unit circle, and
degenerated to a point for $t$ close to $0$. In particular, the tips
of the axon trees, $\mathfrak{a}_t$ depend continuously on $t$.
Since the restrictions of the halo of the (pre)neurons to $\Gamma$
depend continuously on the parameter, the halo of the axon
$\mathbf{T}^{ax}_t$ may be chosen to depend continuously on $t$. We
define it in the following way. Let $m_{T_t}(\tau), \; \tau \in
[0,1],$ parametrize the punctured pellicle of $T^{ax}_t$. The
parametrization is chosen symmetric with respect to the sides of the
edge $T^{ax}_t$, i.e. $m_{T_t}(\tau) = m_{T_t}(1- \tau), \; \tau \in
[0,1]$. For the halo on the first side,
$\overset{\circ}{m}_{T_t}(\tau), \; \tau \in [0,1/2]$, we choose a
$G_0$-homotopy of the disc $\overset{\circ}{m}_t(0)$ to a disc
embedded into $G$ which is the value of the halo over the tip of the
axon. The halo on the second side is chosen symmetrically.

We obtain a piecewise continuous family of neurons with halo, which
we denote by $n_t = (\nu _t, m_t, \Psi _t, \overset{\circ}{m}_t),\;
t \in [0,1]$. The neurons have a continuously changing axon attached
whose halo at the tip is a small analytic disc embedded into $G$.
For $t$ close to $0$ the neuron coincides with the original analytic
disc which is embedded into $G$. For $t$ close to $1$ the main body
of the neuron coincides with the original disc.

In the next section we obtain from this family a continuous family of neurons with halo
with a continuously changing axon attached.

\section{A continuous family of neurons. "Peeling"}

This section is the key of the proof of Lemma 13.

Let $t_0$ be the first discontinuity point of the constructed family
$n_t$ of neurons with halo. Denote by $\;n_{t_0}^\pm =
(\nu_{t_0}^\pm, m_{t_0}^\pm,\Psi_{t_0}^\pm,\overset{\;\circ
\;\;\pm}{m_{t_0}})$ the respective limits from the left and from the
right. Note that the main bodies $\Psi_{t_0}^\pm : \overline \D
\rightarrow X^2 $ of the neurons $n_{t_0}^\pm$ coincide. Moreover,
the values of $\overset{\;\circ \;\;\pm}{m_{t_0}}$ coincide on
$\Gamma$.

There may be no homotopy joining the neurons $n_{t_0} ^-$ and
$n_{t_0} ^+$. The following lemma shows that there is such a
homotopy after attaching a special dendrite to $n_{t_0} ^+$.

\begin{lemma}\label{Lemma 18} There is a dendrite $\mathfrak{T}_{t_0}$
with punctured halo and a neuron with halo $n_{t_0}^0 = n_{t_0}^+ \,
\cup \mathfrak{T}_{t_0}$ obtained in the following way. The tree of
$\mathfrak{T}_{t_0}$ is attached to the generalized disc $\nu^+
_{t_0}$ of $n^+_{t_0}$ at a point $\zeta ^* \in \Gamma$. The
pellicle (respectively, the halo) of the neuron $n^+_{t_0}$
punctured at $\zeta ^*$ and the punctured pellicle (respectively,
the punctured halo) of $\mathfrak{T}_{t_0}$ match and define the
pellicle (respectively, the halo) of the neuron $n_{t_0}^+ \, \cup
\mathfrak{T}_{t_0}$. Moreover, there is a homotopy of neurons with
halo joining the neuron $n_{t_0}^- $
with the neuron $n_{t_0}^0 = n_{t_0}^+ \, \cup \mathfrak{T}_{t_0}$ .

\end{lemma}

\begin{proof}
In the proof we will skip everywhere the index $t_0$.

To ease reading we will first work out the proof in simple but
typical situations before giving the formal proof in the general
situation.
\medskip

\noindent {\it Step 1 of the proof. Peeling for $n^+$-regular points
$\zeta$.} Let $\zeta _0 = \exp (i\theta _0), \;\theta _0
>0,$ be a point in $\Gamma$ (counterclockwise from $1$). We let a
one-edge dendrite grow out of $n^-$ at the point $\zeta _0$ and let
its root run counterclockwise along the circle. More precisely, let
$\zeta = \exp (i t_{\zeta}), \; t_{\zeta}
>\theta _0,$ be a point on the unit circle situated counterclockwise from $\zeta _0$.
Let $\gamma_{\zeta} $ be the arc between $\zeta _0$ and $\zeta$,
$\gamma _{\zeta} = \{\gamma _{\zeta}(t) = \exp (it): \theta _0 \leq
t \leq t_{\zeta}\}$. Assume that all points of $\gamma_{\zeta} $ are
regular for the neuron $n^+$. Let $e_{\zeta}$ be a closed straight
line segment attached to $\bd \D$ at the endpoint $\zeta$ of $\gamma
_{\zeta}$ which is transversal to $\bd \D $ and meets $\overline \D$
exactly at $\zeta$.

Consider the generalized disc $\nu _{\zeta} \defined \nu ^- \cup e_
{\zeta}$. Give it the structure of a neuron $n_{\zeta} = (\nu _
{\zeta}, \Psi _{\zeta}, m_{\zeta}, \overset{\circ}{m}_{\zeta})$ with
halo in the following way.

Remove the point $\zeta _0$ from the unit circle and close up the
arc by adding two points over $\zeta _0$. We refer to this set as
the punctured circle (punctured at $\zeta _0$). In the same way we
define the pellicle of $\nu^-$ punctured at $\zeta _0$.  Denote by
$O_{\zeta}$ the union of the closed arc $\overline{\bd \D \setminus
\gamma_{\zeta}}$ of the circle with the "outer" side of $e_{\zeta}$
(i.e. the "right" side of the edge $e_{\zeta}$ with orientation
towards the root, in other words, the second side when surrounding
the edge counterclockwise starting from the root). This side is
pasted to $\overline{\bd \D \setminus \gamma_{\zeta}}$ at the point
$\zeta$.

Consider the excrescence $\mathcal{E}^-$ of the punctured circle
which is equal to the pellicle of $n^-$ punctured at $\zeta _0$. Let
$A_{\zeta}$ be a homeomorphism of $\mathcal{E}^-$ onto an
excrescence $O^*_{\zeta}$ of $O_{\zeta}$. Suppose $A_{\zeta}$ is the
identity on $\bd \D \setminus \gamma_{\zeta}$, maps $\gamma
_{\zeta}$ onto $e_{\zeta}$ and fixes $\zeta$. Moreover, assume that
$A_{\zeta}$ is affine on each segment of $\mathcal{E}^-$ that is
contained in an edge of an $n^-$-tree.

Assign a halo to $O^*_{\zeta}$ in the following way. Let for some
interval $I$ the mapping $m_{\mathcal{E}^-}(t), \; t \in I,$ be a
parametrization of $\mathcal{E}^-$. Then $m _ {O^*_{\zeta}}(t)
\defined A_{\zeta} \circ m_{\mathcal{E}^-}(t), \; t \in I,$
parametrizes $O^*_{\zeta}$ and we put
$\overset{\circ}{m}_{O^*_{\zeta}}(t) \defined
\overset{\circ}{m}_{\mathcal{E}^-}(t), \; t \in I$.

Assign to the arc $t \rightarrow \gamma _{\zeta}(t), \; t \in
[\theta _0,t_{\zeta}],$ the $n^+$-halo: choose the parametrization
${m}^+(t)= {\gamma _{\zeta}}(t), t \in [\theta _0,t_{\zeta}]$, for
the arc of the pellicle of $n^+$ and put $\overset{\circ}{m}_{\gamma
_{\zeta}}(t) = \overset{\circ}{m}^+(t),\; t \in [\theta
_0,t_{\zeta}]$.

Finally, parametrize the "inner" (i.e. "left")  side $e^l_{\zeta}$
of $e_{\zeta}$ by $e^l_{\zeta}(t)= A_{\zeta} \circ m^+(t), t \in
[\theta _0,t_{\zeta}],$ and define the halo on $e^l_{\zeta}$ by
$\overset{\circ}{m}_{e^l_{\zeta}}(t) = \overset{\circ}{m}^+(t), \; t
\in [\theta _0,t_{\zeta}]$. Note that the halo on $e^l_{\zeta}$ is
"$A_{\zeta}$-symmetric" (i.e symmetric with respect to the
homeomorphism $A_{\zeta}$) to the halo on $\gamma _{\zeta}$.

The three arcs $O^*_{\zeta}$, $\gamma_{\zeta}$ and $e^l_{\zeta}$
cover the pellicle of the generalized disc $\nu _{\zeta}$. The
values of the halo match at the common endpoints of the arcs.
Indeed, they match at the tip of
 $e_{\zeta}$ because this point is the image of $\zeta _0 \in \Gamma$ under the map $A_{\zeta}$
 and for points in $\Gamma$ the $n^+$-halo takes the same value as the $n^-$-halo.
They also match at the point $\zeta$ because $A_{\zeta}$ fixes this
point.

We obtained a neuron $n_{\zeta}$ with halo. It has a distinguished
attached dendrite $\mathbf{e}_{\zeta}$.

The construction proceeds as long as no $n^+$- dendrite is attached
to the interior of the arc $\gamma _{\zeta}$. It is arranged so that
it provides a family of neurons $n_{\zeta}$ that depend continuously
on the parameter ${\zeta}$ so that the values of the halo of each of
it is contained in the union of the set of values of the halo of
$n^+_{t_0}$ and $n^-_{t_0}$ . Notice that the parametrization of the
pellicle of $n^+$ can be chosen so that the arc
$t \rightarrow m^+(t),\; t \in [\theta _0,t_{\zeta}],$  of the
pellicle of the generalized disc $\nu^+$ is identical to the arc $t
\rightarrow \gamma _{\zeta}(t)$ of the circle.

\medskip

\noindent {\it Step 2. Reaching edge-like dendrites of $n^+$.}
Suppose the construction of step 1 has been made up to a point
$\zeta \in \partial \D$. We obtained a continuous family of neurons
joining $n^-$ with a neuron $n_{\zeta}= (\nu _ {\zeta}, \Psi
_{\zeta}, m_{\zeta},\overset{\circ}{m}_{\zeta})$. Recall that no
$n^+$-neuron is attached to the interior $Int \,(\gamma _{\zeta})$
so that $\gamma _{\zeta} = m^+([\theta _0,t_{\zeta}]) \subset \bd
\D$ for a parameter $t_{\zeta}$.

Suppose that  $\zeta$ is the root of a tree $T_{\zeta}$ of the
neuron $n^+$. Hence $t_{\zeta}$ parametrizes the initial point of
the pellicle of the tree $T_{\zeta}$. Let $t'_{\zeta}$ parametrize
the terminating point of the pellicle of $T_{\zeta}$.

Denote by $\mathcal{B}_{\zeta}$ the (closed) ray that bisects the
angle between $\gamma _{\zeta}$  and the edge $e_{\zeta}$ obtained
at step 1 (more precisely, the angle between the tangent ray to
$\gamma _{\zeta}$ at $\zeta$ and $e_{\zeta}$; we mean the angle
which is covered moving in counterclockwise direction around the
point $\zeta$.) Choose a closed convex cone $U_{\zeta}$ with vertex
$\zeta$ and non-empty interior which is symmetric with respect to
reflection in the symmetry ray $\mathcal{B}_{\zeta}$, (hence, it
contains $\mathcal{B}_{\zeta}$) and is contained in the sector
between $\gamma _{\zeta}$ and $e_{\zeta}$.

Suppose the tree $T_{\zeta}$ of the $n^+$-dendrite
$\mathbf{T}_{\zeta}$ attached at $\zeta$ consists of a single edge.

Our goal is to construct a continuous family of neurons which differ
only by a dendrite whose tree is attached at $\zeta$ and situated
inside the cone $U_{\zeta}$. The family is constructed so that it
joins the neuron $n_{\zeta}$ with a neuron $n'_{\zeta}$ so that
$n'_{\zeta}$ has the following property. The pellicle of its
generalized disc $\nu'_{\zeta}$ contains an arc that coincides with
$t \rightarrow m^+(t), \; t \in [\theta _0,t'_{\zeta}]$ (i.e. the
arc is constituted by $\gamma _{\zeta}$ together with the punctured
pellicle of $T_{\zeta}$).

For defining the family of neurons it is enough to describe the
family of dendrites.

Realize  $T_{\zeta}$ as a straight line segment in the plane in the
direction of $\mathcal{B}_{\zeta}$ meeting the generalized disc
$\nu_{\zeta}$ exactly at $\zeta$. Reparametrize the punctured
pellicle $m_{T_{\zeta}}$ of $T_{\zeta}$ by the interval $[0,1]$ and
symmetrically with respect to its sides. More precisely, denote by
$m_{T_{\zeta}}:[0,1] \rightarrow T_{\zeta}$ the (reparametrized)
punctured pellicle of $T_{\zeta}$. We require that this mapping has
the following symmetry property: for each $t \in [0,1]$ the points
$m_{T_{\zeta}}(t)$ and $m_{T_{\zeta}}(1-t)$ are at different sides
of the pellicle over the same point. The halo
$\overset{\circ}{m}_{T_{\zeta}}$ is reparametrized accordingly by
the interval $[0,1]$.

Construct a continuous family of dendrites $\mathbf{T}_{\zeta}^s, \;
s \in [0,1],$ with punctured halo, the tree ${T}_{\zeta}^s$ of which
has root $\zeta$ and such that
\begin{itemize}
\item for each $s$ the tree ${T}_{\zeta}^s$ is contained in
$U_{\zeta}$ and meets the boundary of $U_{\zeta}$ exactly at
$\zeta$;
\item  for each $s$ the values of the punctured halo of
$\mathbf{T}_{\zeta}^s$ at the initial and terminating point coincide
and are equal to $\overset{\circ}{m}_{T_\zeta}(0)$; the dendrites
are mirror symmetric with respect to reflection in the symmetry ray
$\mathcal{B}_{\zeta}$;
\item  $\mathbf{T}_{\zeta}^0 $ is a one-point dendrite;
\item  $\mathbf{T}_{\zeta}^1$ consists of the union of two dendrites
attached at $\zeta$ ("dendrite twins"). The first of the two
dendrites (i.e. its underlying tree, its punctured pellicle and
punctured halo) is a homeomorphic copy of $\mathbf{T}_{\zeta}$ and
is (by a slight abuse) denoted again by $\mathbf{T}_{\zeta}$. Its
tree is placed in the closed part $U_{\zeta}^-$ of the cone
$U_{\zeta}$ which is clockwise from $\mathcal{B}_{\zeta}$, and meets
the boundary of  $U_{\zeta}^-$ exactly at $\zeta$;\\
\noindent The second dendrite is mirror symmetric to the first one
with respect to reflection in the symmetry ray $\mathcal{B}_{\zeta}$
and is denoted by $\mathbf{T}^*_{\zeta}$.\\
\noindent The value of the punctured halo of the dendrite
$\mathbf{T}^1_{\zeta}$ at the point that lies over $\zeta$ between
the dendrite twins coincides with the value at the terminating point
of the pellicle of $\mathbf{T}_{\zeta}$.
\end{itemize}

We call this procedure "growing of dendrite twins" (see below Lemma
21 for the general case).
\medskip

The construction is the following. For $s=0$ we obtain a one-point dendrite
$\mathbf{T}^0_{\zeta}$. The procedure of attaching this dendrite $\mathbf{T}^0_{\zeta}$
does not change $n_{\zeta}$.

For $s \in (0,1/2]$ the tree $T_{\zeta}^s$ of the dendrite
$\mathbf{T}_{\zeta}^s$ is an edge and consists of the points
$m_{T_{\zeta}} ( [0,s]) $.

Parametrize the pellicle of the tree $T_{\zeta}^s$ by the interval
$[0,2s]$ and symmetrically with respect to the sides of the tree:
take $m_{T^s_{\zeta}}|[0,s]
\defined m_{T_{\zeta}}|[0,s]$ (parametrization of the first side of
the tree), and symmetrically, $m_{T^s_{\zeta}}(\tau) \defined
m_{T_{\zeta}}(2s -\tau)$ for $\tau \in [s,2s]$ (parametrization of
the second side of the tree).

Respectively, the halo of the dendrite $\mathbf{T}_{\zeta}^s$ is
defined by the relations $\overset{\circ}m _{T^s_{\zeta}} |[0,s]
\defined \overset{\circ}m _{T_{\zeta}} |[0,s]$ on the first side, and
symmetrically, $\overset{\circ}m _{T^s_{\zeta}}(\tau) \defined
\overset{\circ}{m} _{T^s_{\zeta}}(s-\tau ), \;\tau \in [s,2s],$ on
the second side of the dendrite.

For $s \in (1/2,1]$  the tree  $T_{\zeta}^s$ of the dendrite becomes
a letter "$ \textsf{Y}$" which is symmetric with respect to the
symmetry ray.

Describe the tree $T_{\zeta}^s$. Denote by $a_{\zeta}^s$ the segment
$m_{T^s_{\zeta}}([0,1-s])$ of $T_{\zeta}$ which is adjacent to
$\zeta$ (note that the number $1-s$ is less than $\frac{1}{2}$).
Denote the remaining segment  $b_{\zeta}^s
\defined \overline {T_{\zeta} \setminus a_{\zeta}^s}$. The segment $a_{\zeta}^s
\subset \mathcal{B}_{\zeta}$ is the "trunk" of the letter
$\textsf{Y}$. The "first branch of the letter $\textsf{Y}$" is the
image $\mathcal{R}^s_{\zeta} ( b^s_{\zeta})$ of $b^s_{\zeta}$ under
a rotation $\mathcal{R}^s_{\zeta}$ around the common endpoint
$m_{T^s_{\zeta}}(1-s)$ of $a^s_{\zeta}$ and $b^s_{\zeta}$. The
rotation is chosen so that the rotated segment
$\mathcal{R}^s_{\zeta} ( b^s_{\zeta})$ is placed in $U_{\zeta}^-$
and meets the boundary of  $U_{\zeta}^-$  exactly at
$m_{T^s_{\zeta}}(1-s)$. The rotations $\mathcal{R}^s_{\zeta}$ are
chosen continuously depending on $s$.

\begin{figure}[h]
\centering
\includegraphics[scale=0.9]{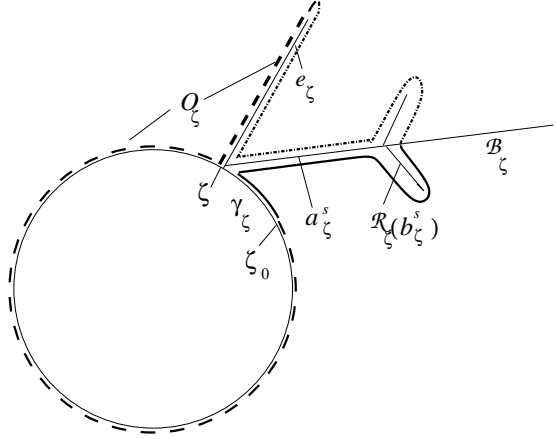}
\caption{"Peeling" in case of a single $n^+$-edge at $\zeta$}
\end{figure}

The second branch of the letter $\textsf{Y}$ is chosen symmetric to
the first one with respect to mirror reflection in the symmetry ray
$\mathcal{B}_{\zeta}$.

Describe the punctured pellicle $m^s_{\zeta }$ of the tree
$T_{\zeta}^s$ and the halo of the dendrite $\mathbf{T}_{\zeta}^s$.
The part of the pellicle of $T_{\zeta}^s$  corresponding to the
first side of $a^s_{\zeta}$ coincides with the corresponding part of
the pellicle of  $T_{\zeta}$: $m^s_{\zeta}(\tau) = m_{\zeta}(\tau)$
for $\tau \in [0,1-s]$. Respectively, for the halo the relation $
\overset{\;\; \circ \;\; s}{m}_{\zeta}(\tau) \defined \overset{
\circ} {m}_{\zeta}(\tau)$ for $\tau \in [0,1-s]$ holds.

For $\tau $ in the interval $[1-s,s]$ the relation is $m^s_{\zeta}
(\tau) = \mathcal{R}^s_{\zeta} \circ m_{\zeta} (\tau)$. This part of
the pellicle $m^s_{\zeta}$ surrounds $\mathcal{R}^s_{\zeta} (
b^s_{\zeta})$.
The halo of  $\mathbf{T}_{\zeta}^s$ for those parameters $\tau$ is
defined by the halo of  $\mathbf{T}_{\zeta}$: we put $\overset {\;
\; \circ \;\; s}{m_{\zeta}} (\tau) = \overset {\circ}{m}_{\zeta}
(\tau)$ for $\tau \in [1-s,s]$.

The remaining part of the punctured pellicle and punctured halo of the
dendrite $\mathbf{T}^s_{\zeta}$ is mirror symmetric to the just described part.

For $s=1$ we arrive at a mirror symmetric pair of dendrites with the properties described above.
The construction for this case is completed.

\medskip

\noindent {\it Step 3. The general case.} Let $\zeta _0$ be as above
a point in $\Gamma$ situated counterclockwise from the root $1$ of
the axon. Suppose $\zeta \in \bd \D \setminus \Gamma$ is reached by
moving counterclockwise from $\zeta _0$ and $\gamma _{\zeta }$ is
the closed arc of the circle between $\zeta _0$ and $\zeta$
(counterclockwise traveling). Let $m^+$ parametrize the pellicle of
$n^+$ punctured at $\zeta _0$, $m^+(\theta _0) = \zeta _0$.

If $\zeta$ is a regular point for $n^+$ then there is a unique
parameter $t_{\zeta}$ in the pellicle of $n^+$ for which the
equality $m^+(t_{\zeta}) = \zeta$ holds. If $\zeta $ is not regular
for $n^+$ then there is a finite collection of increasing parameters
$t^1_{\zeta}, ...,t^l_{\zeta}$ for which $m^+(t^j_{\zeta}) = \zeta$.
Here $t^1_{\zeta}$ parametrizes the initial point of the $n^+$-tree
attached at $\zeta$ and $t^l_{\zeta}$ parametrizes its terminating
point. The points $ t^j_{\zeta}$ and $t^{j+1}_{\zeta}$ parametrize
the initial, respectively the terminating, points of the simple
trees constituting the tree at $\zeta$.

The plan is the following. Let $\zeta \in \partial \D \setminus
\Gamma$ be any point counterclockwise from $\zeta _0$ and let
$t_{\zeta}$ denote one of the parameters for which $ m^+(t_{\zeta})
=\zeta$. Assume a neuron $n_{t_{\zeta}}$ is constructed such that
the pellicle of its tree contains the arc $\tau \rightarrow
m^+(\tau), \tau \in [\theta _0,t_{\zeta}]$. We will construct a
neuron such that an arc of its pellicle coincides with $\tau
\rightarrow m^+(\tau),\; \tau \in [\theta _0,t]$, for some
parameters $t>t_{\zeta}$.

Here is the precise description of the induction hypothesis.

Suppose a neuron $n _{t_{\zeta}}$ is constructed with the following
properties. Its generalized tree $\nu _{\zeta}$ has an edge
$e_{\zeta}$  attached at $\zeta$. Let as in step 2
$\mathcal{B}_{\zeta}$ be the (closed) ray that bisects the angle
between $\gamma _{\zeta}$ and the edge $e_{\zeta}$. The main
property of $n _{t_{\zeta}}$ is the following. The pellicle of $\nu
_{\zeta}$ (considered as a curve parametrized by the unit circle
$\partial \D$) has a partition into three parts each reparametrized
by an interval.
\begin{itemize}
\item [(1)] The first part is the excrescence $O^*_{\zeta}$ of
$O_{\zeta}$. Its halo is defined by $\mathcal{E}^-$ as in step 1.
\item [(2)] The second part is the excrescence $ \gamma ^*_{\zeta}(t) \defined m^+(t),\; t \in
[\theta _0,t_{\zeta}],$ of $ \gamma_{\zeta}$. We assume the
excrescence is chosen so that its image $m^+([\theta _0,t_{\zeta}])$
is situated clockwise from $\mathcal{B}_{\zeta}$ and meets
$\mathcal{B}_{\zeta}$ exactly at the points $m^+(t^j_{\zeta}) =
\zeta, \; j=1,...,i$, where $t^i_{\zeta}= t_{\zeta}$ . The halo on
the second part is defined by $\overset{\circ}{m}^+(t), \; t \in
[\theta _0,t_{\zeta}]$.
\item [(3)] To define the third part we consider an extension of the
homeomorphism $A_{\zeta}$ from $\mathcal{E}^-$ to the image
$m^+([\theta _0,t_{\zeta}])$ such that $A_{\zeta}$ is affine on each
straight line segment contained in $m^+([\theta _0,t_{\zeta}])$.
Moreover, the image $A_{\zeta} \circ m^+([\theta _0,t_{\zeta}])$ is
contained in the closed angle between $\mathcal{B}_{\zeta}$ and
$e_{\zeta}$ (i.e. counterclockwise from $\mathcal{B}_{\zeta}$) and
meets $\mathcal{B}_{\zeta}$ exactly at the points $m^+(t^j_{\zeta})
= \zeta,
\; j=1,...,i,$  .\\
The third part is the excrescence $(e^l_{\zeta})^*(t) = A_{\zeta}
\circ m^+ (t) , \; t \in [\theta _0,t_{\zeta}]$. The halo is defined
by $\overset{\circ}{m}_{e(^l_{\zeta})^*}(t) =
\overset{\circ}{m}^+(t),\; t \in [\theta _0,t_{\zeta}]$.

\end{itemize}

\begin{figure}[h]
\centering
\includegraphics[scale=0.50]{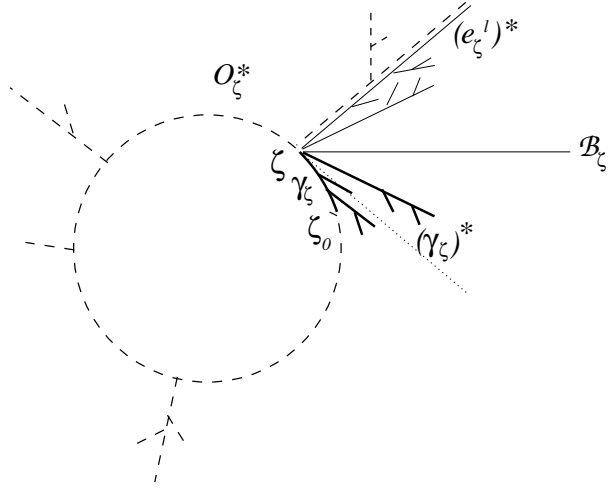}
\caption{"Peeling": The generalized disc $\nu _{t_{\zeta}}$}
\end{figure}

Two possibilities may arise.
\begin{itemize}
\item [(a)] Points in the pellicle of $n^+$ parametrized by $t > t_{\zeta}$ and close to $t_{\zeta}$
are regular points contained in $\bd \D$.
\item [(b)] $t_{\zeta}$ is the initial point of one of the simple trees that constitute the $n^+$-tree
attached at $\zeta$. We denote this tree for short by $T_{\zeta}$
and the respective dendrite by $\mathbf{T}_{\zeta}$. (Notice that
the structure of the whole $n^+$-dendrite that is attached at
$\zeta$ does not play a role in the proof.) Let $t'_{\zeta}$
parametrize the terminating point of the pellicle of $T_{\zeta}$. So
$m^+(t'_{\zeta})= \zeta$ and the arc $t \rightarrow m^+(t), \; t \in
[t_{\zeta},t'_{\zeta}]$, of the pellicle of $\nu ^+$ is the
punctured pellicle of the tree $T_{\zeta}$.
\end{itemize}

\medskip

Here are the constructions in cases (a) and (b). In the first case
(a) we proceed like in step 1 of the proof. We change the root of
the main edge $e_{\zeta}$ in counterclockwise direction along the
circle and let the edge grow.
More precisely, let $\zeta ' \in \partial \D$ be counterclockwise of
$\zeta$ and let the arc between $\zeta$ and $\zeta'$ consist of
regular points. At $\zeta '$ we attach an edge $e_{\zeta'}$ and
equip it with the following structure. For a segment of $e_{\zeta'}$
adjacent to the leaf we take an excrescence on each of the sides of
the edge (and the respective halo on it) that is homeomorphic to the
respective one for $e_{\zeta}$. For the remaining segment of
$e_{\zeta'}$ that is adjacent to the root $\zeta '$ we proceed as in
step 1.

In the second case (b) we will construct a continuous family of
neurons that connects $n_{\zeta}$ with a neuron  $n'_{\zeta}$ so
that the final neuron $n'_{\zeta}$ has the following properties. As
for the original neuron the pellicle of $n'_{\zeta}$ has a
decomposition into three parts satisfying properties (1), (2) and
(3) with $t_{\zeta}$ replaced by $t_{\zeta'}$. Thus, the pellicle of
$n'_{\zeta}$ contains the arc $t \rightarrow m^+(t),\; t \in [\theta
_0,t'_{\zeta}]$. (Recall that $t \rightarrow m^+(t),\; t \in
[t_{\zeta},t'_{\zeta}]$, is the punctured pellicle of $T_{\zeta}$).

To construct the family of neurons it is enough to construct the
respective family of dendrites attached to $n_{\zeta}$ at the point
$\zeta$. The following lemma provides this construction.


\begin{lemma}\label{Lemma 19}{\it (On growing of dendrite twins)}
Let $\gamma_{\zeta}$, $e_{\zeta}$ and $\mathcal{B_{\zeta}}$ be as
above.
Let $U_{\zeta}$ be a closed convex cone with vertex $\zeta$ which is
symmetric with respect to reflection in $\mathcal{B_{\zeta}}$ and
contained in the sector between $\gamma_{\zeta}$ and $e_{\zeta}$.
Denote by $U_{\zeta}^-$ the closed part of $U_{\zeta}$ which
is situated clockwise from $\mathcal{B_{\zeta}}$. Let
$\mathbf{T}=(T,m_T, \Phi_T, \overset{\circ}m_T)$ be a dendrite with
halo. Suppose the pellicle of the underlying tree $T$ is
parametrized by $[t',t'']$.

Consider a point $\xi \in \mathcal{B_{\zeta}}$ and a closed convex
cone $U_{\xi} \subset U_{\zeta} $ with vertex $\xi$ which is
symmetric with respect to reflection in $\mathcal{B_{\zeta}}$.

Then there exists a continuous family of dendrites
$\mathbf{T}_{\xi}^s,\; s \in [0,1],$ with root $\xi$ with the
following properties:
\begin{itemize}
\item [(1)] For all $s \in [0,1]$ the tree of $\mathbf{T}_{\xi}^s$ is contained in the
cone $U_{\xi}$ and meets the boundary of the cone exactly at $\xi$.
\item [(2)] For all $s$ the values of the halo
$\overset{\circ}{m}^s_{{T}_{\xi}}$ at the initial point and at the
terminating point
coincide and equal $\overset{\circ}{m}^+_{T}(t')$.
The dendrites are mirror
symmetric for reflection in the ray $\mathcal{B}_{\zeta}$.
\item [(3)] The dendrite $\mathbf{T}_{\xi}^0$ is a one-point dendrite.
\item [(4)] The dendrite $\mathbf{T}_{\xi}^1$ is a dendrite twin
attached at $\xi$. 
The tree of the first labeled twin dendrite is contained in
$U_{\xi}^- \defined U_{\zeta}^- \cap U_{\xi}$ and meets the boundary
of $U_{\xi}^- $ exactly at $\xi$.
The first labeled dendrite (i.e. its underlying tree, its punctured
pellicle and punctured halo) is a homeomorphic copy of $\mathbf{T}$.
The second dendrite is mirror symmetric to the first one with
respect to reflection in the symmetry ray. The value of the halo of
the dendrite twin at the point between the twins equals
$\overset{\circ}{m}_{T}(t'')$.

\end{itemize}
\end{lemma}

\begin{proof}
If the tree of the dendrite $\mathbf{T}$ consists of a single edge
the proof was given in step 2. Prove the lemma by induction on the
number of edges of the tree $\mathbf{T}$.

Suppose first that the (planar) tree $T$ is not simple, i.e. it has
more than one (non-empty) edges {\it adjacent} to the root. Then the
tree is the union of two (planar) trees $T'$ and $T''$ with the same
root labeled so that $T''$ is counterclockwise of $T'$.  Each of the
trees $T'$ and $T''$ has less edges than $T$. By induction
hypothesis the required family $\mathbf{T'}_{\xi}^s,\; s \in [0,1],$
of dendrites rooted at $\xi$ exists for the first dendrite
$\mathbf{T'}$. The final dendrite  $\mathbf{T'}_{\xi}^1$ is the
union of mirror symmetric twins. The first of the twins is denoted
by $\mathbf{T'}_{\xi}$ (situated clockwise from
$\mathcal{B}_{\zeta}$) and the second twin is denoted by
$(\mathbf{T'}_{\xi})^*$ (situated counterclockwise from
$\mathcal{B}_{\zeta}$)).

Consider a smaller closed convex cone $\overset{\circ}U _{\xi}
\subset U_{\xi}$ with vertex $\xi$ which is symmetric with respect
to $\mathcal{B}_{\zeta}$ and meets the trees $T'_{\xi}$ and
$(T'_{\xi})^*$ exactly at the point $\xi$.

An application of the induction hypothesis to the point $\xi$, the
cone $\overset{\circ}U _{\xi} $ and the second dendrite
$\mathbf{T}''$ finishes the proof in this case.

\medskip

Consider the remaining case when the tree $T$ is simple, i.e. it has
a single edge $E$ adjacent to its root. Realize $E$ as a segment
$E_{\xi}$ with initial point $\xi$ on the symmetry ray
$\mathcal{B}_{\zeta}$ (traveled in positive direction of
$\mathcal{B}_{\zeta}$). Associate to the tree $E_{\xi}$  the
following dendrite $\mathbf{E}^-_{\xi}$ with halo. The tree of
$\mathbf{E}^-_{\xi}$ is chosen equal to $E_{\xi}$. The halo on the
first side of $E_{\xi}$ is taken to coincide with the halo of
$\mathbf{T}$ along the first side of the edge $E$ of its tree. The
halo on the second side of $E_{\xi}$ is chosen symmetrically. There
is a continuous family of dendrites which join the one-point
dendrite with root $\xi$ with the dendrite $\mathbf{E}^-_{\xi}$.

Denote by $\mathbf{T}^E$ the dendrite obtained by removing
$\mathbf{E}$ from $\mathbf{T}$. In other words, the tree of
$\mathbf{T} ^E$ equals $T^E \defined T \setminus E$.
The halo of the dendrite
$\mathbf{T}^E$ is the restriction of the halo of $\mathbf{T}$.

The tree $T^E$ has an edge less than $T$. The induction hypothesis
applied to $T^E$, the endpoint $\eta$ of the tree $E_{\xi}$ and a
closed convex cone $U_{\eta} \subset {U}_{\xi}$ symmetric with
respect to $\mathcal{B}_{\zeta}$ gives a continuous family of
dendrites that join the one-point dendrite at $\eta$ with a dendrite
twin $\mathbf{T}^E_{\eta} \bigcup (\mathbf{T}^E_{\eta})^*$ rooted at
$\eta$. Here ${T}^E_{\eta} \subset U_{\eta}$ is situated clockwise
from $\mathcal{B}_{\zeta}$ and $({T}^E_{\eta})^* \subset U_{\eta}$
is counterclockwise from $\mathcal{B}_{\zeta}$. Cut the punctured
pellicle of $E^-_{\xi}$ at the tip $\eta$ and paste the punctured
pellicle of ${T}^E_{\eta} \bigcup ({T}^E_{\eta})^*$.

\begin{figure}[h]
\centering
\includegraphics[scale=0.7]{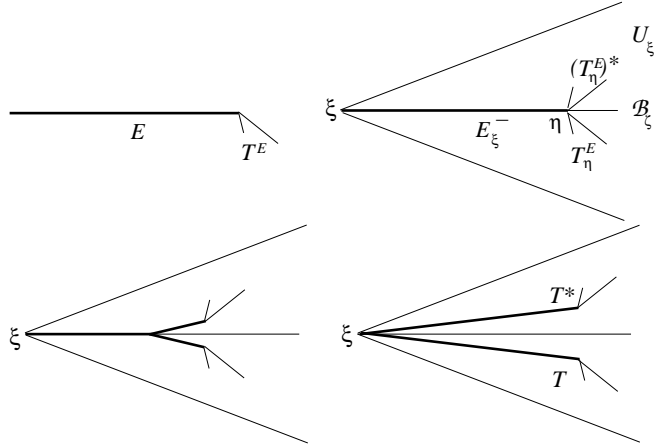}
\caption{Growing of dendrite twins}
\end{figure}

The punctured halo of the twin dendrite $\mathbf{T}^E_{\eta} \bigcup
(\mathbf{T}^E_{\eta})^*$ matches with that of $\mathbf{E}_{\xi}^-$
at the point $\eta$. The result of pasting is a dendrite with halo
which can be joined with the one-point dendrite at $\xi$ by a
continuous family.

The rest of the construction is based, as in step 2, on splitting
the segment $E_{\xi}$ into a letter ${\textsf{Y}}$ but with copies
of $T^E_{\xi}$ (respectively $(T^E_{\xi})^*$) attached at the tip of
the first branch (respectively, of the second branch) of the letter
$\textsf{Y}$.

It remains to define the halo on the $\textsf{Y}$. The pellicle of
the $\textsf{Y}$ punctured at the bottom point has a partition into
three arcs: the part, seen from the right (the union of the first
side of the steam and the first side of the first branch), the part
seen from above (the union of the second side of the first branch
and the first side of the second branch) and the part seen from the
left (the union of the second side of the second branch and the
second side of the steam). The halo on the part seen from the right
(respectively seen from the left) is the halo on the first side
(respectively on the second side) of $\mathbf{E}^-_{\xi}$ after a
change of variables. The halo on the part seen from above is defined
as in step 2.




We defined a continuous family of dendrites with halo. The final
dendrite of the family is the required twin dendrite. The proof of
lemma 21 is finished.

\end{proof}

To finish step 3 of the proof of Lemma 20 we apply Lemma 21 to the
dendrite $\mathbf{T}_{\zeta}$, the point $\zeta$ and a closed convex
cone $U_{\zeta}$ contained in the sector between $\gamma _{\zeta}$
and $e_{\zeta}$ which meets the trees of $n_{\zeta}$ at most at
$\zeta$. The desired continuous family of neurons is obtained by
pasting the constructed family of dendrites obtained in Lemma 21.

The general "peeling"-procedure described in step 3 can be continued
until a point $\zeta ^* \in \Gamma \subset \bd \D$ situated
clockwise (within $\Gamma) $ from the point $1$ is reached.

By assumption $\Gamma \setminus \{1\}$ consists of regular points
for both, $n^+$ and $n^-$, and the $n^+$-halo coincides with the
$n^-$-halo on $\Gamma \setminus \{1\}$. Hence, the obtained neuron
$n_{\zeta^*}$ has the required property: it differs from $n^+$ by a
dendrite attached at $\zeta ^*$. Lemma 20 is proved.
\end{proof}

Lemma 20 yields a continuous family of neurons with halo that joins
the neuron $n_0$ with the neuron $n^+_{t_0} \cup
\mathfrak{T}_{t_0}$. By a change of the $t$-variables we may assume
that the parameter set is again the interval $[0,t_0]$. For $t$
close to $0$ the new neurons coincide with the previous ones and for
$t=t_0$ the new neuron coincides with $n^+_{t_0} \cup
\mathfrak{T}_{t_0}$.

For all $t > t_0$ we attach to the neuron $n_t$  a dendrite
$\mathbf{\mathfrak{T}}_{t}$ with halo and root $\zeta ^*$ of the
underlying tree. The family of dendrites with halo
$\mathbf{\mathfrak{T}}_t$ is chosen continuously depending on $t$
and converging to $\mathbf{\mathfrak{T}}_{t_0}$ for $t \rightarrow
t_0$. A continuous choice of the dendrites can be made since the
halo of the neurons on the arc $\Gamma$ changes continuously.

We obtain a piecewise continuous family of neurons with halo.
Moreover, the family of neurons has one discontinuity point less
than the previous family. Shrink the arc $\Gamma$ (keeping the same
notation) so that the arc still contains the point $1$ and $\Gamma
\setminus \{1\}$ is free from roots of attached trees for all $t \in
[0,1]$.


Consider all (finitely many) discontinuity points $t_j$ (in
increasing order) of the family $n_t$. Apply Lemma 20 successively
to each $n_{t_i}$ and attach to the $n_t, \; t \ge t_i, $ dendrites
that depend continuously on $t$. At each step the arc $\Gamma$ is
shrinken suitably.

We arrive at a continuous family of neurons with halo. Denote the
neurons by\\
$N_t = (\nu' _t, \phi _t, M_t, \overset{\circ}{M}_t)$. All
generalized discs $\nu' _t$ coincide with the closed unit disc with
a number of trees attached. In particular, each generalized disc
$\nu'_t$ contains the tree $T^{ax}_t$ of an axon attached at the
point $1$. For each $t$ there is a number $t'$ such that the
restricted mapping $\phi _t |\overline \D$ coincides with the
original mapping $\Phi _{t'}$ from Lemma 13. Moreover, for $t$ close
to $1$ the restrictions coincide with the mappings from Lemma 13:
$\phi _t |\overline \D = \Phi _t$ for $t$ close to $1$. For $t$
close to $0$ the generalized discs $\nu' _t$ coincide with the unit
disc $ \overline \D$ and the neurons coincide with the original
analytic discs of Lemma 13. They are small discs embedded into $G$
and the values of their halo are small discs embedded into $G$. For
all $t$ the halo $\overset{\circ}{M}_t$ is associated to the lift of
$\Phi _t |\partial \D$ to $\hat G$. In other words, the restriction
of the mapping $\hat M _t = \hat{\mathcal{P}}_0 \circ
\overset{\circ}{M}_t$ to $\partial \D$  coincides with $\hat \Phi _t
|\partial \D$.

In the sequel we need also the following property of the neurons.
Choose parametrizations ${M}_t({\xi}),\;  t \in [0,1], \; \xi \in
\partial \D,$ of the pellicle of $\nu _t$ which depend continuously
on $t$. The property is the following. There exists a compact subset
$\kappa$ of $G$ such that $\bigcup_{t \in [0,1], \xi \in
\partial \D} \overset{\circ}{M}_t({\xi})(\partial \D) \subset
\kappa$. Moreover, let for each $t$ the point ${M}_t({\xi}_0)$ be
the tip of the axon tree of $\nu_t$. Then $\bigcup_{t \in [0,1]}
\overset{\circ}{M}_t({\xi}_0)(\overline \D) \subset \kappa$ and,
hence, in particular, $\bigcup_{t \in [0,1]} \phi _t \circ
{M}_t({\xi}_0)\subset \kappa$ and $\hat {\phi}_t \circ
{M}_t({\xi}_0) \subset \hat i (G) $.



\section{Proof of lemma 13}

Using the continuous family of neurons $N_t$ with halo obtained in
the previous section the proof of Lemma 13 can be completed
essentially along the same lines as the proof of lemma 12. Here are
the details.

Fix an $\varepsilon >0$ which is small compared to the distance of
$\kappa $ to the boundary of $G$. Apply the procedure of continuous
fattening of dendrites (Lemma 16) to all neurons $N_t$ and all
attached dendrites. We obtain a continuous family of analytic discs
with continuously varying halo, denoted by $(D_t, m_t, \psi _t,
\overset{\circ}{m}_t)$, $t \in [0,1]$, for which $\max _{\nu' _t}
|\psi _t - \phi _t| < \varepsilon$ and $\overset{\circ}m_t$ is
$\varepsilon$-close to the halo $\overset{\circ}M_t$ of the
respective original neuron. (We abuse notation for the pellicle and
the halo using the same letter as for the objects related to the
original family $\Phi _t$). The sets $\overline D _t$ are closures
of continuously changing bounded simply connected and smoothly
bounded domains in the complex plane. The sets $\overline D _t$ are
obtained from the closed unit disc by attaching "closed thickened
trees". The "closed thickened axons" play a special role. These are
thin closed neighbourhoods of the interiors of the axon trees
$T^{ax}_t$ that depend continuously on $t$ and are pasted to the
closed unit disc along an arc of the circle. For each $t$ the tip
$\mathfrak{a}_t$ of the axon is the only point of the axon that is
located on the boundary of the respective domain $D_t$.

Since for each $t$ the inclusion $\hat {\phi} _t(\mathfrak{a}_t) \in
\hat i(G)$ holds,
there are closed arcs $\mathfrak{A}_t$ contained in $\bd D_t$,
$\mathfrak{a}_t \in \mathfrak{A}_t$, for which $\hat
\psi_t(\mathfrak{A}_t) \subset \hat i(G)$, provided $\varepsilon$ is
small enough. Choose continuously changing open arcs
$\mathfrak{A}^0_t$ which are relatively compact in $Int
\,\mathfrak{A} _t$ with $\mathfrak{a}_t \in \mathfrak{A}^0_t $.


Use the same notation as in the proof of lemma 12: $\mathcal{S}
_{\bd D_t}
\defined (\overline D_t \times \{0\})\; \bigcup\;
(\partial D_t \times \overline \D)$, $\mathcal{S}_{\bd D_t \setminus
{\mathfrak{A}}^0_t} \defined (\overline D _t \times \{0\}) \bigcup
(\partial D_t \setminus \mathfrak{A}_t^0 \times \overline \D)$ and
$\mathcal{Q} ^t
\defined (\bd D_t \, \times {\bd \D}) \; \bigcup \; (\mathfrak{A}_t \times \overline
\D)$.

Define, as in the proof of Lemma 12, for each $t$ a mapping
$\mathcal{J}_t$ on $\mathcal{S}_{\bd D_t}$ which equals $\psi_t$ on
the central disc $\overline D_t \times \{0\}$ and is equal to the
evaluation map for $\overset{\circ}{m}_t$ on the disc fibers over
$\bd D_t$. The mappings $\mathcal{J}_t$ depend continuously on $t$.

For each neighbourhood $\mathcal{V}$ of $\kappa$ the number
$\varepsilon$ and the arcs $\mathfrak{A} _t$ may be chosen so that\\
$\bigcup _{t \in [0,1]} \mathcal{J}_t (\mathcal{Q}_t) \subset
\mathcal{V}$.

Let $K_t$ denote the following compact subset of $\D \;\bigcup \;
T^{ax}_t$: $K_t \defined r \overline \D \;\bigcup \; [r,1] \;\bigcup
\; T^{ax}_t $ .(Recall that for each $t$ we denote by $T^{ax}_t$ the
tree of the axon of the neuron $n_t$.) Note that $K_t$ is a compact
subset of $ D_t \;\bigcup \; \mathfrak{A}^0_t$. For $t$ close to $1$
$K_t = r\overline \D \;\bigcup \; [0,2] $.

By Lemma 17 there is a continuous family of mappings $\H _t \in
A_{X^2} (D_t \times \D)$ such that for an arbitrary point $z \in
\partial \D$ the mappings $f^z_t, \; t \in [0,1], \;
f^z_t (\zeta) \defined \H_t (\zeta,z), \; \zeta \in \overline D _t,$
define a continuous family of analytic discs with boundary in
$\mathcal{V}$ satisfying the inequality $\max_{K_t}|\psi _t -f^z_t|
< \varepsilon$ for all $t \in [0,1]$.  Moreover, by the special
choice of $\psi _0 = \Phi _0$ Lemma 17 implies that the disc $f^z_0
(\overline \D)$ is entirely contained in $G$.




Fix a point $z \in \partial \D$. An application of lemma 9 to the
family $f^z_t$ produces a family of {\it immersed} discs $f_t$ with
all above listed properties preserved. In particular the boundaries
of the discs $f_t(\partial \D)$ are contained in $\mathcal{V}$.

Take an arbitrary point $p \in \Phi_1 (\D)$. Choosing $r$ close
enough to $1$ we may assume that $p \in \Phi_1 (K_1)$.
Further, we may assume that the family $f_t$ is chosen so that $p
\in f_1 (K_1)$. (This can be achieved considering, in case $X^2 = \C
^2$, small translations of the discs of the family and in the
general case by applying compactly defined holomorphic mappings
close to the identity on $X^2$. ) We proved that $p$ is contained in
the projection $\hat {\mathcal{P}}(\hat G)$.

To choose a standard lift of a neighbourhood of $p$ to $\hat G$ we
reparametrize $f_1$. More precisely, consider the composition $f_1
\circ \varphi _1$ with a conformal mapping $\varphi_1$ from the unit
disc onto $D_1$ such that $f_1 \circ \varphi_1 (0)= p$. For a number
$\textsf{r}<1$ and close to $1$ we consider the function $\zeta
 \rightarrow f_1 \circ \varphi _1 (\textsf{r} \zeta)$ and denote it by $d_p$.
Let $\hat d_p$ be the equivalence class represented by $d_p$.

Consider a standard neighbourhood ${\hat {\mathcal{P}}}:\hat V
\rightarrow Q_p$ of $\hat d_p$ associated to the representative
$d_p$ (see section 3). Here $Q_p$ is a neighbourhood of $p$ in
$X^2$, $\hat{\mathcal{P}}$ is biholomorphic and for $q \in Q_p$ the
classes $\hat d _q = (\hat {\mathcal{P}}| \hat V) ^{-1} (q)$ are
represented by a continuous family of analytic discs $d_q$. For
$q=p$ the disc coincides with the one defined before.

It remains to see that this standard lift of $Q_p$ to $\hat G$ is
compatible with the lift $\hat \Phi $ of $\Phi$. More precisely, let
$(t,z') \in [0,1) \times \D$ be close to $(1,z)$, so that $q
\defined \Phi(t,z')$ is contained in $Q_p$. We have to prove the
following Lemma 22.

\begin{lemma}\label{Lemma 20} The equivalence classes $\hat d _q$ and
$\hat {\Phi} (t,z')$ coincide.
\end{lemma}




\medskip

\noindent {\it Proof of Lemma 22.} Recall that for $t$ close to $1$
$\hat \phi _t | \partial \D = \hat  \Phi _t| \partial \D $. For $t <
1$ close to $1$ we extend $\hat \phi _t$ to $\overline \D$ by $\hat
\phi _t| \overline \D \defined \hat \Phi _t$. It is enough to find
two curves $\hat {\gamma} _d$ and $\hat {\gamma}_{\Phi}$ in $\hat G$
with equal projections $\hat {\mathcal{P}} \circ \hat {\gamma} _d =
\hat {\mathcal{P}} \circ \hat {\gamma}_{\Phi}$ such that for the
initial points of the curves $\hat {\gamma} _d(0) = \hat d _{q}$ and
$\hat {\gamma}_{\Phi}(0)= \hat{\Phi}(t,z') (=\hat{\phi}(t,z'))$ and
the terminating points of the curves $\hat {\gamma} _d$ and $\hat
{\gamma}_{\Phi}$ coincide.

Each curve will be the sum of two curves. To define the first part
of $\hat {\gamma}_{\Phi}$ we choose a number $\mathfrak{a} \in
[0,2]$ close to $2$ and let $\beta : [0,\mathfrak{a}] \rightarrow
\{t\} \times K_t$ be a curve that joins the point $(t,z')$ with the
point $(t,\mathfrak{a})$. Recall that for $t$ close to $1$ the set
$K_t$ has the form $r\overline \D \cup [0,2]$.
Define the first part of $\hat {\gamma}_{\Phi}$ by $\hat
{\gamma}_{\Phi}(\tau) = \hat{\phi}(\beta(\tau)), \;\tau \in
[0,\mathfrak{a}]$. Hence, as required $\hat {\gamma}_{\Phi}(0) =
\hat{\phi}(t,z')$. For the projected curve we have $\hat
{\mathcal{P}} \circ \hat {\gamma}_{\Phi}(\tau) =
{\phi}(\beta(\tau)), \tau \in [0,\mathfrak{a}]$.

Since for $t$ close to $1$ the point $2$ is the tip of the axon tree
$T_t^{ax}$ the inclusions  $\phi_1(2) = \phi (1,2) \in \kappa$,
$\hat {\phi}(1,2) = \hat i \circ {\phi}(1,2) \in \hat i (G) $ hold
(see the end of section 9). Hence we may assume that
${\phi}(\beta)(\mathfrak{a})$ is contained in the neighbourhood
$\mathcal{V}$ of $\kappa$ and $\hat{\phi}(\beta(\mathfrak{a})) =\hat
i \circ {\phi}(\beta(\mathfrak{a}) $ is in $\hat i (G)$.

To define the first part of $\hat {\gamma} _d$ we find a continuous
family of $G$-discs $d^{\tau}$ that are all close to $d_q$ and have
center $d^{\tau} (0) = {\phi}(\beta(\tau))$ so that $d^0 = d_q$. For
this we recall that $f_1$ is $2\varepsilon$-close to $\phi_1$ on
$K_1$ (since it is $\varepsilon$-close to $\psi _1$ on $K_1$ and
$\psi _1$ is $\varepsilon$-close to $\phi_1$ on $\nu_1$) and $\phi_1
= \Phi_1$ on $\overline \D$. Also, $d_p$ differs from $f_1$ by a
reparametrization. Further, if $(t,z')$ is close to $(1,z)$ then
$d_{q}$ is $\varepsilon$-close to $d_p$ on $\overline \D$. Moreover,
for $t$ close to $1$ $\max_{K_1} |\phi(t,z) -\phi(1,z)| <
\varepsilon$. Hence, in case $X^2= \C ^2$, there are points
$z_{\tau} \in \D$ depending continuously on $\tau \in
[0,\mathfrak{a]}$ and  a continuous family of translations
$d_q^{\tau}$ of $d_q$ such that the relation $d_q ^{\tau} (z_{\tau})
= \phi (\beta (\tau)),\; \tau \in [0,\mathfrak{a}],$ holds and
$d_q^0 = d_{q}$. For general $X^2$ instead of translations one can
use a continuous family of compactly defined holomorphic maps close
to the identity on $X^2$. Renormalize the discs $d_q^{\tau}$ so that
the centers become $\phi(\beta(\tau)))$ and let $\hat {\gamma}
_d(\tau)$ be the equivalence class represented by the renormalized
disc $d^{\tau}
\defined d_q^{\tau} \circ \varphi_{z_{\tau}}$.

For defining the second part of the curves we consider an arc
$\gamma:[\mathfrak{a},3] \rightarrow \mathcal{V} \subset G \subset
X^2$ which joins the point $\phi(\beta(\mathfrak{a})))$ with a point
$q_1$ in the image $d_p(\D)= f_1 \circ \varphi _1 (\textsf{r} \D)$
which is close to $\phi_1(2)=\phi(1,2) \in \kappa$. Define $\hat
{\gamma}_{\Phi}$ on $[\mathfrak{a},3]$ to coincide with the lift
$\hat i \circ \gamma $ of $\gamma$.

To define $\hat {\gamma} _d$ on $[\mathfrak{a},3]$ we consider again
a continuous family
of small perturbations of $d_q$ such that for each $\tau$ the
respective disc passes through ${\gamma}(\tau)$, for $\tau =
\mathfrak{a}$ the disc coincides with the disc $d_q^\mathfrak{a}$
defined before and for $\tau = 3$ the disc equals $d_p$.

 Reparametrize the discs so that the centers become
 ${\gamma}(\tau)$,
 and consider the equivalence classes represented by the
 reparametized discs. We obtain a curve $\hat {\gamma} _d|[\mathfrak{a},3]$
which is the second part of $\hat {\gamma} _d$.  Note that $\hat
{\gamma} _d(3)$ is represented by a reparametrization of $d_p$ for
which the center is the point $q_1$. With a suitable choice of $q_1
\in d_p(\D)$ we may assume that the conditions of Lemma 7 are
satisfied and, hence, $\hat {\gamma} _d(3)$ coincides with the class
represented by small discs in $G$ centered at $ {\gamma} _d(3)$.
Since the same is true for $\hat {\phi}(\gamma(3))$ the proof of
Lemma 22 is completed.\hfill $\square$

\medskip

Lemma 13 and, hence, the theorem are proved. \hfill $\square$

\section{Proof of the corollaries}

\medskip

\noindent {\it Proof of Corollary 1.} By Theorem 1 and Lemma 1 for
each point $p$ in the envelope of holomorphy $\tilde G$ there exists
an immersed analytic disc $\tilde d: \overline \D \rightarrow \tilde
G$ such that $ \tilde d (0) = p$ and $\tilde d (\partial \D) \subset
i(G) \subset \tilde G$.

We may assume that $\tilde d$ extends to an analytic immersion of
$(1 + \varepsilon ) \D$ for some positive number $\varepsilon$. The
mapping can be uniformly approximated on $(1 + 1/2 \varepsilon ) \D$
by an immersion of the disc with only double self-intersection
points and transversal self-intersection. This is a standard
Morse-Sard type argument. The obtained disc can be considered as a
nodal curve with boundary, i.e. as a singular Riemann surface with
boundary all singularities of which are nodal singularities. By
results of Ivashkovich and Shevchishin on the moduli space of
Riemann surfaces (see  \cite{IvSh}, theorem 3.4 and lemma 3.8) the
nodal curve is uniformly close to a smooth Riemann surface embedded
into $\tilde G$. \hfill $\square$

\medskip

\noindent {\it Remark.} Theorem 1 implies the result of \cite{Ke}
that the natural homomorphism $\varphi: \pi_1(G) \rightarrow \pi_1
(\tilde G)$ induced by inclusion is surjective. Indeed, by the
following argument any closed curve $\gamma$ in $\tilde G$ is
homotopic in $\tilde G$ to a curve in $i(G)$. Take an excrescence
$\gamma^*$ of $\gamma$, $\gamma^*:
\partial \D \rightarrow \tilde G$, which lifts to a
mapping $\overset{\circ}{\gamma}^*:\partial \D \rightarrow \GN$.
Note that $\gamma ^*$ is homotopic to $\gamma$. Let
$\overset{\circ}{\gamma}^*(\zeta)(z),\,\zeta \in \partial \D, \, z
\in \overline \D$, be the evaluation mapping. The curve $\zeta
\rightarrow \overset{\circ}{\gamma}^*(\zeta)(1)$ is homotopic in
$\tilde G$ to $\gamma^*$and contained in $i(G)$.

\medskip

\noindent {\it Proof of Corollaries 2 and 3.} Consider the following
slightly more general situation which includes the case of each of
the two corollaries. Let $S$ be an orientable compact connected
surface with or without boundary. Let $f: S \rightarrow \tilde G$ be
a continuous mapping. If the boundary $\partial S$ is not empty we
will assume that $f(\partial S) \subset i(G)$. In case of a closed
surface $S$ we think about $f: S \rightarrow \tilde G$ representing
a homology class in $H_2(\tilde G)$. The case when $S = b^2$ is a
disc corresponds to the homotopy of the loop representing an element
in the kernel of the homomorphism $\varphi$ in Corollary 3. We may
always deform the surface so that $f(S)$ contains the point $p$. Say
$p = f (\zeta^*)$.

Since $\tilde G = \hat G$ and locally each mapping into $\hat G$
lifts to a mapping into $\GN$ we may consider a simplicial
decomposition of $S$ which is fine enough so that the following
properties hold:
\begin{itemize}

\item [(1)] On each $2$-simplex $\sigma _j$ of the decomposition there is a continuous lift
$\overset{\circ}{f} _j : \sigma _j \rightarrow \GN$ of $f_j \defined f\mid \sigma _j$ to $\GN$.

\item [(2)] Consider an arbitrary edge $e_k$ of the simplicial complex. Let $\sigma _i$ and
$\sigma _j$ be the adjacent $2$-simplices. For $\zeta \in  e_k$ we
denote by $(\overset{\circ}{f} _i (\zeta), \overset{\circ}{f} _j
(\zeta))$ the equivalent discs corresponding to the two simplices by
property (1). We require that there is a family of dendrites
$\mathbf{T}_{i,j}(\zeta)$ with punctured halo associated to the
family $(\overset{\circ}{f} _i (\zeta), \overset{\circ}{f} _j
(\zeta))$ of pairs of discs by lemma 5, depending continuously on
the point $\zeta$ and such that the underlying trees of the
dendrites are homeomorphic.

\item [(3)]  The point $\zeta ^*$ is a vertex of the simplicial decomposition.
\end{itemize}

We will use now properties (1) and (2) to obtain a homotopy of the
mapping $f$ to a new mapping $f^1:S \rightarrow \tilde G$ with the
following property. There is a tree $\mathfrak{T} \subset S$ such
that $f^1|S\setminus \mathfrak{T}$ lifts to $\GN$. Moreover, the
lifted mapping extends continuously to the pellicle of
$\mathfrak{T}$ (the latter defined in the above sense assuming a
simply connected neighbourhood of $\mathfrak{T}$ in $S$ being
extended to a sphere).

To find a suitable tree $\mathfrak{T}$ we will color each
$1$-simplex either white or black in such a way that the union of
black simplices constitutes a (connected) tree wich contains each of
the vertices of the triangulation. The coloring is done as follows.
Since $S$ is connected the union of all $1$-simplices (edges) of the
triangulation is connected. If the boundary $\partial S$ is not
empty then all edges contained in it are colored white. Since for
each $2$-simplex no more than one adjacent edge is contained in
$\partial S$ the union of uncolored edges is connected and contains
all vertices of the triangulation. If the union of uncolored edges
contains a closed loop we give white color to one of the edges
constituting the loop. The union of uncolored edges still
constitutes a connected set and contains all vertices. After
finitely many steps the union of uncolored edges is a connected set
without closed loops containing all vertices. Color the so far
uncolored edges black. We obtained a coloring with the desired
properties. Denote the tree constituted by the union of all black
edges by $\mathfrak{T}'$.

Consider the barycentric subdivision of the simplicial complex. Associate to each edge $e_k$
of the original complex the union $\tilde {\sigma} _k$ of those four $2$-simplices of the subdivision
that contain a "half" of $e_k$. The $\tilde {\sigma} _k$ have pairwise disjoint interior and cover $S$.

Let $e_k$ be a white edge. We describe now a homotopy of the
restriction $f|\tilde {\sigma} _k$ to a mapping $f^1 |\tilde
{\sigma} _k$ which fixes the values at the boundary of $\tilde
{\sigma} _k$. Let $\sigma_i$ and $\sigma_j$ be the $2$-simplices of
the original simplicial complex that are adjacent to $e_k$ and let
$\mathbf{T}_{i,j}(\zeta),\; \zeta \in e_k,$ be the dendrites
associated to $e_k$ according to property 2. Let further
$m_{i,j}(t,\zeta),\; t \in [0,1], \zeta \in e_k,$ be a
parametrization of the pellicles of the trees $T_{i,j}(\zeta)$
depending continuously on $\zeta$.

Cut $\tilde {\sigma} _k$ along $e_k$ and glue back the union
$\bigcup _{t \in [0,1], \zeta \in e_k} m_{i,j}(t,\zeta)$ with the
natural gluing homeomorphism on the two sides of $e_k$ (the point
$m_{i,j}(0,\zeta)$ (respectively, the point $m_{i,j}(1,\zeta)$) is
identified with the point on the side of $\sigma_i$ (respectively,
$\sigma _j$) over $\zeta \in e_k$). We obtain a (singular) closed
square $\sigma_k^*$. The mapping $f|\tilde \sigma _k$ extends to a
continuous mapping on $\tilde \sigma _k \, \cup  \, \bigcup _{t \in
[0,1], \zeta \in e_k} T_{i,j}(\zeta) $, moreover, it extends to a
continuous mapping $f_k^*$ on $\sigma_k^*$ which lifts to $\GN$.
Moreover, reparametrize $\sigma_k^*$ in the following way. Consider
disjoint trees $T_0= T_0^k$ and $T_1=T_1^k$, both homeomorphic to
the underlying tree of the dendrites $\mathbf{T}_{i,j}(\zeta)$,
having their root respectively at the endpoints $\zeta_0$ and
$\zeta_1$ of the edge $e_k$, being contained in $\tilde {\sigma} _k$
and each meeting the boundary of $\tilde {\sigma} _k$ exactly at its
root.

Let $\varphi$ be a homeomorphism of the set $\tilde {\sigma} _k
\setminus (T_0 \cup T_1)$ onto $\sigma_k^* \setminus
(T_{i,j}(\zeta_0)\cup T_{i,j}(\zeta_1))$
which is the identity on the boundary $\partial {\tilde {\sigma} _k
}$. Require, moreover, that $\varphi$ extends continuously to the
pellicle of $T_0$ ($T_1$, respectively) and maps it homeomorphically
onto the pellicle of $T_{i,j}(\zeta_0)$ ( $T_{i,j}(\zeta_1)$,
respectively).
Put $f^1|\tilde {\sigma} _k \setminus (T_0 \cup T_1) \defined f^*_k
\circ \varphi|\tilde {\sigma} _k \setminus (T_0 \cup T_1)$. This
mapping extends to a continuous mapping on $\tilde {\sigma} _k$,
also denoted by $f^1$. Since each rooted tree is contractible to its
root and the construction an be made for subtrees and so that it
depends countinuously on the choice of subtrees, the mappings
$f|\tilde {\sigma} _k$ and $f^1|\tilde {\sigma} _k$ are homotopic.

As required, the restriction $f^1|\tilde {\sigma} _k \setminus (T_0
\cup T_1)$ lifts to $\GN$. The lift extends continuously to the
punctured pellicle of $T_0$ and $T_1$. Attach the trees $T_0=T_0^k$
and $T_1=T_1^k$ to $\mathfrak{T}'$.

Proceed in the same way with each of the white edges. We obtain a
new tree $\mathfrak{T} \subset S$ and a homotopy of $f$ on the whole
of $S$ to a mapping $f^1$. The restriction $f^1|S \setminus
\mathfrak{T}$ of the final mapping $f^1$ admits a lift
$\overset{\circ}{f^1}$ to $\GN$ which extends continuously to the
pellicle of the tree $\mathfrak{T}$ .

Approximate the mapping $f^1: \mathfrak{T} \rightarrow \tilde G$ of
the tree by a true analytic disc $f^2 : \overline \Delta \rightarrow
\tilde G$. Here $\Delta$ denotes a small simply connected
neighbourhood of $\mathfrak{T}$ on $S$ which we endow with complex
structure. Extend the mapping to a continuous mapping $f^2:S
\rightarrow \tilde G$ which equals $f^1$ outside a small
neighbourhood of the closure $\overline \Delta$. If $f^2$ is close
to $f^1$ on $S$ then the two mappings are homotopic and $f^2|S
\setminus \Delta$ lifts to a mapping $\overset{\circ}{f^2}|S
\setminus \Delta \rightarrow \GN$.


Note that the (images of the) circle fibers $\bigcup _{\zeta \in S
\setminus \Delta}\overset{\circ}{f^2}(\zeta)(\partial \D) $ are
contained in $G$. Moreover, there is an open subset $U_0$ of $S$
such that for $\zeta \in U_0$ the (image of the) whole disc fiber
$\overset{\circ}{f^2}(\zeta)(\overline \D)$ is contained in $G$. For
each $k$ the points in $\tilde {\sigma} _k$ which are close to a
leaf of $T_0^k$ or $T_1^k$ belong to $U_0$. If $\Delta$ is a
sufficiently small neighbourhood of $\mathfrak{T}$ its boundary
$\partial \Delta$ intersects $U_0$ since $\mathfrak{T}$ contains the
trees $T_0^k$ and $T_1^k$ for each white edge $e_k$ . Hence, the
mapping $f^2|\Delta$ is a disc neuron and the restriction
$\overset{\circ}{f^2}|\partial \Delta$ is its halo.

We may consider the lift $\overset{\circ}{f^2}$ of $f^2$ up to
changing it on the set $U_0$. More precisely, consider lifts
$\overset{\circ}{F^2}$ of  $f^2$ on $S \setminus \Delta$ such that
$\overset{\circ}{F^2} = \overset{\circ}{f^2}$ outside $U_0$ and for
all $\zeta$ in $U_0$ the property $\overset{\circ}{F^2}(\zeta)
(\overline \D) \subset G$ holds. We call such lifts
$\overset{\circ}{F^2}$ admissible changes of $\overset{\circ}{f^2}$.

Lemma 17 and 18 apply to $f^2|\overline \Delta$ and its halo (and
the Stein manifold $\tilde G$). Lemma 18 provides an approximation
(take, for instance, the mapping $\mathfrak{H}(\zeta,\cdot)$ in the
notation of lemma 17) of $\overset{\circ}{f^2}(\zeta), \; \zeta \in
\partial \Delta \setminus U_0$, and (the proof of) Lemma 17 states that after an
admissible change on $U_0$ we obtain a new lift
$\overset{\circ}{f^3}$ on $\partial \Delta$ of the same mapping
$f^2|\partial \Delta$ such that the Riemann-Hilbert boundary value
problem is solvable: There exists a section $\partial \Delta
\backepsilon \zeta \rightarrow \overset{\circ}{f^3}(\zeta)(g(\zeta))
\in \bigcup _{\zeta \in
\partial \Delta} \overset{\circ}{f^3}(\zeta)(\partial \D)$ which coincides
with the boundary values of an  analytic disc in $\tilde G$. This
disc is a $\G$-disc. Denote it by $F(\zeta),\; \zeta \in \Delta$.
The mappings $\overline \Delta \backepsilon \zeta \rightarrow
\overset{\circ}{f^3}(\zeta)(r g(\zeta)) \in \bigcup _{\zeta \in
\partial \Delta} \overset{\circ}{f^3}(\zeta)(\overline \D)$, $r \in [0,1],$
provide a homotopy of mappings into $\tilde G$ joining
$f^2|\overline \Delta$ with $F|\overline \Delta$.

Extend $\overset{\circ}{f^3}$ to the whole set $S \setminus \Delta$
as a continuous lift of $f^2$ such that the extended mapping equals
$\overset{\circ}{f^2}$ outside a neighbourhood of $\partial \Delta$.
Denote the mapping again by $\overset{\circ}{f^3}$. After admissible
changes of the mapping $\overset{\circ}{f^3}$ on $U_0$ it remains to
find a section $S \setminus \Delta \backepsilon \zeta \rightarrow
\overset{\circ}{f^3}(g(\zeta)) \in   \bigcup _{\zeta \in S^\setminus
\Delta} \overset{\circ}{f^3}(\zeta)(\partial \D)$ extending the
section found before on $\partial \Delta$. Since $U_0$ intersects
$\tilde \sigma _k$ for each white edge $e_k$ this is always
possible. The new mapping $F$ is now defined on $ S \setminus
\Delta$ by this section: $F(\zeta) = \overset{\circ}{f^3}(g(\zeta)),
\zeta \in S \setminus \Delta$, and the homotopy is given by
$\overset{\circ}{f^3}(rg(\zeta))$, $r \in [0,1]$.

Note that the disc $\Delta$ contains the point $\zeta ^*$. The
construction can be made in such a way that $F$ is close to $f$ in a
neighbourhood of $\zeta^*$. A small perturbation of the surface $F:S
\rightarrow \tilde G$ will pass through $p$.

Corollaries 2 and 3 are proved. \hfill $\square$

\medskip

\noindent {\it Proof of Corollaries 4 and 5.} The proof uses
Corollaries 2 and 3. Let $\Omega $ be a strictly pseudoconvex domain
in a Stein surface $X^2$, $\Omega = \{\rho < 0 \}$ for a strictly
plurisubharmonic function $\rho$ defined in a neighbourhood of the
closure $ \overline \Omega $ of $\Omega$. Let $G = \{0< \rho <
\varepsilon\}$ for a small positive number $\varepsilon$ so that
$\rho$ does not have critical points in $G$. Then $\tilde G = \Omega
_{\varepsilon} \defined \{\rho < \varepsilon\}$. Denote by
$\mathfrak{I}$ a retraction of $\Omega _{\varepsilon}$  onto
$\overline \Omega$.

Let $f:S \rightarrow  \overline \Omega$ be a continuous mapping of
an orientable connected compact surface. If the boundary $\partial
S$ is not empty we require that $f(\partial S) \subset \overline
\Omega$. Consider $f$ as a mapping into $\tilde G = \Omega
_{\varepsilon}$. If $\partial S$ is not empty we perturb the mapping
slightly so that $f(\partial S) \subset G$. By the proof of the
Corollaries 2 and 3 there is a homotopy of $f$ (in $\Omega
_{\varepsilon}$) to a mapping $F_1: S \rightarrow \Omega
_{\varepsilon} $ and a disc $\Delta \subset S$ such that
$F_1|\overline \Delta $ is an analytic disc and $F_1(S \setminus
\Delta)$ is contained in $G$. We may assume that $\Delta $ is not
empty. After a small perturbation of $F_1$ the analytic disc
$F_1(\Delta)$ has no self-intersection points on $\partial \Omega$
and intersects $\partial \Omega$ transversally. Let $\Delta_1$ be
the subset of $\Delta $ that is mapped into $\Omega$: $\Delta_1
\defined \{\zeta \in \Delta: F_1(\zeta) \in \Omega\}$. By the
maximum principle for the function $\rho$ the set $\Delta_1$ is the
union of simply connected planar domains. If $\Delta_1$ is connected
then $\mathfrak{I } \circ F_1$ is the desired mapping.

If $\Delta_1$ is not connected, let $\delta_1,...,\delta _N$ be its
connected components. There are pairwise disjoint arcs
$\gamma_1,...\gamma_{N-1}$ on $\Delta$ without self-intersections
such that $\gamma_i$ joins a point in $\partial \delta_i$ with a
point in $\partial \delta_{i+1}$ and does not meet the union of the
$\overline \delta _i$ otherwise. After a further (small) homotopy of
the mapping $F_1|\Delta \setminus \bigcup \overline \delta _i$
inside $\Omega _{\varepsilon} \setminus G$ which fixes the mapping
on the union of the boundaries $\bigcup \partial \delta _i$ we may
assume that the arcs $F_1 ( \gamma _i)$ are contained in $\partial
\Omega$, are pairwise disjoint without self-intersection points and
meet the union of the $F_1(\partial \delta _i)$ exactly at the
endpoints of the arcs. After approximating the arcs and the mapping
$F_1$ we may assume that the arcs are Legendrian arcs in $\partial
\Omega$. (It is well-known in contact geometry that arbitrary curves
in contact manifolds may be $C^0$ approximated by Legendrian curves,
for an elementary proof see, e.g. \cite{Gei}). We arrived at the
union of analytic discs with Legendrian arcs $F_1: \cup \overline
\delta _i \bigcup \cup \gamma_i \rightarrow \overline \Omega$.

\begin{lemma}\label{Lemma 23}
Let $E \subset \C$ be a connected compact simply connected set
consisting of the union of pairwise disjoint closed discs and
pairwise disjoint arcs meeting the discs at most at their endpoints.
Let $\Omega$ be a relatively compact strictly pseudoconvex domain in
a Stein surface $X^2$ and let $f:E \rightarrow \overline \Omega$ be
a continuous mapping for which the restriction to each closed disc
in $E$ is an analytic disc with boundary in $\partial \Omega$ and
each of the arcs is a Legendrian arc in $\partial \Omega$.

Then the mapping can be approximated by a true analytic disc $F:
\Delta \rightarrow \overline \Omega$ with boundary in $\partial
\Omega$. Here $\Delta$ is a simply connected planar domain with $E
\subset \overline \Delta$ and $\Delta$ is contained in a small
neighbourhood of $E$. Moreover, if $z$ is the tip of an arc in $E$
(not contained in the boundary of any of the closed discs in $E$)
then $\Delta $ can be chosen so that $z \in \partial \Delta$ and
$F(z) = f(z)$.

\end{lemma}

The Lemma seems to be folklore but we have no direct reference.
After the proof of the Corollaries we will sketch the proof.

The lemma allows to find a homotopy of $F_1$ to a mapping $F_2:S
\rightarrow \Omega _{\varepsilon}$ such that for a simply connected
domain $\Delta _2 \subset \Delta$ the restriction $F_2|\Delta _2$ is
an analytic disc with boundary in $\partial \Omega $ and the set
$F(S \setminus \Delta_2)$ is contained in $\Omega _{\varepsilon}
\setminus \Omega$ . Composing $F_2$ with the retraction
$\mathfrak{I}$ finishes the proof. \hfill $\square$

\medskip

It remains to sketch the proof of Lemma 23. Notice that the Lemma
was used in the example in section 1 and also implies the following
fact. The boundary of the disc of Corollary 5 which represents an
element of the fundamental group of $\partial \Omega$ can be chosen
to pass through a given base point $p \in \partial \Omega$.

\medskip

\noindent {\it Sketch of the proof of Lemma 23.} Notice that after
approximating we may assume that for each analytic disc $f(\overline
\delta_j)$ contained in $f(E)$ the mapping $f$ extends to an
analytic immersion of a larger disc $\delta '\supset \overline
\delta$ to a neighbourhood of $\overline \Omega$ in $X^2$ (keeping
the condition $f(\partial \delta) \subset \partial \Omega$).
Consider a small connected neighbourhood $V$ of $f(E \setminus Int
\,E)$. (The set $f(E \setminus Int \,E)$ is the union of the
boundaries of the analytic discs contained in $f(E)$ and the
Legendrian arcs. Notice that $f(E \setminus Int \,E) \subset
\partial\Omega$.) With each of the analytic discs $f_i: \delta _i '
\rightarrow X^2$ we associate (as in section 3) a Riemann domain
$\mathcal{R}_i$ over $X^2$ (biholomorphic to $\delta _i ' \times
\varepsilon _i \D$ for some $\varepsilon _i >0$)  to which the disc
lifts as an embedded disc. Consider the disjoint union of the
Riemann domains $\mathcal{R}_i$ and glue each $\mathcal{R}_i$ in a
natural way to $V$ along a neighbourhood of the respective circle
$f(\partial \delta _i)$. Shrinking the Riemann domains and the
domain $V$ suitably we obtain a (strictly) pseudoconvex Riemann
domain $\mathcal{R}$ over $X^2$ which is diffeomorphic to a ball
(see \cite{Sh} where the method of gluing tubular neighbourhoods of
arcs to strictly pseudoconvex domains to obtain strictly
pseudoconvex domains appeared first).

Denote by $M$ the lift of $V \cap \partial \Omega$ to $\mathcal{R}$.
$M$ is a relatively closed hypersurface in $\mathcal{R}$ which is
strictly pseudoconvex from one side. The lifts to $\mathcal{R}$ of
the analytic discs contained in $f(E)$ extend to embedded relatively
closed analytic discs in $\mathcal{R}$, denoted by $F_i(\overline
\D)$. Denote the lifts of the arcs in $f(E)$ by $\gamma _i $. The
$\gamma_i$ are Legendrian arcs in $M$. To each $\gamma_i$ we
associate a chain of small analytic discs $g_k: \overline\D
\rightarrow \mathcal{R},\; k=1,...,N,$ so that $g_k(\partial \D)
\subset M$, $g_1(-1)$ is an endpoint of $\gamma _i$,
$g_k(1)=g_{k+1}(-1), \; k=1, ...,N-1,$ and $g_N(1)$ is the other
endpoint of $\gamma_i$. The discs may be taken to be intersections
with the pseudoconvex side of $M$ of complex lines in suitable
coordinates. By further shrinking the Riemann domain we assume that
these discs extend to relatively closed embedded analytic discs in
$\mathcal{R}$ which meet transversally and do not meet the
$F_i(\overline \D)$ except at $g_1(-1)$ and possibly $g_N(1)$. We
may assume that the latter intersections are also transversal. We
obtained a finite collection of relatively closed discs in
$\mathcal{R}$. Since $\mathcal{R}$ is diffeomorphic to a ball, each
disc is the zero set $\{\mathcal{F}_i =0\}$ of an analytic function
$\mathcal{F}_i$ on $\mathcal{R}$. For a generic choice of a small
number $\eta$ the set $X_{\eta} \defined \{\prod \mathcal{F}_i =
\eta\} \cap \overline \Omega$ is an analytic disc (see, e.g.
\cite{Or}, Lemma 3.7). If $\gamma _i$ is an arc with the second
endpoint not contained in the boundary of any of the analytic discs
$F_i(\partial \D)$ we may adjust the choice of the last small disc
$g_N$ and the number $\eta$ so that the boundary of the disc
$X_{\eta}$ passes through the endpoint of $\gamma _i$.

The lemma is proved. \hfill $\square$


\begin{thebibliography}{9}



\bibitem{DoGr}
F.~ Docquier, H.~ Grauert, \emph{Levisches Problem und Rungescher
Satz f\"ur Teilgebiete Steinscher Mannigfaltigkeiten}, Math.Ann.
\textbf{140} (1960), 94-123.


\bibitem{FoZa}
J.-E.~Fornaess, W.R. ~Zame, \emph{Riemann domains and envelopes of
holomorphy}, Duke Math.J. \textbf{50} (1983), 273-283.

\bibitem{EGr}
Y.~ Eliashberg, M.~ Gromov, \emph{Embeddings of Stein manifolds of
dimension $n$ into the affine space of dimension $3n/2+1$},  Ann. of
Math. textbf{136}  (1992), 123-135.

\bibitem{Fo}
F.~ Fortstneric, \emph{Holomorphic flexibility properties of complex
manifolds}, Amer. J. Math. \textbf{128} (2006), 239-270.



\bibitem{FoGl}
F.~ Forstneric, J.~ Globevnik, \emph{Discs in pseudoconvex domains},
Comment.Math.Helvetici \textbf{67} (1992), 129-145.

\bibitem{FrGr}
K.~ Fritsche, H.~ Grauert, \emph{From holomorphic functions to
complex manifolds}, Graduate Texts in Mathematics,\textbf{213},
Springer, New York, Berlin, Heidelberg, 2002.

\bibitem{Gei}
H.~ Geiges, \emph{Contact Geometry}, Handbook of differential
geometry. Vol. \textbf{II}, 315-382, Elsevier/North-Holland,
Amsterdam, 2006.


\bibitem{Grau}
H.~ Grauert, \emph{Charakterisierung der holomorph vollst\"andigen
komplexen R\"aume}, Math.Ann. \textbf{129} (1955),
  233-259.

\bibitem{IvSh}
S.M.~ Ivashkovich, V.V.~ Shevchishin, \emph{Deformations of
non-compact complex curves and envelopes of meromorphy of spheres},
Sbornik:Mathematics \textbf{189:9} (1998), 1295-1333.



\bibitem{H}
L.~ H\"ormander, \emph{An Introduction to Complec Analysis in
Several Variables}, Third Edition (revised), North Holland Math.
Library, Amsterdam, New York, Oxford, Tokio, (1990).



\bibitem{Ju}
M.~ Jurchescu, \emph{On a theorem of Stoilow}, Math. Ann.,
\textbf{138} (1959), 332-334.

\bibitem{KaZa}
Sh.~ Kaliman, M.~ Zaidenberg \emph{A transversality theorem for
holomorphic mappings and stability of Eisenman-Kobayashi measures},
Trans.Amer.Math.Soc. \textbf{348}, no.~2,(1996),
  1-12.




\bibitem{Ke}
H.~ Kerner, \emph{\"Uberlagerungen und Holomorphieh\"ullen},
Math.Ann. \textbf{144} (1961),
  126-134.

\bibitem{Ma}
B.~ Malgrange, \emph{Lectures on the theory of functions of several
complex variables}, Tata Institute of fundamenbtal research, Bombay
(1958) (Reissued 1965).

\bibitem{Oka}
K.~ Oka, \emph{Sur les fonctions analytiques de plusieurs variables.
IX. - Domains finis sans point critique interieur}, Japanese J.
Math. \textbf{23} (1953), 87-155.

\bibitem{Or}
S.~ Orevkov, \emph{An algebraic curve in the unit ball in
$\mathbf{C}^ 2$ that passes through the origin and all of whose
boundary components are arbitrarily short} (Russian), Tr. Mat. Inst.
Steklova \textbf{253} (2006), 135-157.

\bibitem{OSt}
B.~ Ozbagci, A.~ Stipsicz, \emph{Surgery on contact $3$-manifolds
and Stein surfaces}, Bolyai Society Mathematical Society and
Springer (2004).


\bibitem{Rossi}
H.~ Rossi, \emph{On Envelopes of Holomorphy}, Comm. Pur Appl. Math.
\textbf{XVI} (1963), 9 -17.


\bibitem{Ru}
W.~ Rudin, \emph{Real and complex analysis}, Third Edition, McGraw-Hill Book Company, New York etc (1987).

\bibitem{Roy}
H.L.~ Royden, \emph{One-dimensional cohomology in domains of
holomorphy}, Ann. Math. \textbf{78} (1963), 197-200.

\bibitem{Sh}
N.~ Shcherbina, \emph{Decomposition of a common boundary of two
domains of holomorphy into analytic curves} (Russian),  Izv. Akad.
Nauk SSSR Ser. Mat.,\textbf{46} (1982),  1106-1123, 1136.


\bibitem{Sto}
E.L.~ Stout, \emph{A domain whose envelope of holomorphy is not a
domain}, Ann. Polon. Math. \textbf{89} (2006), 197-201.



\end{thebibliography}
\end{document}